\documentclass[11pt,thmsa]{article}
\usepackage{amssymb}
\usepackage{amsfonts}
\usepackage{amsmath}
\usepackage{cite}
\usepackage[top=3.8cm,bottom=3.8cm,textwidth=16cm,centering,a4paper]{geometry}

\numberwithin{equation}{section}
\newtheorem{theorem}{Theorem}[section]

\newtheorem{corollary}[theorem]{Corollary}

\newtheorem{lemma}[theorem]{Lemma}

\newtheorem{proposition}[theorem]{Proposition}
\newtheorem{remark}[theorem]{Remark}

\newenvironment{proof}[1][Proof]{\noindent\textbf{#1.} }{\ \rule{0.5em}{0.5em}}

\begin{document}

\title{On the Kirchhoff type equations in $\mathbb{R}^{N}$}
\author{Juntao Sun$^{a}$\thanks{%
E-mail address: jtsun@sdut.edu.cn (J. Sun)}, Tsung-fang Wu$^{b}$\thanks{%
E-mail address: tfwu@nuk.edu.tw (T.-F. Wu)} \\
{\footnotesize $^a$\emph{School of Mathematics and Statistics, Shandong
University of Technology, Zibo, 255049, P.R. China }}\\
{\footnotesize $^b$\emph{Department of Applied Mathematics, National
University of Kaohsiung, Kaohsiung 811, Taiwan }}}
\date{}
\maketitle

\begin{abstract}
Consider a nonlinear Kirchhoff type equation as follows
\begin{equation*}
\left\{
\begin{array}{ll}
-\left( a\int_{\mathbb{R}^{N}}|\nabla u|^{2}dx+b\right) \Delta
u+u=f(x)\left\vert u\right\vert ^{p-2}u & \text{ in }\mathbb{R}^{N}, \\
u\in H^{1}(\mathbb{R}^{N}), &
\end{array}%
\right.
\end{equation*}%
where $N\geq 1,a,b>0,2<p<\min \left\{ 4,2^{\ast }\right\}$($2^{\ast }=\infty
$ for $N=1,2$ and $2^{\ast }=2N/(N-2)$ for $N\geq 3)$ and the function $f\in
C(\mathbb{R}^{N})\cap L^{\infty }(\mathbb{R}^{N})$. Distinguishing from the
existing results in the literature, we are more interested in the geometric
properties of the energy functional related to the above problem.
Furthermore, the nonexistence, existence, unique and multiplicity of
positive solutions are proved dependent on the parameter $a$ and the
dimension $N.$ In particular, we conclude that a unique positive solution
exists for $1\leq N\leq4$ while at least two positive solutions are
permitted for $N\geq5$.
\end{abstract}

\section{Introduction}

We are concerned with the following nonlinear Kirchhoff type equations:%
\begin{equation}
\left\{
\begin{array}{ll}
-\left( a\int_{\mathbb{R}^{N}}|\nabla u|^{2}dx+b\right) \Delta u+V(x)u=h(x,u)
& \text{ in }\mathbb{R}^{N}, \\
u\in H^{1}(\mathbb{R}^{N}), &
\end{array}%
\right.  \label{1-0}
\end{equation}%
where $N\geq 1,a,b>0,$ $V\in C(\mathbb{R}^{N},\mathbb{R})$ and $h\in C(%
\mathbb{R}\times \mathbb{R}^{N},\mathbb{R}).$

Kirchhoff type equations, of the form similar to Eq. $(\ref{1-0}),$ are
analogous to the stationary case of equations that arise in the study of
string or membrane vibrations, namely,
\begin{equation}
u_{tt}-\left( a\int_{\Omega }|\nabla u|^{2}dx+b\right) \Delta u=h(x,u)\text{
in }\Omega ,  \label{1-1}
\end{equation}%
where $\Omega $ is a bounded domain in $\mathbb{R}^{N}.$ As an extension of
the classical D'Alembert's wave equation, Eq. $(\ref{1-1})$ was first
presented by Kirchhoff \cite{K} in 1883 to describe the transversal
oscillations of a stretched string, particularly, taking into account the
subsequent change in string length caused by oscillations, where $u$ denotes
the displacement, $h$ is the external force and $b$ is the initial tension
while $a$ is related to the intrinsic properties of the string, such as
Young's modulus. Equations of this type are often referred to as being
nonlocal because of the presence of the integral.

After the pioneering work of Pohozaev \cite{P} and Lions \cite{L}, the
solvability of the Kirchhoff type equation $(\ref{1-1})$ has been
well-studied in general dimension by various authors, see for examples,
D'Ancona-Shibata \cite{DS}, D'Ancona-Spagnolo \cite{DS1} and Nishihara \cite%
{N}. More recently, the corresponding elliptic version like Eq. $(\ref{1-0})
$ has begun to receive much attention via variational methods. We refer the
reader to \cite{Az1,Az2,CKW,DPS,G,HL,HZ,LLS,LY,LLS1,N1,SZ,SW,SW1,TC,Y} and
the references therein.

Most of researchers have of late years focused on the existence of positive
solutions, ground states, radial solutions and semiclassical states for Eq. $%
(\ref{1-0})$ in lower dimensions, i.e., $N=1,2,3.$ The typical way to deal
with such problem is to apply the mountain-pass theorem or the Nehari
manifold method. Owing to the fourth power of the nonlocal term, one usually
assumes that the nonlinearity $h(x,u)$ is either $4$-superlinear at infinity
on $u$ in the sense that%
\begin{equation*}
\lim_{|u|\rightarrow \infty }\frac{\int_{0}^{u}h(x,s)ds}{u^{4}}=\infty \text{
uniformly in }x\in \mathbb{R}^{N},
\end{equation*}%
or satisfies the following (AR)-condition:%
\begin{equation}
\exists \mu >4\text{ such that }0<\mu \int_{0}^{u}h(x,s)ds\leq h(x,u)u\text{
for }u\neq 0.  \label{1-10}
\end{equation}%
For example, $h(x,u)=f(x)\left\vert u\right\vert ^{p-2}u$ with $4<p<2^{\ast
}(2^{\ast }=\infty $ for $N=1,2$ and $2^{\ast }=2N/(N-2)$ for $N\geq 3).$ By
so doing, one can easily verify the mountain-pass geometry and the
boundedness of (PS) sequences for the energy functional. However, there have
a large number of functions $h(x,u)$ not satisfying the above assumptions,
such as $h(x,u)=f(x)\left\vert u\right\vert ^{p-2}u(2<p<\min \left\{
4,2^{\ast }\right\}).$ For that reason, some other approaches need to be
introduced in this case.

In studying the radial solutions for a class of autonomous Schr\"{o}%
dinger-Poisson systems in $\mathbb{R}^{3},$ Ruiz \cite{R1} established a
manifold as follows%
\begin{equation*}
\overline{\mathbf{M}}=\{u\in H_{rad}^{1}(\mathbb{R}^{3})\backslash
\{0\}:2\left\langle I^{\prime }(u),u\right\rangle -P(u)=0\},
\end{equation*}%
where $I$ is the energy functional and $P(u)$ is the Pohozaev identity
corresponding to the system. It is usually called the Nehari-Pohozaev
manifold which is different from the Nehari manifold. By restricting the
energy functional to such manifold, the boundedness of (PS) sequences can be
solved effectively when the nonlinearity does not satisfy the (AR)-condition
above.

Inspired by Ruiz \cite{R1}, Li-Ye \cite{LY} applied the Nehari-Pohozaev
manifold (slightly different from $\overline{\mathbf{M}}$) to Kirchhoff type
equations in $\mathbb{R}^{3}$. By using the constraint minimization method,
together with the monotonicity trick by Jeanjean \cite{J}, they found one
ground state solution with positive energy of Eq. $(\ref{1-0})$ when $%
V(x)\leq \liminf_{|y|\rightarrow \infty }V(y)=V_{\infty }<\infty $ and $%
h(x,u)=\left\vert u\right\vert ^{p-2}u(3<p<2^{\ast })$. Later, using the
similar approach to that in \cite{LY}, Guo \cite{G} and Tang-Chen \cite{TC}
also obtained the existence of ground state solutions for Eq. $(\ref{1-0})$
with a general nonlinearity $h(x,u)\equiv h(u)$ in $\mathbb{R}^{3}$,
respectively. In addition, Ye \cite{Y} proved the existence of high energy
solutions for Eq. $(\ref{1-0})$ with $h(x,u)\equiv h(u)$ in $\mathbb{R}^{3}$
via the Nehari-Pohozaev manifold and the linking theorem. It is worthy
noting that the nonlinearity $h(u)$ given by \cite{G,TC,Y} can cover the
power functions $\left\vert u\right\vert ^{p-2}u(2<p<4).$

Azzollini \cite{Az1,Az2} investigated a class of autonomous Kirchhoff type
equations in higher dimensions $N\geq 3,$ i.e., Eq. $(\ref{1-0})$ with $%
V(x)\equiv 0$ and $h(x,u)\equiv h(u)$ satisfying the Berestycki--Lions type
conditions (see \cite{BL}). With the aid of the radial ground state solution
$\overline{u}$ to the semilinear elliptic equation $-\Delta u=h(u)$ in $%
\mathbb{R}^{N}(N\geq 3)$, the following results were obtained:\newline
$(i)$ $N=3:$ one radial ground state solution exists for all $a,b>0;$\newline
$(ii)$ $N=4:$ one radial ground state solution exists if and only if $%
a<\left( \int_{\mathbb{R}^{N}}|\nabla \overline{u}|^{2}dx\right) ^{-1};$%
\newline
$(iii)$ $N\geq 5:$ one radial solution exists if and only if
\begin{equation*}
a\leq \left( \frac{N-4}{N-2}\right) ^{\frac{N-2}{2}}\frac{2}{(N-4)b^{\frac{%
N-4}{2}}\int_{\mathbb{R}^{N}}|\nabla \overline{u}|^{2}dx}.
\end{equation*}

Motivated by these findings mentioned above, in the present paper we are
likewise interested in looking for positive solutions of Kirchhoff type
equations. The problem we consider is thus%
\begin{equation}
\left\{
\begin{array}{ll}
-\left( a\int_{\mathbb{R}^{N}}|\nabla u|^{2}dx+b\right) \Delta
u+u=f(x)\left\vert u\right\vert ^{p-2}u & \text{ in }\mathbb{R}^{N}, \\
u\in H^{1}(\mathbb{R}^{N}), &
\end{array}%
\right.  \tag{$E_{a}$}
\end{equation}%
where $N\geq 1,a,b>0,2<p<\min \left\{ 4,2^{\ast }\right\} $ and the function
$f(x)$ satisfies:

\begin{itemize}
\item[$(D1)$] $f\in C(\mathbb{R}^{N})\cap L^{\infty }(\mathbb{R}^{N})$ with $%
f_{\min }=\inf_{x\in \mathbb{R}^{N}}f\left( x\right) >0.$
\end{itemize}

Eq. $(E_{a})$ is variational, and its solutions correspond to critical
points of the energy functional $J_{a}:H^{1}(\mathbb{R}^{N})\rightarrow
\mathbb{R}$ given by
\begin{equation*}
J_{a}\left( u\right) =\frac{a}{4}\left( \int_{\mathbb{R}^{N}}|\nabla
u|^{2}dx\right) ^{2}+\frac{1}{2}\int_{\mathbb{R}^{N}}(b|\nabla
u|^{2}+u^{2})dx-\frac{1}{p}\int_{\mathbb{R}^{N}}f(x)|u|^{p}dx.
\end{equation*}%
Furthermore, one can see that $J_{a}$ is a $C^{1}$ functional with the
derivative given by%
\begin{equation*}
\left\langle J_{a}^{\prime }(u),\varphi \right\rangle =\left( a\int_{\mathbb{%
R}^{N}}|\nabla u|^{2}dx+b\right) \int_{\mathbb{R}^{N}}\nabla u\nabla \varphi
dx+\int_{\mathbb{R}^{N}}(u\varphi -f(x)|u|^{p-2}u\varphi )dx
\end{equation*}%
for all $\varphi \in H^{1}(\mathbb{R}^{N})$, where $J_{a}^{\prime }$ denotes
the Fr\'{e}chet derivative of $J_{a}.$

Distinguishing from the existing literature, this paper is devoted to study
a series of questions as follows:

\begin{itemize}
\item[$(I)$] In spite of the amount of papers dealing with Eq. $(\ref{1-0})$%
, the geometric properties of the energy functional $J_{a}$ have not been
described in detail. One objective of this study is to shed some light on
the behavior of $J_{a}$. We will study whether $J_{a}$ is bounded below or
not, depending on the parameter $a$ and the dimension $N.$

\item[$(II)$] As we can see, the Nehari-Pohozaev manifold can help to find
positive solutions with positive energy for Eq. $(\ref{1-0})$ when the
nonlinearity $h(x,u)$ does not satisfy the (AR)-condition (\ref{1-10}) (see
\cite{G,LY,TC,Y}). However, to our knowledge, such approach is only valid
for the case of $N=3$. In our study, since the nonlinearity $f(x)\left\vert
u\right\vert ^{p-2}u(2<p<\min \left\{ 4,2^{\ast }\right\} )$ does not
satisfy the (AR)-condition (\ref{1-10}) as well, we would like to know
whether there exists an approach to study the existence of positive solution
with positive energy of Eq. $(E_{a})$ in any dimensions $N\geq 1.$

\item[$(III)$] According to the geometry of the energy functional $J_{a},$
we think that $J_{a}$ should have two critical points in some dimensions $N,$
where one is a global minimizer with negative energy and the other one is a
local minimizer with positive energy. In view of this, another objective of
this study is to explore the existence of two positive solutions for Eq. $%
(E_{a})$, which seems not to be involved in the literature.
\end{itemize}

In what follows, without loss of generality, we always assume $b=1.$ For any
$u\in H^{1}(\mathbb{R}^{N})\backslash \left\{ 0\right\} ,$ we define%
\begin{equation}
\overline{A}_{f}(u)=\frac{\left( \int_{\mathbb{R}^{N}}f(x)|u|^{p}dx\right)
^{2/(p-2)}}{\left\Vert u\right\Vert _{D^{1,2}}^{4}\left\Vert u\right\Vert
_{H^{1}}^{2(4-p)/(p-2)}}\text{ for }N\geq 4  \label{15-1}
\end{equation}%
and%
\begin{equation*}
\underline{A}_{f}(u)=\frac{\left( \int_{\mathbb{R}^{N}}f_{\min
}|u|^{p}dx\right) ^{2/(p-2)}}{\left\Vert u\right\Vert
_{D^{1,2}}^{4}\left\Vert u\right\Vert _{H^{1}}^{2(4-p)/(p-2)}}\text{ for }%
N=4,
\end{equation*}%
where $\left\Vert u\right\Vert _{H^{1}}=\left( \int_{\mathbb{R}^{N}}(|\nabla
u|^{2}+u^{2})dx\right) ^{1/2}$ and $\left\Vert u\right\Vert
_{D^{1,2}}=\left( \int_{\mathbb{R}^{N}}|\nabla u|^{2}dx\right) ^{1/2}.$
Since $\overline{A}_{f}$\ and $\underline{A}_{f}$ are $0$-homogeneous, we
can denote the extremal values by%
\begin{equation*}
\overline{\mathbf{A}}_{f}=\sup_{u\in H^{1}(\mathbb{R}^{N})\backslash \left\{
0\right\} }\overline{A}_{f}(u)\text{ and }\underline{\mathbf{A}}%
_{f}=\sup_{u\in H^{1}(\mathbb{R}^{N})\backslash \left\{ 0\right\} }%
\underline{A}_{f}(u).
\end{equation*}%
Applying the Gagliardo-Nirenberg and Young inequalities gives
\begin{eqnarray*}
\overline{A}_{f}(u) &\leq &\left( \frac{f_{\max }C_{p}^{p}\left\Vert
u\right\Vert _{D^{1,2}}^{\frac{N(p-2)}{2}}\left\Vert u\right\Vert
_{L^{2}}^{p-\frac{N(p-2)}{2}}}{\left\Vert u\right\Vert
_{D^{1,2}}^{2(p-2)}\left\Vert u\right\Vert _{H^{1}}^{4-p}}\right) ^{\frac{2}{%
p-2}} \\
&\leq &\left( f_{\max }C_{p}^{p}\right) ^{\frac{2}{p-2}}\left( \frac{%
\left\Vert u\right\Vert _{D^{1,2}}^{\frac{\left( N-4\right) \left(
p-2\right) }{4-p}}\left\Vert u\right\Vert _{L^{2}}^{\frac{2p-N\left(
p-2\right) }{4-p}}}{\left\Vert u\right\Vert _{D^{1,2}}^{2}+\left\Vert
u\right\Vert _{L^{2}}^{2}}\right) ^{\frac{4-p}{p-2}} \\
&\leq &\left( f_{\max }C_{p}^{p}\right) ^{\frac{2}{p-2}}\left( \frac{\max
\left\{ (N-4)(p-2),(2p-N(p-2))\right\} }{2(4-p)}\right) ^{\frac{4-p}{p-2}}
\end{eqnarray*}%
and
\begin{equation*}
\underline{A}_{f}\left( u\right) \leq \left( f_{\min }C_{p}^{p}\right) ^{%
\frac{2}{p-2}}\left( \frac{\left\Vert u\right\Vert _{L^{2}}^{2}}{\left\Vert
u\right\Vert _{D^{1,2}}^{2}+\left\Vert u\right\Vert _{L^{2}}^{2}}\right) ^{%
\frac{4-p}{p-2}}\leq \left( f_{\min }C_{p}^{p}\right) ^{\frac{2}{p-2}},
\end{equation*}%
where $f_{\max }:=\sup_{x\in \mathbb{R}^{N}}f\left( x\right) ,\left\Vert
u\right\Vert _{L^{2}}=\left( \int_{\mathbb{R}^{N}}u^{2}dx\right) ^{1/2}$and $%
C_{p}>0$ is a sharp constant of Gagliardo-Nirenberg inequality. Thus, there
exist two positive numbers $\overline{C}_{0}(N,p,f)$ and $\underline{C}%
_{0}(p,f)$ such that%
\begin{equation*}
0<\overline{\mathbf{A}}_{f}\leq \overline{C}_{0}(N,p,f)\text{ for }N\geq 4
\end{equation*}%
and
\begin{equation*}
0<\underline{\mathbf{A}}_{f}\leq \underline{C}_{0}(p,f)\text{ for }N=4.
\end{equation*}

Let%
\begin{equation}
\overline{a}_{\ast }=\frac{2(p-2)}{4-p}\left( \frac{4-p}{p}\right) ^{2/(p-2)}%
\overline{\mathbf{A}}_{f}\text{ for }N\geq 4  \label{15-3}
\end{equation}%
and
\begin{equation*}
\underline{a}_{\ast }=\frac{2(p-2)}{4-p}\left( \frac{4-p}{p}\right)
^{2/(p-2)}\underline{\mathbf{A}}_{f}\text{ for }N=4.
\end{equation*}

We now summarize the first part of our main results as follows.

\begin{theorem}
\label{t0-1}Suppose that $2<p<\min \left\{ 4,2^{\ast }\right\} $ and
condition $(D1)$ holds. Then the following statements are true.\newline
$(i)$ If $N=1,2,3,$ then $J_{a}$ is not bounded below on $H^{1}(\mathbb{R}%
^{N})$ for all $a>0;$\newline
$(ii)$ If $N=4,$ then for each $0<a<\underline{a}_{\ast },$ $J_{a}$ is not
bounded below on $H^{1}(\mathbb{R}^{N}),$ whereas for each $a>\overline{a}%
_{\ast },$ $J_{a}$ is bounded below on $H^{1}(\mathbb{R}^{N})$ and $%
\inf_{u\in H^{1}(\mathbb{R}^{N})\backslash \{0\}}J_{a}(u)>0;$\newline
$(iii)$ If $N\geq 5,$ then $J_{a}$ is bounded below on $H^{1}(\mathbb{R}%
^{N}) $ for all $a>0.$ More precisely, for each $0<a<\overline{a}_{\ast },$
there holds $-\infty <\inf_{u\in H^{1}(\mathbb{R}^{N})\backslash
\{0\}}J_{a}(u)<0,$ whereas for each $a>\overline{a}_{\ast },$ there holds $%
\inf_{u\in H^{1}(\mathbb{R}^{N})\backslash \{0\}}J_{a}(u)>0.$
\end{theorem}

For brevity, we sum up the main result of Theorem \ref{t0-1} with the table
below:%
\begin{equation*}
\begin{tabular}{|c|c|c|c|c|}
\hline
& $a>0$ & $0<a<\underline{a}_{\ast }$ & $0<a<\overline{a}_{\ast }$ & $a>%
\overline{a}_{\ast }$ \\ \hline
$N=1,2,3$ & $\inf J_{a}(u)=-\infty $ & - & - & - \\ \hline
$N=4$ & - & $\inf J_{a}(u)=-\infty $ & - & $\inf J_{a}(u)>0$ \\ \hline
$N\geq 5$ & $\inf J_{a}(u)>-\infty $ & - & $\inf J_{a}(u)<0$ & $\inf
J_{a}(u)>0$ \\ \hline
\end{tabular}%
\end{equation*}

Note that $\frac{p^{2/\left( p-2\right) }}{2^{p/\left( p-2\right) }}>1,$
since $2<p<2^{\ast }$ for $N\geq 4.$ From Theorem \ref{t0-1} $(ii)-(iii)$,
one can see that $\inf_{u\in H^{1}(\mathbb{R}^{N})\backslash \{0\}}J_{a}(u)>0
$ for $a>\overline{a}_{\ast }$, if $N\geq 4$. However, we obtain the
following nonexistence result.

\begin{theorem}
\label{t0-2}Suppose that $N\geq 4$ and condition $(D1)$ holds. Then for each
$a>\frac{p^{2/\left( p-2\right) }}{2^{p/\left( p-2\right) }}\overline{a}%
_{\ast },$ Eq. $(E_{a})$ does not admit any nontrivial solutions.
\end{theorem}

Next, we need the following assumption on $f.$

\begin{itemize}
\item[$\left( D2\right) $] $\lim_{\left\vert x\right\vert \rightarrow \infty
}f\left( x\right) =f_{\infty }>0\ $uniformly on$\ \mathbb{R}^{N}.$
\end{itemize}

\begin{theorem}
\label{t0-3}$(i)$ Suppose that $N\geq 5$ and $f(x)\equiv f_{\infty }>0.$
Then for each $0<a<\overline{a}_{\ast },$ Eq. $(E_{a})$ has a positive
ground state solution $v_{a}^{+}\in H^{1}(\mathbb{R}^{N})$ satisfying
\begin{equation*}
J_{a}^{\infty }(v_{a}^{+})=\inf_{u\in H^{1}(\mathbb{R}^{N})\backslash
\{0\}}J_{a}^{\infty }(u)<0,
\end{equation*}%
where $J_{a}^{\infty }=J_{a}$ with $f(x)\equiv f_{\infty }.$\newline
$(ii)$ Suppose that $N\geq 5$ and conditions $(D1)-(D2)$ hold. In addition,
we assume that

\begin{itemize}
\item[$(D3)$] $\int_{\mathbb{R}^{N}}(f(x)-f_{\infty })(v_{a}^{+})^{p}dx>0,$
where $v_{a}^{+}$ is the positive solution as described in part $(i)$.
\end{itemize}

Then for each $0<a<\overline{a}_{\ast },$ Eq. $(E_{a})$ has a positive
ground state solution $u_{a}^{+}\in H^{1}(\mathbb{R}^{N})$ satisfying
\begin{equation*}
J_{a}(u_{a}^{+})=\inf_{u\in H^{1}(\mathbb{R}^{N})\backslash \{0\}}J_{a}(u)<0.
\end{equation*}
\end{theorem}

In order to obtain the existence of positive solution with positive energy
for Eq. $(E_{a}),$ it is necessary to introduce the filtration of the Nehari
manifold. That is,
\begin{equation*}
\mathbf{M}_{a}(c)=\{u\in \mathbf{M}_{a}:J_{a}(u)<c\}\text{ for some }c>0,
\end{equation*}%
where $\mathbf{M}_{a}=\{u\in H^{1}(\mathbb{R}^{N})\backslash
\{0\}:\left\langle J_{a}^{\prime }(u),u\right\rangle =0\}$ is the Nehari
manifold. We will show that $\mathbf{M}_{a}(c)$ can be divided into two
parts
\begin{equation*}
\mathbf{M}_{a}^{(1)}(c)=\{u\in \mathbf{M}_{a}(c):\Vert u\Vert
_{H^{1}}<C_{1}\}\ \text{and}\ \mathbf{M}_{a}^{(2)}(c)=\{u\in \mathbf{M}%
_{a}(c):\Vert u\Vert _{H^{1}}>C_{2}\},
\end{equation*}%
in which each local minimizer of the functional $J_{a}$ is a critical point
of $J_{a}$ in $H^{1}(\mathbb{R}^{N})$. Our approach is to minimize the
energy functional $J_{a}$ on $\mathbf{M}_{a}^{(1)}(c)$, where $J_{a}$ is
bounded below and the minimizing sequence is bounded. In fact, such approach
has been applied in the study of Schrodinger-Poisson systems in $\mathbb{R}%
^{3}$ by us (see \cite{SWF1,SWF2}).

We assume that $f$ satisfies the following condition:

\begin{itemize}
\item[$\left( D4\right) $] $f_{\max }=\sup_{x\in \mathbb{R}^{N}}f\left(
x\right) <\frac{f_{\infty }}{D(p)^{(p-2)/2}},$ where%
\begin{equation*}
D(p)=\left\{
\begin{array}{ll}
\left( \frac{4-p}{2}\right) ^{1/(p-2)}, & \text{ if }2<p\leq 3, \\
\frac{1}{2}, & \text{ if }3<p<\min \{4,2^{\ast }\}.%
\end{array}%
\right.
\end{equation*}
\end{itemize}

\begin{remark}
\label{r-1} By a direct calculation, we obtain that for $2<p<4,$
\begin{equation*}
\frac{1}{2}\leq D(p)<\frac{1}{\sqrt{e}}\text{ and }D(p)\left( \frac{2}{4-p}%
\right) ^{2/(p-2)}>1.
\end{equation*}
\end{remark}

Let
\begin{equation}
\Lambda _{0}=\left[ 1-D(p)\left( \frac{f_{\max }}{f_{\infty }}\right)
^{2/(p-2)}\right] \left( \frac{f_{\infty }}{S_{p}^{p}}\right) ^{2/(p-2)},
\label{1-7}
\end{equation}%
where $S_{p}$ is the best Sobolev constant for the embedding of $H^{1}(%
\mathbb{R}^{N})$ in $L^{p}(\mathbb{R}^{N}).$ In particular, if $f(x)\equiv
f_{\infty },$ then equality (\ref{1-7}) becomes
\begin{equation*}
\Lambda _{0}=(1-D(p))\left( \frac{f_{\infty }}{S_{p}^{p}}\right) ^{2/(p-2)}.
\end{equation*}%
Set%
\begin{equation*}
\Lambda =\left\{
\begin{array}{ll}
\frac{4-p}{2}\left( \frac{f_{\infty }(4-p)}{2pS_{p}^{p}}\right) ^{2/(p-2)} &
\text{ if }N=1,2,3, \\
\min \left\{ \frac{p-2}{2(4-p)}\left( \frac{4-p}{p}\right) ^{2/(p-2)}\Lambda
_{0},\overline{a}_{\ast }\right\} & \text{ if }N\geq 4.%
\end{array}%
\right.
\end{equation*}%
Let $w_{0}$ be the unique positive solution of the following Schr\"{o}dinger
equation%
\begin{equation}
\begin{array}{ll}
-\Delta u+u=f_{\infty }|u|^{p-2}u & \text{ in }\mathbb{R}^{N}.%
\end{array}
\tag*{$\left( E_{0}^{\infty }\right) $}
\end{equation}%
From \cite{K}, we see that
\begin{equation}
\left\Vert w_{0}\right\Vert _{H^{1}}^{2}=\int_{\mathbb{R}^{N}}f_{\infty
}|w_{0}|^{p}dx=\left( \frac{S_{p}^{p}}{f_{\infty }}\right) ^{2/(p-2)},
\label{1-8}
\end{equation}%
and
\begin{equation*}
J_{0}^{\infty }(w_{0})=\frac{p-2}{2p}\left( \frac{S_{p}^{p}}{f_{\infty }}%
\right) ^{2/(p-2)},
\end{equation*}%
where $J_{0}^{\infty }$ is the energy functional of equation $(E_{0}^{\infty
})$ in $H^{1}(\mathbb{R}^{N})$ in the form
\begin{equation*}
J_{0}^{\infty }(u)=\frac{1}{2}\left\Vert u\right\Vert _{H^{1}}^{2}-\frac{1}{p%
}\int_{\mathbb{R}^{N}}f_{\infty }|u|^{p}dx.
\end{equation*}

We now summarize the second part of our main results as follows.

\begin{theorem}
\label{t1}Assume that $f(x)\equiv f_{\infty }>0.$ Then the following
statements are true.\newline
$\left( i\right) $ If $N\geq 1,$ then for each $0<a<\Lambda ,$ Eq. $(E_{a})$
has at least a positive solution $v_{a}^{-}\in H^{1}(\mathbb{R}^{N})$
satisfying
\begin{equation*}
\left\Vert v_{a}^{-}\right\Vert _{H^{1}}<\left( \frac{2S_{p}^{p}}{f_{\infty
}(4-p)}\right) ^{1/(p-2)}\text{ and }J_{a}^{\infty }(v_{a}^{-})>\frac{p-2}{2p%
}\left( \frac{S_{p}^{p}}{f_{\infty }}\right) ^{2/(p-2)}>0;
\end{equation*}%
\newline
$\left( ii\right) $ If $1\leq N\leq 4,$ then for each $0<a<\Lambda ,$ Eq. $%
(E_{a})$ has a unique positive solution $v_{a}^{-}\in H^{1}(\mathbb{R}^{N}),$
which is radially symmetric;\newline
$\left( iii\right) $ If $N\geq 5,$ then for each $0<a<\Lambda ,$ Eq. $%
(E_{a}) $ has at least two positive solutions $v_{a}^{-},v_{a}^{+}\in H^{1}(%
\mathbb{R}^{N})$ satisfying%
\begin{equation*}
\left\Vert v_{a}^{-}\right\Vert _{H^{1}}<\left( \frac{2S_{p}^{p}}{f_{\infty
}(4-p)}\right) ^{1/(p-2)}<\sqrt{2}\left( \frac{2S_{p}^{p}}{f_{\infty }(4-p)}%
\right) ^{1/(p-2)}<\left\Vert v_{a}^{+}\right\Vert _{H^{1}}
\end{equation*}%
and $J_{a}^{\infty }(v_{a}^{+})<0<J_{a}^{\infty }(v_{a}^{-}).$ In
particular, $v_{a}^{+}$ is a ground state solution of Eq. $(E_{a}).$
\end{theorem}

With the aid of Theorem \ref{t1}, we obtain the following results in the
nonautnomous case.

\begin{theorem}
\label{t2}Suppose that $N\geq 1$ and conditions $(D1)-(D2),(D4)$ hold. In
addition, we assume that

\begin{itemize}
\item[$(D5)$] $\int_{\mathbb{R}^{N}}(f(x)-f_{\infty })(v_{a}^{-})^{p}dx>0,$
where $v_{a}^{-}$ is the positive solution as described in Theorem \ref{t1}.
\end{itemize}

Then for each $0<a<\Lambda ,$ Eq. $(E_{a})$ has at least a positive solution
$u_{a}^{-}\in H^{1}(\mathbb{R}^{N})$ satisfying
\begin{equation*}
\left\Vert u_{a}^{-}\right\Vert _{H^{1}}<\left( \frac{2S_{p}^{p}}{f_{\max
}\left( 4-p\right) }\right) ^{1/(p-2)}\text{ and }J_{a}(u_{a}^{-})>\frac{p-2%
}{4p}\left( \frac{S_{p}^{p}}{f_{\max }}\right) ^{2/(p-2)}.
\end{equation*}
\end{theorem}

\begin{theorem}
\label{t3}Suppose that $N\geq 5$ and conditions $(D1)-(D5)$ hold. Then for
each $0<a<\Lambda ,$ Eq. $(E_{a})$ has at least two positive solutions $%
u_{a}^{-},u_{a}^{+}\in H^{1}(\mathbb{R}^{N})$ satisfying
\begin{equation*}
\left\Vert u_{a}^{-}\right\Vert _{H^{1}}<\left( \frac{2S_{p}^{p}}{f_{\max
}(4-p)}\right) ^{1/(p-2)}<\sqrt{2}\left( \frac{2S_{p}^{p}}{f_{\max }(4-p)}%
\right) ^{1/(p-2)}<\left\Vert u_{a}^{+}\right\Vert _{H^{1}}
\end{equation*}%
and
\begin{equation*}
J_{a}(u_{a}^{+})<0<\frac{p-2}{4p}\left( \frac{S_{p}^{p}}{f_{\max }}\right)
^{2/(p-2)}<J_{a}(u_{a}^{-}).
\end{equation*}%
In particular, $u_{a}^{+}$ is a ground state solution of Eq. $(E_{a}).$
\end{theorem}

\begin{theorem}
\label{t4} Let $u_{a}^{-}$ be the positive solution of Eq. $(E_{a})$ as
described in Theorem \ref{t1} or \ref{t2}. Then we have the following
conclusions:\newline
$\left( i\right) $ If $N=1$ and $f(x)$ is weakly differentiable satisfying
\begin{equation*}
(p-1)(p-2)f\left( x\right) +2\langle \nabla f(x),x\rangle \geq 0,
\end{equation*}%
then $u_{a}^{-}$ is a positive ground state solution of Eq. $(E_{a}).$%
\newline
$\left( ii\right) $ If $N=2$ and $f(x)$ is weakly differentiable satisfying%
\begin{equation*}
(p-2)f(x)+\langle \nabla f(x),x\rangle \geq 0,
\end{equation*}%
then $u_{a}^{-}$ is a positive ground state solution of Eq. $(E_{a}).$
\end{theorem}

The following table sums up the above main results:

\begin{equation*}
\begin{tabular}{|c|c|c|}
\hline
& $a$ small & $a$ large \\ \hline
$N=1,2,3$ & One solution & - \\ \hline
$N=4$ & One solution & No solution \\ \hline
$N\geq 5$ & Two solutions & No solution \\ \hline
\end{tabular}%
\end{equation*}

In the above table, "Two solutions" (respectively "One solution") means that
there exist at least two (respectively one) positive solutions. On the other
hand, "No solution" means that there are no nontrivial solutions.

Recently, Azzollini \cite{Az1,Az2} has proved that Eq. $(E_{a})$ with $%
f(x)\equiv f_{\infty }$ admits a ground state solution with positive energy
for all $a>0$ when $N=3$ and for $a>0$ sufficiently small when $N=4.$ In the
following, we shall further describe some characteristics of such solution
depending on $a$ and $f_{\infty }$, which are not concerned in \cite{Az1,Az2}%
.

Define the fibering map $h_{a,u}:t\rightarrow J_{a}\left( tu\right) $ as%
\begin{equation*}
h_{a,u}\left( t\right) =\frac{t^{2}}{2}\left\Vert u\right\Vert _{H^{1}}^{2}+%
\frac{at^{4}}{4}\left( \int_{\mathbb{R}^{N}}\left\vert \nabla u\right\vert
^{2}dx\right) ^{2}-\frac{t^{p}}{p}\int_{\mathbb{R}^{N}}f(x)\left\vert
u\right\vert ^{p}dx\text{ for }t>0.
\end{equation*}%
About its theory and application, we refer the reader to \cite{BZ,DP}. Note
that for $u\in H^{1}(\mathbb{R}^{N})\backslash \left\{ 0\right\} $ and $t>0,$
$h_{a,u}^{\prime }\left( t\right) =0$ holds if and only if $tu\in \mathbf{M}%
_{a}$. In particular, $h_{a,u}^{\prime }\left( 1\right) =0$ holds if and
only if $u\in \mathbf{M}_{a}.$ It is natural to split $\mathbf{M}_{a}$ into
three parts corresponding to the local minima, local maxima and points of
inflection. As a consequence, following \cite{T}, we can define
\begin{eqnarray*}
\mathbf{M}_{a}^{+} &=&\{u\in \mathbf{M}_{a}:h_{a,u}^{\prime \prime }\left(
1\right) >0\}, \\
\mathbf{M}_{a}^{0} &=&\{u\in \mathbf{M}_{a}:h_{a,u}^{\prime \prime }\left(
1\right) =0\}, \\
\mathbf{M}_{a}^{-} &=&\{u\in \mathbf{M}_{a}:h_{a,u}^{\prime \prime }\left(
1\right) <0\}.
\end{eqnarray*}

For $2<p<4,$ we set%
\begin{eqnarray}
A_{0} &=&\frac{3\left( p-1\right) (-p^{2}+2p+12)}{p^{2}(p-2)}\left( \frac{%
S_{p}^{p}}{f_{\infty }}\right) ^{2/(p-2)}>0,  \label{1-3} \\
\overline{A}_{0} &=&\frac{p^{2}}{16}\left( \frac{f_{\infty }}{S_{p}^{p}}%
\right) ^{2/(p-2)}>0,  \label{1-4} \\
A_{0}^{\ast } &=&\frac{p-2}{2}\left( \frac{4-p}{p}\right) ^{\left(
4-p\right) /(p-2)}(f_{\infty }C_{p}^{p})^{2/(p-2)}>0.  \label{1-5}
\end{eqnarray}

It is clearly that $\overline{A}_{0}>\Lambda .$ We now state the last part
of our main results as follows.

\begin{theorem}
\label{t5}Let $u_{0}$ be a nontrivial solution of Eq. $(E_{a})$ with $%
f(x)\equiv f_{\infty }$. Then the following statements are true.\newline
$(i)$ When $N=3,$ for each $a>0$ with $\sqrt{a^{2}+4}+\frac{2}{a}\geq A_{0},$
there holds $u_{0}\in \mathbf{M}_{a}^{-}.$ In particular, $v_{a}^{-}$ is a
ground state solution as in Theorem \ref{t1} $\left( i\right) $ for $N=3.$%
\newline
$(ii)$ When $N=4,$ for each $0<a\leq \overline{A}_{0},$ there holds $%
u_{0}\in \mathbf{M}_{a}^{-},$ whereas for each $a>A_{0}^{\ast },$ there
holds $u_{0}\in \mathbf{M}_{a}^{+}.$ In particular, $v_{a}^{-}$ is a ground
state solution as in Theorem \ref{t1} $\left( i\right) $ for $N=4.$
\end{theorem}

\begin{remark}
$(i)$ Note that $\inf_{a>0}\left( \sqrt{a^{2}+4}+\frac{2}{a}\right) >0.$
Then when%
\begin{equation*}
f_{\infty }\geq \left[ \frac{3\left( p-1\right) (-p^{2}+2p+12)}{%
p^{2}(p-2)\inf_{a>0}\left( \sqrt{a^{2}+4}+\frac{2}{a}\right) }\right]
^{\left( p-2\right) /2}S_{p}^{p},
\end{equation*}%
there holds $\sqrt{a^{2}+4}+\frac{2}{a}\geq A_{0}$ for all $a>0.$ This shows
that $u_{0}\in \mathbf{M}_{a}^{-}$ for all $a>0.$\newline
$(ii)$ If there is a number $a_{0}>0$ such that $\sqrt{a_{0}^{2}+4}+\frac{2}{%
a_{0}}<A_{0},$ then there exist positive numbers $a_{1},a_{2}$ such that $%
\sqrt{a^{2}+4}+\frac{2}{a}\geq A_{0}$ for all $a\in (0,a_{1}]\cup \lbrack
a_{2},\infty )$ and $\sqrt{a^{2}+4}+\frac{2}{a}<A_{0}$ for all $a\in
(a_{1},a_{2}).$
\end{remark}

The structure of this paper is as follows. After briefly introducing some
technical lemmas in Section 2, we prove Theorems \ref{t0-1}--\ref{t0-3} in
Section 3, and demonstrate proofs of Theorems \ref{t1} and \ref{t2} in
Sections 4 and 5, respectively. Section 6 is dedicated to the proof of
Theorems \ref{t4} and \ref{t5}.

\section{Preliminaries}

Firstly, we consider the boundedness below for energy functional $J_{a}$ on $%
H^{1}(\mathbb{R}^{N})$ for $N\geq 4.$ For $u\in H^{1}(\mathbb{R}%
^{N})\backslash \{0\},$ we define
\begin{equation}
T_{f}(u)=\left( \frac{\left\Vert u\right\Vert _{H^{1}}^{2}}{\int_{\mathbb{R}%
^{N}}f(x)|u|^{p}dx}\right) ^{1/(p-2)}.  \label{2-0}
\end{equation}

\begin{lemma}
\label{g4}Suppose that $N\geq 4$ and condition $(D1)$ holds. Then the
following results are true.\newline
$\left( i\right) $ For each $0<a<\overline{a}_{\ast }$ and $u\in H^{1}(%
\mathbb{R}^{N})\backslash \{0\},$ there exists a constant $\widehat{t}%
_{a}^{(0)}>\left( \frac{p}{4-p}\right) ^{1/(p-2)}T_{f}(u)$ such that%
\begin{equation}
\inf_{t\geq 0}J_{a}(tu)=\inf_{\left( \frac{p}{4-p}\right)
^{1/(p-2)}T_{f}(u)<t<\widehat{t}_{a}^{(0)}}J_{a}(tu)<0.  \label{eqq17}
\end{equation}%
$\left( ii\right) $ For each $a\geq \overline{a}_{\ast }$ and $u\in H^{1}(%
\mathbb{R}^{N})\backslash \left\{ 0\right\} ,$ there holds $J_{a}\left(
tu\right) \geq 0$ for all $t>0.$
\end{lemma}

\begin{proof}
$\left( i\right) $ For $u\in H^{1}(\mathbb{R}^{N})\backslash \{0\}$ and $t>0$%
, it has
\begin{eqnarray*}
h_{a,u}(t) &=&J_{a}(tu)=\frac{t^{2}}{2}\left\Vert u\right\Vert _{H^{1}}^{2}+%
\frac{at^{4}}{4}\left( \int_{\mathbb{R}^{N}}|\nabla u|^{2}dx\right) ^{2}-%
\frac{t^{p}}{p}\int_{\mathbb{R}^{N}}f(x)|u|^{p}dx \\
&=&t^{4}\left[ g(t)+\frac{a}{4}\left( \int_{\mathbb{R}^{N}}|\nabla
u|^{2}dx\right) ^{2}\right] ,
\end{eqnarray*}%
where
\begin{equation*}
g(t)=\frac{t^{-2}}{2}\left\Vert u\right\Vert _{H^{1}}^{2}-\frac{t^{p-4}}{p}%
\int_{\mathbb{R}^{N}}f(x)|u|^{p}dx.
\end{equation*}%
Clearly, $J_{a}(tu)=0$ if and only if
\begin{equation*}
g(t)+\frac{a}{4}\left( \int_{\mathbb{R}^{N}}|\nabla u|^{2}dx\right) ^{2}=0.
\end{equation*}%
It is easy to see that
\begin{equation*}
g(\hat{t}_{a})=0,\ \lim_{t\rightarrow 0^{+}}g(t)=\infty \ \text{and }%
\lim_{t\rightarrow \infty }g(t)=0,
\end{equation*}%
where $\hat{t}_{a}=\left( \frac{p}{2}\right) ^{1/(p-2)}T_{f}(u).$ By
calculating the derivative of $g(t)$, we obtain
\begin{equation*}
g^{\prime }(t)=t^{-3}\left[ -\left\Vert u\right\Vert _{H^{1}}^{2}+\frac{%
(4-p)t^{p-2}}{p}\int_{\mathbb{R}^{N}}f(x)|u|^{p}dx\right] ,
\end{equation*}%
which implies that $g(t)$ is decreasing when $0<t<\left( \frac{p}{4-p}%
\right) ^{1/(p-2)}T_{f}(u)$ and is increasing when $t>\left( \frac{p}{4-p}%
\right) ^{1/(p-2)}T_{f}(u).$ This indicates that%
\begin{eqnarray}
\inf_{t>0}g(t) &=&g\left( \left( \frac{p}{4-p}\right)
^{1/(p-2)}T_{f}(u)\right)  \notag \\
&=&-\frac{p-2}{2(4-p)}\left( \frac{p\left\Vert u\right\Vert _{H^{1}}^{2}}{%
(4-p)\int_{\mathbb{R}^{N}}f(x)|u|^{p}dx}\right) ^{-2/(p-2)}\Vert u\Vert
_{H^{1}}^{2}.  \label{3-5}
\end{eqnarray}%
Note that
\begin{equation*}
0<a<\frac{2(p-2)(4-p)^{(4-p)/(p-2)}}{p^{2/(p-2)}}\overline{\mathbf{A}}_{f}.
\end{equation*}%
Then there exists $u\in H^{1}(\mathbb{R}^{N})\backslash \{0\}$ such that%
\begin{equation*}
\frac{2(p-2)(4-p)^{(4-p)/(p-2)}\left( \int_{\mathbb{R}^{N}}f(x)|u|^{p}dx%
\right) ^{2/(p-2)}}{p^{2/(p-2)}\left\Vert u\right\Vert
_{H^{1}}^{2(4-p)/(p-2)}\left\Vert u\right\Vert _{D^{1,2}}^{4}}>a.
\end{equation*}%
Using the above inequality, together with $(\ref{3-5}),$ leads to%
\begin{equation*}
\inf_{t>0}g(t)<-\frac{a}{4}\left\Vert u\right\Vert _{D^{1,2}}^{4}.
\end{equation*}%
This implies that there exist two numbers $\widehat{t}_{a}^{(0)}$ and $%
\widehat{t}_{a}^{(1)}$ satisfying
\begin{equation*}
0<\widehat{t}_{a}^{(1)}<\left( \frac{p}{4-p}\right) ^{1/(p-2)}T_{f}(u)<%
\widehat{t}_{a}^{(0)}
\end{equation*}%
such that
\begin{equation*}
g\left( \widehat{t}_{a}^{(j)}\right) +\frac{a}{4}\left\Vert u\right\Vert
_{D^{1,2}}^{4}=0\text{ for }j=0,1.
\end{equation*}%
That is,
\begin{equation*}
J_{a}\left( \widehat{t}_{a}^{(j)}u\right) =0\text{ for }j=0,1.
\end{equation*}%
Thus,
\begin{eqnarray*}
J_{a}\left( \left( \frac{p}{4-p}\right) ^{1/(p-2)}T_{f}(u)u\right)&=&\left[
\left( \frac{p}{4-p}\right) ^{1/(p-2)}T_{f}(u)\right]^{4}\left[ g\left(
\left( \frac{p}{4-p}\right) ^{1/(p-2)}T_{f}(u)\right) +\frac{a}{4}\left\Vert
u\right\Vert _{D^{1,2}}^{4}\right] \\
&<&0,
\end{eqnarray*}%
and so%
\begin{equation*}
\inf_{t\geq 0}J_{a}(tu)<0.
\end{equation*}%
Note that%
\begin{equation*}
h_{a,u}^{\prime }(t)=4t^{3}\left( g(t)+\frac{a}{4}\left\Vert u\right\Vert
_{D^{1,2}}^{4}\right) +t^{4}g^{\prime }(t).
\end{equation*}%
Then we have
\begin{equation*}
h_{a,u}^{\prime }(t)<0\text{ for all }t\in \left( \widehat{t}%
_{a}^{(1)},\left( \frac{p}{4-p}\right) ^{1/(p-2)}T_{f}(u)\right]
\end{equation*}%
and%
\begin{equation*}
h_{a,u}^{\prime }\left( \widehat{t}_{a}^{(0)}\right) >0.
\end{equation*}%
Consequently, we arrive at inequality (\ref{eqq17}).\newline
$\left( ii\right) $ For each $u\in H^{1}(\mathbb{R}^{N})\backslash \{0\},$
we can find a unique $t_{0}:=t_{0}(u)>0$ such that $h_{a,u}(t_{0})=0$ and $%
h_{a,u}^{\prime }(t_{0})=0.$ In fact, we only need to solve the system with
respect to the variables $t,a$%
\begin{equation*}
\left\{
\begin{array}{c}
h_{a,u}(t)=t^{2}\left( \frac{1}{2}\left\Vert u\right\Vert _{H^{1}}^{2}+\frac{%
at^{2}}{4}\left\Vert u\right\Vert _{D^{1,2}}^{4}-\frac{t^{p-2}}{p}\int_{%
\mathbb{R}^{N}}f(x)|u|^{p}dx\right) =0, \\
h_{a,u}^{\prime }(t)=t\left( \left\Vert u\right\Vert
_{H^{1}}^{2}+at^{2}\left\Vert u\right\Vert _{D^{1,2}}^{4}-t^{p-2}\int_{%
\mathbb{R}^{N}}f(x)|u|^{p}dx\right) =0.%
\end{array}%
\right.
\end{equation*}%
A direct calculation shows that%
\begin{equation*}
t_{0}(u)=\left( \frac{2(p-2)\int_{\mathbb{R}^{N}}f(x)|u|^{p}dx}{ap\left\Vert
u\right\Vert _{D^{1,2}}^{4}}\right) ^{1/(4-p)},
\end{equation*}%
and accordingly,
\begin{eqnarray*}
a_{0}(u) &=&\frac{2(p-2)(4-p)^{(4-p)/(p-2)}\left( \int_{\mathbb{R}%
^{N}}f(x)|u|^{p}dx\right) ^{2/(p-2)}}{p^{2/(p-2)}\left\Vert u\right\Vert
_{D^{1,2}}^{4}\left\Vert u\right\Vert _{H^{1}}^{2(4-p)/(p-2)}} \\
&=&\frac{2(p-2)(4-p)^{(4-p)/(p-2)}}{p^{2/(p-2)}}\overline{A}_{f}(u).
\end{eqnarray*}%
Since
\begin{equation*}
\overline{a}_{\ast }=\frac{2(p-2)(4-p)^{(4-p)/(p-2)}}{p^{2/(p-2)}}\sup_{u\in
H^{1}(\mathbb{R}^{N})\backslash \{0\}}\overline{A}_{f}(u)=\frac{%
2(p-2)(4-p)^{(4-p)/(p-2)}}{p^{2/(p-2)}}\overline{\mathbf{A}}_{f},
\end{equation*}%
we have for each $a\geq \overline{a}_{\ast }$ and $u\in H^{1}(\mathbb{R}%
^{N})\backslash \{0\},$
\begin{equation*}
h_{a,u}(t)=J_{a}(tu)\geq 0\text{ for all }t>0.
\end{equation*}%
This completes the proof.
\end{proof}

\begin{corollary}
\label{g17}Suppose that $N=4$ and condition $(D1)$ holds. Then for all $%
a\geq \overline{a}_{\ast }$, $J_{a}$ is bounded below on $H^{1}(\mathbb{R}%
^{N})$ and $\inf_{u\in H^{1}(\mathbb{R}^{N})\backslash \{0\}}J_{a}(u)\geq 0.$
\end{corollary}

\begin{lemma}
\label{g11}Suppose that $N\geq 5$ and condition $(D1)$ holds. Then for all $%
a>0,$ $J_{a}$ is bounded below on $H^{1}(\mathbb{R}^{N})$ and there exist
numbers $\widehat{r},\widehat{R}_{a}>0$ such that%
\begin{equation*}
J_{a}(u)>0\text{ for all }u\in H^{1}(\mathbb{R}^{N})\text{ with }%
0<\left\Vert u\right\Vert _{H^{1}}<\widehat{r}\text{ or }\left\Vert
u\right\Vert _{H^{1}}\geq \widehat{R}_{a}.
\end{equation*}%
Furthermore, for each $0<a<\overline{a}_{\ast },$ there holds%
\begin{equation*}
-\infty <\inf_{\widehat{r}<\left\Vert u\right\Vert _{H^{1}}<\widehat{R}%
_{a}}J_{a}(u)=\inf_{u\in H^{1}(\mathbb{R}^{N})\backslash \{0\}}J_{a}(u)<0.
\end{equation*}
\end{lemma}

\begin{proof}
Applying the Gagliardo-Nirenberg and Young inequalities leads to
\begin{eqnarray*}
J_{a}\left( u\right) &\geq &\frac{a}{4}\left\Vert u\right\Vert
_{D^{1,2}}^{4}+\frac{1}{2}\left\Vert u\right\Vert _{H^{1}}^{2}-\frac{f_{\max
}C_{p}^{p}}{p}\left\Vert u\right\Vert _{D^{1,2}}^{\alpha p}\left\Vert
u\right\Vert _{L^{2}}^{(1-\alpha )p} \\
&\geq &\frac{a}{4}\left\Vert u\right\Vert _{D^{1,2}}^{4}+\frac{1}{2}%
\left\Vert u\right\Vert _{H^{1}}^{2}-\frac{2^{\ast }}{\alpha p^{2}}\left(
f_{\max }C_{p}^{p}\beta ^{-\frac{(1-\alpha )p}{2}}\right) ^{\frac{2^{\ast }}{%
\alpha p}}\left\Vert u\right\Vert _{D^{1,2}}^{2^{\ast }}-\frac{\beta }{%
(1-\alpha )p^{2}}\left\Vert u\right\Vert _{L^{2}}^{2} \\
&\geq &\left\Vert u\right\Vert _{D^{1,2}}^{2^{\ast }}\left[ \frac{a}{4}%
\left\Vert u\right\Vert _{D^{1,2}}^{4-2^{\ast }}-\frac{2^{\ast }}{\alpha
p^{2}}\left( f_{\max }C_{p}^{p}\beta ^{-\frac{(1-\alpha )p}{2}}\right) ^{%
\frac{2^{\ast }}{\alpha p}}\right] \\
&\geq &-D_{0}\text{ for some }D_{0}>0,
\end{eqnarray*}%
where $\alpha =\frac{2^{\ast }(p-2)}{p(2^{\ast }-2)}$ and $0<\beta <\frac{%
p(2^{\ast }-p)}{2^{\ast }-2}.$ This implies that $J_{a}(u)$ is bounded below
on $H^{1}(\mathbb{R}^{N})$ for all $a>0.$ Moreover, for each $a>0,$ there
exists%
\begin{equation*}
R_{a}:=\left[ \frac{42^{\ast }}{a\alpha p^{2}}\left( f_{\max }C_{p}^{p}\beta
^{-\frac{(1-\alpha )p}{2}}\right) ^{\frac{2^{\ast }}{\alpha p}}\right]
^{1/(4-2^{\ast })}>0
\end{equation*}%
such that%
\begin{equation*}
J_{a}(u)>0\text{ for all }u\in H^{1}(\mathbb{R}^{N})\text{ with }\left\Vert
u\right\Vert _{D^{1,2}}>R_{a}.
\end{equation*}

Let
\begin{equation*}
\widehat{R}_{a}=\left[ R_{a}+\left( \frac{1}{2}-\frac{1}{(1-\alpha )p}%
\right) ^{-1}\frac{2^{\ast }}{\alpha p}R_{a}^{2^{\ast }}\right] ^{1/2}.
\end{equation*}%
We now prove that
\begin{equation*}
J_{a}(u)>0\text{ for all }u\in H^{1}(\mathbb{R}^{N})\text{ with }\left\Vert
u\right\Vert _{H^{1}}>\widehat{R}_{a}.
\end{equation*}%
Let $u\in H^{1}(\mathbb{R}^{N})$ with $\left\Vert u\right\Vert _{H^{1}}\geq
\widehat{R}_{a}.$ If $\left\Vert u\right\Vert _{D^{1,2}}>R_{a},$ then the
result is done clearly. If $\left\Vert u\right\Vert _{D^{1,2}}<R_{a},$ then
it is enough to show that $J_{a}(u)\geq 0$ when
\begin{equation*}
\int_{\mathbb{R}^{N}}u^{2}dx>\left( \frac{1}{2}-\frac{\beta }{\left(
1-\alpha \right) p^{2}}\right) ^{-1}\frac{2^{\ast }}{\alpha p}\left( f_{\max
}C_{p}^{p}\beta ^{-\frac{(1-\alpha )p}{2}}\right) ^{\frac{2^{\ast }}{\alpha p%
}}R_{a}^{2^{\ast }}.
\end{equation*}%
Indeed, note that%
\begin{equation}
\frac{f_{\max }}{p}\int_{\mathbb{R}^{N}}|u|^{p}dx\leq \frac{2^{\ast }}{%
\alpha p^{2}}\left( f_{\max }C_{p}^{p}\beta ^{-\frac{(1-\alpha )p}{2}%
}\right) ^{\frac{2^{\ast }}{\alpha p}}R_{a}^{2^{\ast }}+\frac{\beta }{%
(1-\alpha )p^{2}}\int_{\mathbb{R}^{N}}u^{2}dx.  \notag
\end{equation}%
Then we have%
\begin{eqnarray*}
J_{a}(u) &\geq &\frac{a}{4}\left\Vert u\right\Vert _{D^{1,2}}^{4}+\frac{1}{2}%
\left\Vert u\right\Vert _{H^{1}}^{2}-\frac{f_{\max }}{p}\int_{\mathbb{R}%
^{N}}|u|^{p}dx \\
&>&\left( \frac{1}{2}-\frac{\beta }{\left( 1-\alpha \right) p^{2}}\right)
\int_{\mathbb{R}^{N}}u^{2}dx-\frac{2^{\ast }}{\alpha p^{2}}\left( f_{\max
}C_{p}^{p}\beta ^{-\frac{(1-\alpha )p}{2}}\right) ^{\frac{2^{\ast }}{\alpha p%
}}R_{a}^{2^{\ast }} \\
&>&0.
\end{eqnarray*}%
Thus, we obtain that there exists a positive number $\widehat{R}_{a}>R_{a}$
such that%
\begin{equation*}
J_{a}(u)>0\text{ for all }u\in H^{1}(\mathbb{R}^{N})\text{ with }\left\Vert
u\right\Vert _{H^{1}}\geq \widehat{R}_{a}.
\end{equation*}

Moreover, using the Sobolev inequality gives%
\begin{equation*}
J_{a}(u)\geq \left\Vert u\right\Vert _{H^{1}}^{2}\left( \frac{1}{2}-\frac{%
f_{\max }}{pS_{p}^{p}}\left\Vert u\right\Vert _{H^{1}}^{p-2}\right) ,
\end{equation*}%
which implies that there exists a number%
\begin{equation*}
\widehat{r}:=\left( \frac{pS_{p}^{p}}{2f_{\max }}\right) ^{1/(p-2)}>0
\end{equation*}%
such that%
\begin{equation*}
J_{a}(u)>0\text{ for all }u\in H^{1}(\mathbb{R}^{N})\text{ with }%
0<\left\Vert u\right\Vert _{H^{1}}<\widehat{r}.
\end{equation*}%
Hence, we have%
\begin{equation*}
J_{a}(u)>0\text{ for all }u\in H^{1}(\mathbb{R}^{N})\text{ with }%
0<\left\Vert u\right\Vert _{H^{1}}<\widehat{r}\text{ or }\left\Vert
u\right\Vert _{H^{1}}>\widehat{R}_{a}.
\end{equation*}

It follows from Lemma \ref{g4}$(i)$ that for each $0<a<\overline{a}_{\ast },$%
\begin{equation*}
-\infty <\inf_{\widehat{r}\leq \left\Vert u\right\Vert _{H^{1}}\leq \widehat{%
R}_{a}}J_{a}(u)=\inf_{u\in H^{1}(\mathbb{R}^{N})\backslash \{0\}}J_{a}(u)<0.
\end{equation*}%
Consequently, this completes the proof.
\end{proof}

As pointed out in the section of Introduction, the Nehari manifold $\mathbf{M%
}_{a}$ given by%
\begin{equation*}
\mathbf{M}_{a}=\{u\in H^{1}(\mathbb{R}^{N})\backslash \{0\}:\left\langle
J_{a}^{\prime }(u),u\right\rangle =0\},
\end{equation*}%
can be decomposed into three parts, i.e.,
\begin{equation*}
\mathbf{M}_{a}=\mathbf{M}_{a}^{-}\cup \mathbf{M}_{a}^{0}\cup \mathbf{M}%
_{a}^{+}.
\end{equation*}%
Then we have the following result.

\begin{lemma}
\label{g2}Assume that $u_{0}$ is a local minimizer for $J_{a}$ on $\mathbf{M}%
_{a}$ and $u_{0}\notin \mathbf{M}_{a}^{0}.$ Then $J_{a}^{\prime }(u_{0})=0$
in $H^{-1}(\mathbb{R}^{N}).$
\end{lemma}

\begin{proof}
The proof of Lemma \ref{g2} is essentially same as that in Brown-Zhang \cite[%
Theorem 2.3]{BZ}, so we omit it here.
\end{proof}

Note that $u\in \mathbf{M}_{a}$ if and only if $\left\Vert u\right\Vert
_{H^{1}}^{2}+a\left( \int_{\mathbb{R}^{N}}|\nabla u|^{2}dx\right) ^{2}-\int_{%
\mathbb{R}^{N}}f(x)|u|^{p}dx=0.$ By the Sobolev inequality one has%
\begin{equation*}
\left\Vert u\right\Vert _{H^{1}}^{2}\leq \left\Vert u\right\Vert
_{H^{1}}^{2}+a\left( \int_{\mathbb{R}^{N}}|\nabla u|^{2}dx\right) ^{2}\leq
S_{p}^{-p}f_{\max }\left\Vert u\right\Vert _{H^{1}}^{p}\text{ for all }u\in
\mathbf{M}_{a},
\end{equation*}%
which leads to
\begin{equation}
\int_{\mathbb{R}^{N}}f(x)|u|^{p}dx\geq \left\Vert u\right\Vert
_{H^{1}}^{2}\geq \left( \frac{S_{p}^{p}}{f_{\max }}\right) ^{2/(p-2)}>0\text{
for all }u\in \mathbf{M}_{a}.  \label{2}
\end{equation}%
Moreover, for all $u\in \mathbf{M}_{a},$ we have
\begin{eqnarray}
h_{a,u}^{\prime \prime }(1) &=&\left\Vert u\right\Vert _{H^{1}}^{2}+3a\left(
\int_{\mathbb{R}^{N}}|\nabla u|^{2}dx\right) ^{2}-(p-1)\int_{\mathbb{R}%
^{N}}f(x)|u|^{p}dx  \notag \\
&=&-(p-2)\left\Vert u\right\Vert _{H^{1}}^{2}+a(4-p)\left( \int_{\mathbb{R}%
^{N}}|\nabla u|^{2}dx\right) ^{2}  \notag \\
&=&-2\left\Vert u\right\Vert _{H^{1}}^{2}+(4-p)\int_{\mathbb{R}%
^{N}}f(x)|u|^{p}dx.  \label{2-2}
\end{eqnarray}%
Thus, using $\left( \ref{2}\right) $ and $(\ref{2-2})$ gives
\begin{equation*}
J_{a}(u)=\frac{1}{4}\left\Vert u\right\Vert _{H^{1}}^{2}-\frac{(4-p)}{4p}%
\int_{\mathbb{R}^{N}}f(x)|u|^{p}dx>\frac{p-2}{4p}\left( \frac{S_{p}^{p}}{%
f_{\max }}\right) ^{2/(p-2)}\text{ for all }u\in \mathbf{M}_{a}^{-}
\end{equation*}%
We need the following conclusion.

\begin{lemma}
\label{g1}Suppose that $N\geq 1$ and $2<p<\min \{4,2^{\ast }\}.$ Then $J_{a}$
is coercive and bounded below on $\mathbf{M}_{a}^{-}.$ Furthermore, there
holds
\begin{equation*}
J_{a}(u)>\frac{p-2}{4p}\left( \frac{S_{p}^{p}}{f_{\max }}\right) ^{2/(p-2)}%
\text{ for all }u\in \mathbf{M}_{a}^{-}.
\end{equation*}
\end{lemma}

\begin{lemma}
\label{g6}Suppose that $N=1,2,3$ and condition $(D1)$ holds$.$ Then for each
$a>0$ and $u\in H^{1}(\mathbb{R}^{N})\backslash \{0\}$ satisfying
\begin{equation*}
\int_{\mathbb{R}^{N}}f(x)|u|^{p}dx>\frac{p}{4-p}\left( \frac{2a(4-p)}{p-2}%
\right) ^{\left( p-2\right) /2}\left\Vert u\right\Vert _{H^{1}}^{p},
\end{equation*}%
then there exist two numbers $t_{a}^{+}$ and $t_{a}^{-}$ satisfying
\begin{equation*}
T_{f}(u)<t_{a}^{-}<\sqrt{D(p)}\left( \frac{2}{4-p}\right)
^{1/(p-2)}T_{f}(u)<\left( \frac{2}{4-p}\right) ^{1/(p-2)}T_{f}(u)<t_{a}^{+}
\end{equation*}%
such that $t_{a}^{\pm }u\in \mathbf{M}_{a}^{\pm }$ and $J_{a}(t_{a}^{-}u)=%
\sup_{0\leq t\leq t_{a}^{+}}J_{a}(tu),$ and
\begin{equation*}
J_{a}(t_{a}^{+}u)=\inf_{t\geq t_{a}^{-}}J_{a}(tu)=\inf_{t\geq 0}J_{a}(tu)<0,
\end{equation*}%
where $T_{f}(u)$ is defined as $(\ref{2-0}).$
\end{lemma}

\begin{proof}
Let
\begin{equation*}
m(t)=t^{-2}\left\Vert u\right\Vert _{H^{1}}^{2}-t^{p-4}\int_{\mathbb{R}%
^{N}}f(x)|u|^{p}dx\text{ for }t>0.
\end{equation*}%
Clearly, $tu\in \mathbf{M}_{a}$ if and only if $m(t)+a\left( \int_{\mathbb{R}%
^{N}}|\nabla u|^{2}dx\right) ^{2}=0.$ A straightforward evaluation gives
\begin{equation*}
m(T_{f}(u))=0,\ \lim_{t\rightarrow 0^{+}}m(t)=\infty \text{ and }%
\lim_{t\rightarrow \infty }m(t)=0.
\end{equation*}%
Since $2<p<4$ and
\begin{equation*}
m^{\prime }\left( t\right) =t^{-3}\left( -2\left\Vert u\right\Vert
_{H^{1}}^{2}+(4-p)t^{p-2}\int_{\mathbb{R}^{N}}f(x)|u|^{p}dx\right) ,
\end{equation*}%
we find that $m(t)$ is decreasing when $0<t<\left( \frac{2}{4-p}\right)
^{1/(p-2)}T_{f}(u)$ and is increasing when $t>\left( \frac{2}{4-p}\right)
^{1/(p-2)}T_{f}(u).$ This indicates that
\begin{equation}
\inf_{t>0}m(t)=m\left( \left( \frac{2}{4-p}\right) ^{1/(p-2)}T_{f}(u)\right)
.  \label{2-4}
\end{equation}%
For each $a>0$ and $u\in H^{1}(\mathbb{R}^{N})\backslash \left\{ 0\right\} $
satisfying
\begin{equation*}
\int_{\mathbb{R}^{N}}f(x)|u|^{p}dx>\frac{p}{4-p}\left( \frac{2a(4-p)}{p-2}%
\right) ^{(p-2)/2}\left\Vert u\right\Vert _{H^{1}}^{p},
\end{equation*}%
we can conclude that
\begin{equation*}
m\left( \left( \frac{2}{4-p}\right) ^{1/(p-2)}T_{f}(u)\right) <-2a\left(
\frac{p}{2}\right) ^{2/(p-2)}\left\Vert u\right\Vert
_{H^{1}}^{4}<-a\left\Vert u\right\Vert _{D^{1,2}}^{4},
\end{equation*}%
where we have used the fact of $\left( \frac{p}{2}\right) ^{2/(p-2)}>1$.
Moreover, by Remark $\ref{r-1}$ we have
\begin{equation}
T_{f}(u)<\sqrt{D(p)}\left( \frac{2}{4-p}\right) ^{1/(p-2)}T_{f}(u)<\left(
\frac{2}{4-p}\right) ^{1/(p-2)}T_{f}(u),  \label{2-3}
\end{equation}%
and a direct calculation shows that
\begin{equation}
\frac{\left( \frac{2}{4-p}\right) D(p)^{(p-2)/2}-1}{D(p)\left( \frac{2}{4-p}%
\right) ^{2/(p-2)}}>\frac{p-2}{2(4-p)}\left( \frac{4-p}{p}\right) ^{2/(p-2)}.
\label{2-1}
\end{equation}%
It follows from $(\ref{2-4})-(\ref{2-1})$ that%
\begin{equation*}
m\left( \sqrt{D(p)}\left( \frac{2}{4-p}\right) ^{1/(p-2)}T_{f}(u)\right)
<-a\left\Vert u\right\Vert _{D^{1,2}}^{4}.
\end{equation*}%
Thus, there exist two numbers $t_{a}^{+},t_{a}^{-}>0$ which satisfy
\begin{equation*}
T_{f}(u)<t_{a}^{-}<\sqrt{D(p)}\left( \frac{2}{4-p}\right)
^{1/(p-2)}T_{f}(u)<\left( \frac{2}{4-p}\right) ^{1/(p-2)}T_{f}(u)<t_{a}^{+}
\end{equation*}%
such that
\begin{equation*}
m(t_{a}^{\pm })+a\left\Vert u\right\Vert _{D^{1,2}}^{4}=0,
\end{equation*}%
leading to $t_{a}^{\pm }u\in \mathbf{M}_{a}.$ By calculating the second
order derivatives, we find
\begin{eqnarray*}
h_{a,t_{a}^{-}u}^{\prime \prime }(1) &=&-2\left\Vert t_{a}^{-}u\right\Vert
_{H^{1}}^{2}+(4-p)\int_{\mathbb{R}^{N}}f(x)|t_{a}^{-}u|^{p}dx \\
&=&(t_{a}^{-})^{5}m^{\prime }(t_{a}^{-})<0,
\end{eqnarray*}%
and
\begin{eqnarray*}
h_{a,t_{a}^{+}u}^{\prime \prime }(1) &=&-2\left\Vert t_{a}^{+}u\right\Vert
_{H^{1}}^{2}+(4-p)\int_{\mathbb{R}^{N}}f(x)|t_{a}^{+}u|^{p}dx \\
&=&(t_{a}^{-})^{5}m^{\prime }(t_{a}^{+})>0.
\end{eqnarray*}%
These imply that $t_{a}^{\pm }u\in \mathbf{M}_{a}^{\pm }.$ Note that
\begin{equation*}
h_{a,u}^{\prime }(t)=t^{3}\left( m(t)+a\left( \int_{\mathbb{R}^{N}}|\nabla
u|^{2}dx\right) ^{2}\right) .
\end{equation*}%
Then one can see that $h_{a,u}^{\prime }(t)>0$ for all $t\in
(0,t_{a}^{-})\cup (t_{a}^{+},\infty )$ and $h_{a,u}^{\prime }(t)<0$ for all $%
t\in (t_{a}^{-},t_{a}^{+})$. It leads to
\begin{equation*}
J_{a}(t_{a}^{-}u)=\sup_{0\leq t\leq t_{a}^{+}}J_{a}(tu)\text{ and }%
J_{a}(t_{a}^{+}u)=\inf_{t\geq t_{a}^{-}}J_{a}(tu),
\end{equation*}%
and so $J_{a}\left( t_{a}^{+}u\right) <J_{a}(t_{a}^{-}u).$ Similar to the
argument of Lemma \ref{g4}$(i),$ we have
\begin{equation*}
J_{a}(t_{a}^{+}u)=\inf_{t\geq 0}J_{a}(tu)<0.
\end{equation*}%
This completes the proof.
\end{proof}

\begin{lemma}
\label{g15}Suppose that $N\geq 4$ and condition $(D1)$ holds. Then for each $%
0<a<\overline{a}_{\ast },$ there exist two numbers $t_{a}^{+}$ and $%
t_{a}^{-} $ satisfying
\begin{equation*}
T_{f}(u)<t_{a}^{-}<\sqrt{D(p)}\left( \frac{2}{4-p}\right)
^{1/(p-2)}T_{f}(u)<\left( \frac{2}{4-p}\right) ^{1/(p-2)}T_{f}(u)<t_{a}^{+}
\end{equation*}%
such that $t_{a}^{\pm }u\in \mathbf{M}_{a}^{\pm }$ and $J_{a}(t_{a}^{-}u)=%
\sup_{0\leq t\leq t_{a}^{+}}J_{a}(tu)$ and $J_{a}(t_{a}^{+}u)=\inf_{t\geq
t_{a}^{-}}J_{a}(tu)=\inf_{t\geq 0}J_{a}(tu)<0.$
\end{lemma}

\begin{proof}
The proof is analogous to those of Lemmas \ref{g4} and \ref{g6}, so we omit
it here.
\end{proof}

Now, we follow a part of idea in \cite{SWF1}, for any $u\in \mathbf{M}_{a}$
with $J_{a}(u)<\frac{D(p)(p-2)}{2p}\left( \frac{2S_{p}^{p}}{f_{\infty }(4-p)}%
\right) ^{2/(p-2)},$ deduce that
\begin{eqnarray*}
\frac{D\left( p\right) (p-2)}{2p}\left( \frac{2S_{p}^{p}}{f_{\infty }\left(
4-p\right) }\right) ^{2/(p-2)} &>&J_{a}(u) \\
&\geq &\frac{p-2}{2p}\left\Vert u\right\Vert _{H^{1}}^{2}-\frac{a(4-p)}{4p}%
\left\Vert u\right\Vert _{H^{1}}^{4},
\end{eqnarray*}%
which implies that for $0<a<\frac{p-2}{2(4-p)}\left( \frac{4-p}{p}\right)
^{2/(p-2)}\Lambda _{0},$ there exist two positive numbers $D_{1}$ and $D_{2}$
satisfying%
\begin{equation*}
\sqrt{D(p)}\left( \frac{2S_{p}^{p}}{f_{\infty }(4-p)}\right)
^{1/(p-2)}<D_{1}<\left( \frac{2S_{p}^{p}}{f_{\max }(4-p)}\right) ^{1/(p-2)}<%
\sqrt{2}\left( \frac{2S_{p}^{p}}{f_{\infty }(4-p)}\right) ^{1/(p-2)}<D_{2}
\end{equation*}%
such that%
\begin{equation*}
\left\Vert u\right\Vert _{H^{1}}<D_{1}\text{ or }\left\Vert u\right\Vert
_{H^{1}}>D_{2}.
\end{equation*}%
Thus, we obtain that
\begin{eqnarray}
&&\mathbf{M}_{a}\left( \frac{D(p)(p-2)}{2p}\left( \frac{2S_{p}^{p}}{%
f_{\infty }(4-p)}\right) ^{\frac{2}{p-2}}\right)  \notag \\
&=&\left\{ u\in \mathbf{M}_{a}:J_{a}(u)<\frac{D(p)(p-2)}{2p}\left( \frac{%
2S_{p}^{p}}{f_{\infty }(4-p)}\right) ^{\frac{2}{p-2}}\right\}  \notag \\
&=&\mathbf{M}_{a}^{(1)}\cup \mathbf{M}_{a}^{(2)},  \label{4-4}
\end{eqnarray}%
where
\begin{equation*}
\mathbf{M}_{a}^{(1)}:=\left\{ u\in \mathbf{M}_{a}\left( \frac{D(p)(p-2)}{2p}%
\left( \frac{2S_{p}^{p}}{f_{\infty }(4-p)}\right) ^{2/(p-2)}\right)
:\left\Vert u\right\Vert _{H^{1}}<D_{1}\right\}
\end{equation*}%
and
\begin{equation*}
\mathbf{M}_{a}^{(2)}:=\left\{ u\in \mathbf{M}_{a}\left( \frac{D(p)(p-2)}{2p}%
\left( \frac{2S_{p}^{p}}{f_{\infty }(4-p)}\right) ^{2/(p-2)}\right)
:\left\Vert u\right\Vert _{H^{1}}>D_{2}\right\} .
\end{equation*}%
We further have
\begin{equation}
\left\Vert u\right\Vert _{H^{1}}<D_{1}<\left( \frac{2S_{p}^{p}}{f_{\max
}(4-p)}\right) ^{\frac{1}{p-2}}\text{ for all }u\in \mathbf{M}_{a}^{(1)}
\label{4-1}
\end{equation}%
and%
\begin{equation}
\left\Vert u\right\Vert _{H^{1}}>D_{2}>\sqrt{2}\left( \frac{2S_{p}^{p}}{%
f_{\infty }(4-p)}\right) ^{\frac{1}{p-2}}\text{ for all }u\in \mathbf{M}%
_{a}^{(2)}.  \label{4-2}
\end{equation}%
Using the Sobolev inequality, $(\ref{2-2})$ and $(\ref{4-1})$ gives
\begin{equation*}
h_{a,u}^{\prime \prime }(1)\leq -2\left\Vert u\right\Vert
_{H^{1}}^{2}+(4-p)S_{p}^{-p}f_{\max }\left\Vert u\right\Vert _{H^{1}}^{p}<0%
\text{ for all }u\in \mathbf{M}_{a}^{(1)}.
\end{equation*}%
By $(\ref{4-2}),$ we derive that
\begin{eqnarray*}
\frac{1}{4}\left\Vert u\right\Vert _{H^{1}}^{2}-\frac{(4-p)}{4p}\int_{%
\mathbb{R}^{N}}f(x)|u|^{p}dx &=&J_{a}(u)<\frac{D(p)(p-2)}{2p}\left( \frac{%
2S_{p}^{p}}{f_{\infty }(4-p)}\right) ^{2/(p-2)} \\
&<&\frac{p-2}{2p}\left( \frac{2S_{p}^{p}}{f_{\infty }(4-p)}\right) ^{2/(p-2)}
\\
&<&\frac{p-2}{4p}\left\Vert u\right\Vert _{H^{1}}^{2}\text{ for all }u\in
\mathbf{M}_{a}^{(2)},
\end{eqnarray*}%
which implies that
\begin{equation*}
2\left\Vert u\right\Vert _{H^{1}}^{2}<(4-p)\int_{\mathbb{R}^{N}}f(x)|u|^{p}dx%
\text{ for all }u\in \mathbf{M}_{a}^{(2)}.
\end{equation*}%
Using the above inequality, together with $(\ref{2-2})$, yields
\begin{equation*}
h_{a,u}^{\prime \prime }\left( 1\right) =-2\left\Vert u\right\Vert
_{H^{1}}^{2}+\left( 4-p\right) \int_{\mathbb{R}^{N}}f(x)\left\vert
u\right\vert ^{p}dx>0\text{ for all }u\in \mathbf{M}_{a}^{(2)}.
\end{equation*}%
Hence, we have the following result.

\begin{lemma}
\label{g7}For $N\geq 1$ and $0<a<\frac{p-2}{2(4-p)}\left( \frac{4-p}{p}%
\right) ^{2/(p-2)}\Lambda _{0},$ we have $\mathbf{M}_{a}^{(1)}\subset
\mathbf{M}_{a}^{-}$ and $\mathbf{M}_{a}^{(2)}\subset \mathbf{M}_{a}^{+}$
both are $C^{1}$ sub-manifolds. Furthermore, each local minimizer of the
functional $J_{a}$ in the sub-manifolds $\mathbf{M}_{a}^{(1)}$ and $\mathbf{M%
}_{a}^{(2)}$ is a critical point of $J_{a}$ in $H^{1}(\mathbb{R}^{N}).$
\end{lemma}

At the end of this section, similar to \cite[Lemma 3.4]{LY}, we introduce a
global compactness result, which is applicable to Kirchhoff type equations.

\begin{proposition}
\label{l0}Suppose that $N\geq 1$ and conditions $(D1)-(D2)$ hold. Let $%
\{u_{n}\}$ be a bounded $(PS)_{\beta }$--sequence in $H^{1}(\mathbb{R}^{N})$
for $J_{a}.$ There exist $u_{0}\in H^{1}(\mathbb{R}^{N})$ and $A\in
\mathbb{R}
$ such that $I_{a}^{\prime }(u)=0,$ where%
\begin{equation*}
I_{a}(u)=\frac{(1+aA^{2})}{2}\int_{\mathbb{R}^{N}}|\nabla u|^{2}dx+\frac{1}{2%
}\int_{\mathbb{R}^{N}}u^{2}dx-\frac{1}{p}\int_{\mathbb{R}^{N}}f(x)|u|^{p}dx,
\end{equation*}%
and either\newline
$(i)$ $u_{n}\rightarrow u_{0}$ in $H^{1}(\mathbb{R}^{N}),$or\newline
$(ii)$ there exists a number $m\in \mathbb{N}$ and $\{x_{n}^{i}\}_{n=1}^{%
\infty }\subset \mathbb{R}^{N}$ with $|x_{n}^{i}|\rightarrow \infty $ as $%
n\rightarrow \infty $ for each $1\leq i\leq m,$ nontrivial solutions $%
w^{1},w^{2},...,w^{m}\in H^{1}(\mathbb{R}^{N})$ of the following equation%
\begin{equation*}
-(aA^{2}+1)\Delta u+u=f_{\infty }|u|^{p-2}u
\end{equation*}%
such that%
\begin{equation*}
\beta +\frac{aA^{4}}{4}=I_{a}(u_{0})+\underset{i=1}{\overset{m}{\sum }}%
I_{a}^{\infty }(w^{i}),
\end{equation*}%
and%
\begin{equation*}
u_{n}=u_{0}+\underset{i=1}{\overset{m}{\sum }}w^{i}(\cdot -x_{n}^{i})+o(1)%
\text{ strongly in }H^{1}(\mathbb{R}^{N}),
\end{equation*}%
and%
\begin{equation*}
A^{2}=\int_{\mathbb{R}^{N}}|\nabla u_{0}|^{2}dx+\underset{i=1}{\overset{m}{%
\sum }}\int_{\mathbb{R}^{N}}|\nabla w^{i}|^{2}dx.
\end{equation*}
\end{proposition}

\section{Proofs of Theorems \protect\ref{t0-1}, \protect\ref{t0-2} and
\protect\ref{t0-3}}

\textbf{At the beginning of this section, we prove Theorem \ref{t0-1}: }$(i)$
For $0<k<\frac{4-N}{4-p}$ and $u\in H^{1}(\mathbb{R}^{N})\backslash \{0\},$
let $v_{t}(x)=t^{k}u(t^{-1}x).$ Then we have\newline
$(i-A)$ $\int_{\mathbb{R}^{N}}|\nabla v_{t}(x)|^{2}dx=t^{2k-2}\int_{\mathbb{R%
}^{N}}|\nabla u(t^{-1}x)|^{2}dx=t^{2k-2+N}\int_{\mathbb{R}^{N}}|\nabla
u(y)|^{2}dy;$\newline
$(i-B)$ $\int_{\mathbb{R}^{N}}|v_{t}(x)|^{2}dx=t^{2k}\int_{\mathbb{R}%
^{N}}|u(tx)|^{2}dx=t^{2k+N}\int_{\mathbb{R}^{N}}|u(y)|^{2}dy;$\newline
$(i-C)$ $\int_{\mathbb{R}^{N}}|v_{t}(x)|^{p}dx=t^{pk}\int_{\mathbb{R}%
^{N}}|u(tx)|^{p}dx=t^{pk+N}\int_{\mathbb{R}^{N}}|u(y)|^{p}dy.$\newline
Using the above conclusions gives%
\begin{eqnarray*}
J_{a}(v_{t}(x)) &=&\frac{a}{4}\left( \int_{\mathbb{R}^{N}}|\nabla
v_{t}(x)|^{2}dx\right) ^{2}+\frac{1}{2}\left( \int_{\mathbb{R}^{N}}|\nabla
v_{t}(x)|^{2}dx+\int_{\mathbb{R}^{N}}|v_{t}(x)|^{2}dx\right) \\
&&-\frac{1}{p}\int_{\mathbb{R}^{N}}f(x)|v_{t}(x)|^{p}dx \\
&\leq &\frac{a}{4}t^{4k-4+2N}\left\Vert u\right\Vert _{D^{1,2}}^{4}+\frac{1}{%
2}t^{2k-2+N}\left\Vert u\right\Vert _{D^{1,2}}^{2}+\frac{1}{2}t^{2k+N}\int_{%
\mathbb{R}^{N}}|u|^{2}dx \\
&&-\frac{f_{\min }}{p}t^{pk+N}\int_{\mathbb{R}^{N}}|u|^{p}dx\rightarrow
-\infty \text{ as }t\rightarrow \infty ,
\end{eqnarray*}%
since $pk+N>4k-4+2N.$ This implies that $J_{a}$ is not bounded below on $%
H^{1}(\mathbb{R}^{N})\ $for $N=1,2,3.$\newline
$(ii)$ It follows from Corollary \ref{g17} that for each $a>\overline{a}%
_{\ast },$ the energy functional $J_{a}$ is bounded below on $H^{1}(\mathbb{R%
}^{4})$ and $\inf_{u\in H^{1}(\mathbb{R}^{4})\backslash \{0\}}J_{a}(u)>0.$
Next, we claim that for each $0<a<\underline{a}_{\ast },$ $J_{a}$ is not
bounded below on $H^{1}(\mathbb{R}^{4}),$ i.e., $\inf_{u\in H^{1}(\mathbb{R}%
^{4})\backslash \{0\}}J_{a}(u)=-\infty .$ Let%
\begin{equation*}
I(u)=\frac{a}{4}\left\Vert u\right\Vert _{D^{1,2}}^{4}+\frac{1}{2}\int_{%
\mathbb{R}^{4}}u^{2}dx-\frac{f_{\min }}{p}\int_{\mathbb{R}^{4}}|u|^{p}dx%
\text{ for }u\in H^{1}(\mathbb{R}^{N}).
\end{equation*}%
Then for $s>0,$ we have%
\begin{equation*}
I(su)=s^{4}\left( \frac{a}{4}\left\Vert u\right\Vert _{D^{1,2}}^{4}+%
\overline{g}(s)\right) ,
\end{equation*}%
where
\begin{equation*}
\overline{g}(s)=\frac{s^{-2}}{2}\int_{\mathbb{R}^{4}}u^{2}dx-\frac{f_{\min
}s^{p-4}}{p}\int_{\mathbb{R}^{4}}|u|^{p}dx.
\end{equation*}%
Clearly, $I(su)=0$ if and only if $\overline{g}(s)+\frac{a}{4}\left\Vert
u\right\Vert _{D^{1,2}}^{4}=0.$ It is not difficult to observe that $%
\overline{g}\left( s_{a}\right) =0,\ \lim_{s\rightarrow 0^{+}}\overline{g}%
(s)=\infty \ $and$\ \lim_{s\rightarrow \infty }\overline{g}(s)=0,$ where
\begin{equation*}
s_{a}=\left( \frac{p\int_{\mathbb{R}^{4}}u^{2}dx}{2f_{\min }\int_{\mathbb{R}%
^{4}}\left\vert u\right\vert ^{p}dx}\right) ^{1/(p-2)}>0.
\end{equation*}%
Considering the derivative of $\overline{g}(s)$, we find
\begin{equation*}
\overline{g}^{\prime }(s)=s^{-3}\left[ \frac{s^{p-2}f_{\min }(4-p)}{p}\int_{%
\mathbb{R}^{4}}|u|^{p}dx-\int_{\mathbb{R}^{4}}u^{2}dx\right] ,
\end{equation*}%
which implies that $\overline{g}(s)$ is decreasing when $0<t<\left( \frac{%
p\int_{\mathbb{R}^{4}}u^{2}dx}{(4-p)f_{\min }\int_{\mathbb{R}^{4}}|u|^{p}dx}%
\right) ^{1/(p-2)}$ and is increasing when $t>\left( \frac{p\int_{\mathbb{R}%
^{4}}u^{2}dx}{(4-p)f_{\min }\int_{\mathbb{R}^{4}}|u|^{p}dx}\right)
^{1/(p-2)},$ and so%
\begin{eqnarray}
\inf_{s>0}\overline{g}(t) &=&\overline{g}\left( \left( \frac{p\int_{\mathbb{R%
}^{4}}u^{2}dx}{(4-p)f_{\min }\int_{\mathbb{R}^{4}}|u|^{p}dx}\right)
^{1/(p-2)}\right)  \notag \\
&=&-\frac{p-2}{2}\left( \frac{f_{\min }\int_{\mathbb{R}^{4}}|u|^{p}dx}{p}%
\right) ^{2/(p-2)}\left( \frac{4-p}{\int_{\mathbb{R}^{4}}u^{2}dx}\right)
^{\left( 4-p\right) /(p-2)}<0.  \label{10-0}
\end{eqnarray}%
Since $0<a<\underline{a}_{\ast },$ there exists $u\in H^{1}(\mathbb{R}%
^{4})\backslash \{0\}$ such that%
\begin{equation*}
a<2(p-2)\left( \frac{4-p}{\int_{\mathbb{R}^{4}}u^{2}dx}\right) ^{\left(
4-p\right) /(p-2)}\left( \frac{f_{\min }\int_{\mathbb{R}^{4}}|u|^{p}dx}{p}%
\right) ^{2/(p-2)}\left\Vert u\right\Vert _{D^{1,2}}^{-4}\leq \underline{a}%
_{\ast }.
\end{equation*}%
Using the above inequality, together with $(\ref{10-0}),$ leads to $%
\inf_{s>0}\overline{g}(s)<-\frac{a}{4}\left\Vert u\right\Vert
_{D^{1,2}}^{4}. $ Set%
\begin{equation*}
s_{0}(u)=\left( \frac{p\int_{\mathbb{R}^{4}}u^{2}dx}{(4-p)f_{\min }\int_{%
\mathbb{R}^{4}}|u|^{p}dx}\right) ^{1/(p-2)}.
\end{equation*}%
Then we obtain
\begin{equation}
I(s_{0}(u)u)=s_{0}^{4}(u)\left[ \overline{g}(s_{0}(u))+\frac{a}{4}\left\Vert
u\right\Vert _{D^{1,2}}^{4}\right] <0.  \label{10-11}
\end{equation}%
Let $u_{0}=s_{0}\left( u\right) u$ and $v_{t}(x)=u_{0}(t^{-1}x).$ Then we
have\newline
$(ii-A)$ $\int_{\mathbb{R}^{4}}|\nabla v_{t}(x)|^{2}dx=t^{2}\int_{\mathbb{R}%
^{4}}|\nabla u_{0}(x)|^{2}dx;$\newline
$(ii-B)$ $\int_{\mathbb{R}^{4}}|v_{t}(x)|^{2}dx=t^{4}\int_{\mathbb{R}%
^{4}}|u_{0}(x)|^{2}dx;$\newline
$(ii-C)$ $\int_{\mathbb{R}^{4}}|v_{t}(x)|^{p}dx=t^{4}\int_{\mathbb{R}%
^{4}}|u_{0}(x)|^{p}dx.$\newline
Combining the above conclusions with $(\ref{10-11})$ gives%
\begin{eqnarray*}
J_{a}(v_{t}(x)) &\leq &t^{4}\left( \frac{a}{4}\left\Vert u_{0}\right\Vert
_{D^{1,2}}^{4}+\frac{1}{2}\int_{\mathbb{R}^{N}}u_{0}^{2}dx-\frac{f_{\min }}{p%
}\int_{\mathbb{R}^{4}}|u_{0}|^{p}dx\right) +\frac{t^{2}}{2}\left\Vert
u_{0}\right\Vert _{D^{1,2}}^{2} \\
&=&t^{4}I(u_{0})+\frac{t^{2}}{2}\left\Vert u_{0}\right\Vert _{D^{1,2}}^{2} \\
&\rightarrow &-\infty \text{ as }t\rightarrow \infty ,
\end{eqnarray*}%
which implies that for each $0<a<\underline{a}_{\ast },$ $J_{a}$ is not
bounded below on $H^{1}(\mathbb{R}^{4}),$ i.e., $\inf_{u\in H^{1}(\mathbb{R}%
^{4})\backslash \{0\}}J_{a}(u)=-\infty .$\newline
$\left( iii\right) $ By Lemmas \ref{g4} and \ref{g11}, we can arrive at the
conclusion.

\textbf{Next, we are ready to prove Theorem \ref{t0-2}: }For $u\in H^{1}(%
\mathbb{R}^{N})\backslash \left\{ 0\right\} ,$ we know that $tu\in \mathbf{M}%
_{a}^{0}$ if and only if $h_{a,tu}^{\prime }\left( 1\right)
=h_{a,tu}^{\prime \prime }\left( 1\right) =0,$ i.e., the following system of
equations is satisfied:%
\begin{equation}
\left\{
\begin{array}{c}
t\left\Vert u\right\Vert _{H^{1}}^{2}+at^{3}\left( \int_{\mathbb{R}%
^{N}}\left\vert \nabla u\right\vert ^{2}dx\right) ^{2}-t^{p-1}\int_{\mathbb{R%
}^{N}}f(x)\left\vert u\right\vert ^{p}dx=0, \\
\left\Vert u\right\Vert _{H^{1}}^{2}+3at^{2}\left( \int_{\mathbb{R}%
^{N}}\left\vert \nabla u\right\vert ^{2}dx\right) ^{2}-(p-1)t^{p-2}\int_{%
\mathbb{R}^{N}}f(x)\left\vert u\right\vert ^{p}dx=0.%
\end{array}%
\right.  \label{15-2}
\end{equation}%
By solving the system $(\ref{15-2})$ with respect to the variables $t$ and $%
a,$ we have%
\begin{equation*}
t(u)=\left( \frac{2\left\Vert u\right\Vert _{H^{1}}^{2}}{(4-p)\int_{\mathbb{R%
}^{N}}f(x)\left\vert u\right\vert ^{p}dx}\right) ^{1/(p-2)}
\end{equation*}%
and%
\begin{equation*}
a(u)=\frac{p-2}{4-p}\left( \frac{4-p}{2}\right) ^{2/(p-2)}\overline{A}%
_{f}(u),
\end{equation*}%
where $\overline{A}_{f}(u)$ is as $(\ref{15-1}).$ We conclude that $a(u)$ is
the unique parameter $a>0$ for which the fibering map $h_{a,u}$ has a
critical point with second derivative zero at $t(u)$. Moreover, if $a>a(u)$,
then $h_{a,u}$ is increasing on $(0,\infty )$ and has no critical point.
Note that $\sup_{u\in H^{1}(\mathbb{R}^{N})\backslash \left\{ 0\right\}
}a(u)=\frac{p^{2/\left( p-2\right) }}{2^{p/\left( p-2\right) }}\overline{a}%
_{\ast }$ by (\ref{15-3}). Hence, the energy functional $J_{a}$ has no any
nontrivial critical points for $a>\frac{p^{2/\left( p-2\right) }}{%
2^{p/\left( p-2\right) }}\overline{a}_{\ast }.$ Consequently, we complete
the proof.

To prove that Theorem \ref{t0-3}\textbf{, }we need the following result.

\begin{lemma}
\label{g12}Suppose that $N\geq 5$ and condition $(D1)$ holds. Let $0<a<%
\overline{a}_{\ast }$ Then every minimizing sequence for $J_{a}$ in $H^{1}(%
\mathbb{R}^{N})$ is bounded.
\end{lemma}

\begin{proof}
Using Lemma \ref{g11} gives%
\begin{equation*}
-\infty <\inf_{u\in H^{1}(\mathbb{R}^{N})\backslash \{0\}}J_{a}(u)<0.
\end{equation*}%
Let $\left\{ u_{n}\right\} $ be a minimizing sequence for $J_{a}$ in $H^{1}(%
\mathbb{R}^{N}).$ Then by Lemma \ref{g11} and the fact of $\inf_{u\in H^{1}(%
\mathbb{R}^{N})\backslash \{0\}}J_{a}(u)<0$, there exists a number $\widehat{%
R}_{a}>0$ such that
\begin{equation*}
\left\Vert u_{n}\right\Vert _{H^{1}}\leq \widehat{R}_{a}\text{ for }n\text{
large enough.}
\end{equation*}%
Consequently, we complete the proof.
\end{proof}

\textbf{At the end of this section, we begin to prove Theorem \ref{t0-3}: }$%
(i)$ By Lemma \ref{g12} and the Ekeland variational principle, for each $0<a<%
\overline{a}_{\ast }$ there exists a bounded minimizing sequence $%
\{u_{n}\}\subset H^{1}(\mathbb{R}^{N})$ such that
\begin{equation*}
J_{a}^{\infty }(u_{n})=\inf_{u\in H^{1}(\mathbb{R}^{N})\backslash
\{0\}}J_{a}^{\infty }(u)+o(1)\text{ and }(J_{a}^{\infty })^{\prime
}(u_{n})=o(1)\text{ in }H^{-1}(\mathbb{R}^{N}).
\end{equation*}%
Similar to the argument of Theorem \ref{g10} in Appendix, we can prove that
the compactness for the sequence $\{u_{n}\}$ holds. Then for each $\theta >0$
there exist a number $R=R(\theta )>0$ and a sequence $\{z_{n}\}\subset
\mathbb{R}^{N}$ such that%
\begin{equation}
\int_{\lbrack B_{R}(z_{n})]^{c}}(|\nabla
u_{n}(x)|^{2}+u_{n}^{2}(x))dx<\theta \text{ uniformly for }n\geq 1.
\label{10-2}
\end{equation}%
Define a new sequence of functions $v_{n}:=u_{n}(\cdot +z_{n})\in H^{1}(%
\mathbb{R}^{N}).$ Clearly, $\left\langle (J_{a}^{\infty })^{\prime
}(v_{n}),v_{n}\right\rangle =o(1)$ and $J_{a}^{\infty }(v_{n})=\inf_{u\in
H^{1}(\mathbb{R}^{N})\backslash \{0\}}J_{a}^{\infty }(u)+o(1).$ By virtue of
$(\ref{10-2})$, for each $\theta >0$ there exists a number $R=R(\theta )>0$
such that%
\begin{equation}
\int_{\lbrack B_{R}(0)]^{c}}(|\nabla v_{n}(x)|^{2}+v_{n}^{2}(x))dx<\theta
\text{ uniformly for }n\geq 1.  \label{10-3}
\end{equation}%
Since $\{v_{n}\}$ is bounded in $H^{1}(\mathbb{R}^{N}),$ one can assume that
there exist a subsequence $\{v_{n}\}$ and $v_{a}^{+}\in H^{1}(\mathbb{R}%
^{N}) $ such that%
\begin{eqnarray}
v_{n} &\rightharpoonup &v_{a}^{+}\text{ weakly in }H^{1}(\mathbb{R}^{N}),
\label{10-4} \\
v_{n} &\rightarrow &v_{a}^{+}\text{ strongly in }L_{loc}^{r}(\mathbb{R}^{N})%
\text{ for }2\leq r<2^{\ast },  \label{10-5} \\
v_{n} &\rightarrow &v_{a}^{+}\text{ a.e. in }\mathbb{R}^{N}.  \notag
\end{eqnarray}%
By $(\ref{10-3})-(\ref{10-5})$ and Fatou's Lemma, for any $\theta >0$ and
sufficiently large $n$, there exists a number $R>0$ such that%
\begin{eqnarray*}
&&\int_{\mathbb{R}^{3}}|v_{n}-v_{a}^{+}|^{p}dx \\
&\leq
&\int_{B_{R}(0)}|v_{n}-v_{a}^{+}|^{p}dx+%
\int_{[B_{R}(0)]^{c}}|v_{n}-v_{a}^{+}|^{p}dx \\
&\leq &\theta +S_{p}^{-p}\left[ \int_{[B_{R}(0)]^{c}}(|\nabla
v_{n}|^{2}+v_{n}^{2})dx+\int_{[B_{R}(0)]^{c}}(|\nabla
v_{a}^{+}|^{2}+(v_{a}^{+})^{2})dx\right] ^{\frac{p}{2}} \\
&\leq &\theta +S_{p}^{-p}(2\theta )^{\frac{p}{2}},
\end{eqnarray*}%
which implies that for every $p\in (2,2^{\ast }),$%
\begin{equation}
v_{n}\rightarrow v_{a}^{+}\text{ strongly in }L^{p}(\mathbb{R}^{N}).
\label{10-6}
\end{equation}%
Since $\left\langle (J_{a}^{\infty })^{\prime }(v_{n}),v_{n}\right\rangle
=o\left( 1\right) $ and $\widehat{r}<\left\Vert v_{n}\right\Vert _{H^{1}}<%
\widehat{R}_{a},$ using $(\ref{10-6})$ gives
\begin{equation*}
\int_{\mathbb{R}^{N}}f_{\infty }|v_{a}^{+}|^{p}dx\geq \left( \frac{S_{p}^{p}%
}{f_{\infty }}\right) ^{2/(p-2)}>0,
\end{equation*}%
which indicates that $v_{a}^{+}\not\equiv 0.$

Next, we show that $v_{n}\rightarrow v_{a}^{+}$ strongly in $H^{1}(\mathbb{R}%
^{N}).$ Suppose on the contrary. Then we have
\begin{equation}
\left\Vert v_{a}^{+}\right\Vert _{H^{1}}<\liminf_{n\rightarrow \infty }\Vert
v_{n}\Vert _{H^{1}}.  \label{10-7}
\end{equation}%
Similar to the argument of Lemma \ref{g6}, there exists a unique $t_{a}>0$
such that
\begin{equation}
(h_{a,v_{a}^{+}}^{\infty })^{\prime }(t_{a})=0,  \label{10-8}
\end{equation}%
where $h_{a,u}^{\infty }(t)=h_{a,u}\left( t\right) $ with $f(x)\equiv
f_{\infty }$. Since $\left\langle (J_{a}^{\infty })^{\prime
}(v_{n}),v_{n}\right\rangle =o(1)$, it follows from $(\ref{10-6})-(\ref{10-7}%
)$ that
\begin{equation}
(h_{a,v_{a}^{+}}^{\infty })^{\prime }(1)<0.  \label{10-9}
\end{equation}%
Combining $(\ref{10-8})-(\ref{10-9})$ with the profile of $%
h_{a,v_{a}^{+}}^{\infty }(t)$ gives $t_{a}<1$. By $\left( \ref{10-6}\right)
-\left( \ref{10-7}\right) $ again, we see $(h_{a,v_{n}}^{\infty })^{\prime
}(t_{a})>0$ for sufficiently large $n$. Note that
\begin{equation*}
\left( h_{a,v_{n}}^{\infty }\right) ^{\prime }(1)=o(1),
\end{equation*}%
because of $\left\langle (J_{a}^{\infty })^{\prime
}(v_{n}),v_{n}\right\rangle =o(1)$. Similar to the proof of Lemma \ref{g6},
we obtain
\begin{equation*}
\left( h_{a,v_{n}}^{\infty }\right) ^{\prime }(t)=t^{3}\left( m^{\infty
}(t)+a\left( \int_{\mathbb{R}^{N}}|\nabla v_{n}|^{2}dx\right) ^{2}\right)
\text{ for }t>0,
\end{equation*}%
where
\begin{equation*}
m^{\infty }(t):=t^{-2}\left\Vert v_{n}\right\Vert _{H^{1}}^{2}-t^{p-4}\int_{%
\mathbb{R}^{N}}f_{\infty }|v_{n}|^{p}dx.
\end{equation*}%
One can see that $m^{\infty }(t)$ is decreasing for
\begin{equation*}
0<t<\left( \frac{2\left\Vert v_{n}\right\Vert _{H^{1}}^{2}}{(4-p)\int_{%
\mathbb{R}^{N}}f_{\infty }|v_{n}|^{p}dx}\right) ^{1/(p-2)},
\end{equation*}%
and
\begin{equation*}
\left( \frac{2\left\Vert v_{n}\right\Vert _{H^{1}}^{2}}{(4-p)\int_{\mathbb{R}%
^{N}}f_{\infty }|v_{n}|^{p}dx}\right) ^{1/(p-2)}>1.
\end{equation*}%
This indicates that $\left( h_{a,v_{n}}^{\infty }\right) ^{\prime }(t)>0$
for $0<t<1$, which implies that $h_{a,v_{n}}^{\infty }$ is increasing on $%
(t_{a},1)$ for sufficiently large $n$. Thus, $h_{a,v_{n}}^{\infty
}(t_{a})<h_{a,v_{n}}^{\infty }(1)$ holds for sufficiently large $n$. This
implies that
\begin{equation*}
J_{a}^{\infty }(t_{a}v_{n})<J_{a}^{\infty }(v_{n})\text{ for sufficiently
large }n,
\end{equation*}%
and so, we have
\begin{equation*}
J_{a}^{\infty }(t_{a}v_{a}^{+})<\liminf_{n\rightarrow \infty }J_{a}^{\infty
}(t_{a}v_{n})\leq \liminf_{n\rightarrow \infty }J_{a}^{\infty
}(v_{n})=\inf_{u\in H^{1}(\mathbb{R}^{N})\backslash \{0\}}J_{a}^{\infty }(u),
\end{equation*}%
which is a contradiction. Thus, we obtain that $v_{n}\rightarrow v_{a}^{+}$
strongly in $H^{1}(\mathbb{R}^{N})$ and
\begin{equation*}
J_{a}^{\infty }(v_{n})\rightarrow J_{a}^{\infty }(v_{a}^{+})=\inf_{u\in
H^{1}(\mathbb{R}^{N})\backslash \{0\}}J_{a}^{\infty }(u)\text{ as }%
n\rightarrow \infty .
\end{equation*}%
Hence, $v_{a}^{+}$ is a minimizer for $J_{a}^{\infty }$ on $H^{1}(\mathbb{R}%
^{N}).$ Since
\begin{equation*}
\left\vert v_{a}^{+}\right\vert \in H^{1}(\mathbb{R}^{N})\backslash \left\{
0\right\} \ \text{and}\ J_{a}^{\infty }(\left\vert v_{a}^{+}\right\vert
)=J_{a}^{\infty }(v_{a}^{+})=\inf_{u\in H^{1}(\mathbb{R}^{N})\backslash
\{0\}}J_{a}^{\infty }(u),
\end{equation*}%
one can see that $v_{a}^{+}$ is a positive solution of Eq. $(E_{a})$.

Before proving Theorem \ref{t0-3} $(ii)$, we need the following compactness
lemma which is regarded as a corollary of Proposition \ref{g10}.

\begin{lemma}
\label{lem1}Suppose that $N\geq 5$ and conditions $(D1)-(D2)$ hold. Let $%
\{u_{n}\}$ be a $(PS)_{\beta }$--sequence in $H^{1}(\mathbb{R}^{N})$ for $%
J_{a}$ with $\beta <\inf_{u\in H^{1}(\mathbb{R}^{N})\backslash
\{0\}}J_{a}^{\infty }(u)<0.$ Then there exist a subsequence $\{u_{n}\}$ and
a nonzero $u_{0}$ in $H^{1}(\mathbb{R}^{N})$ such that $u_{n}\rightarrow
u_{0}$ strongly in $H^{1}(\mathbb{R}^{N})$ and $J_{a}(u_{0})=\beta .$
Furthermore, $u_{0}$ is a nonzero solution of Eq. $(E_{a}).$
\end{lemma}

\textbf{The proof of Theorem \ref{t0-3} }$(ii):$ By condition $(D3),$ we have%
\begin{equation*}
\inf_{u\in H^{1}(\mathbb{R}^{N})\backslash \{0\}}J_{a}(u)\leq
J_{a}(v_{a}^{+})<J_{a}^{\infty }(v_{a}^{+})=\inf_{u\in H^{1}(\mathbb{R}%
^{N})\backslash \{0\}}J_{a}^{\infty }(u).
\end{equation*}%
Moreover, by Lemmas \ref{g11}, \ref{g12}$(i)$ and the Ekeland variational
principle, there exists a bounded minimizing sequence $\{u_{n}\}\subset
H^{1}(\mathbb{R}^{N})$ with $\widehat{r}<\left\Vert u_{n}\right\Vert
_{H^{1}}<\widehat{R}_{a}$ satisfying
\begin{equation*}
J_{a}(u_{n})=\inf_{u\in H^{1}(\mathbb{R}^{N})\backslash \{0\}}J_{a}(u)+o(1)%
\text{ and }J_{a}^{\prime }(u_{n})=o\left( 1\right) \text{ in }H^{-1}(%
\mathbb{R}^{N}).
\end{equation*}%
By virtue of Lemma \ref{lem1}, we know that for each $0<a<\overline{a}_{\ast
},$ Eq. $(E_{a})$ has a nontrivial solution $u_{a}^{+}$ such that $%
J_{a}(u_{a}^{+})=\inf_{u\in H^{1}(\mathbb{R}^{N})\backslash
\{0\}}J_{a}(u)<0. $ Since $J_{a}(|u_{a}^{+}|)=J_{a}(u_{a}^{+})=\inf_{u\in
H^{1}(\mathbb{R}^{N})\backslash \{0\}}J_{a}(u)$, we may assume that $%
u_{a}^{+}$ is a positive solution of Eq. $(E_{a}).$ Consequently, we
complete the proof.

\section{Proofs of Theorem \protect\ref{t1}}

In this section, we consider the following limit problem%
\begin{equation}
\left\{
\begin{array}{ll}
-\left( a\int_{\mathbb{R}^{N}}|\nabla u|^{2}dx+1\right) \Delta u+u=f_{\infty
}|u|^{p-2}u & \text{ in }\mathbb{R}^{N}, \\
u\in H^{1}(\mathbb{R}^{N}). &
\end{array}%
\right.  \tag{$E_{a}^{\infty }$}
\end{equation}%
It is clear that solutions of Eq. $(E_{a}^{\infty })$ are critical points of
the energy functional $J_{a}^{\infty }$ defined by
\begin{equation*}
J_{a}^{\infty }(u)=\frac{1}{2}\left\Vert u\right\Vert _{H^{1}}^{2}+\frac{a}{4%
}\left( \int_{\mathbb{R}^{N}}|\nabla u|^{2}dx\right) ^{2}-\frac{1}{p}\int_{%
\mathbb{R}^{N}}f_{\infty }|u|^{p}dx.
\end{equation*}%
Moreover, $u\in \mathbf{M}_{a}^{\infty }$ if and only if $\left\Vert
u\right\Vert _{H^{1}}^{2}+a\left( \int_{\mathbb{R}^{N}}\left\vert \nabla
u\right\vert ^{2}dx\right) ^{2}-\int_{\mathbb{R}^{N}}f_{\infty }|u|^{p}dx=0,$
where $\mathbf{M}_{a}^{\infty }=\mathbf{M}_{a}$ with $f(x)\equiv f_{\infty
}. $ We denote by
\begin{equation*}
\mathbf{M}_{a}^{\infty ,(j)}=\mathbf{M}_{a}^{(j)}\text{ with }f(x)\equiv
f_{\infty }\text{ for }j=1,2.
\end{equation*}%
Since $w_{0}$ is the unique positive solution of Eq. $(E_{0}^{\infty }),$ it
follows from $\left( \ref{1-8}\right) $ that
\begin{equation*}
\int_{\mathbb{R}^{N}}f_{\infty }|w_{0}|^{p}dx=f_{\infty
}S_{p}^{-p}\left\Vert w_{0}\right\Vert _{H^{1}}^{p}>\frac{p}{4-p}\left(
\frac{2a(4-p)}{p-2}\right) ^{(p-2)/2}\left\Vert w_{0}\right\Vert
_{H^{1}}^{p},
\end{equation*}%
for all $0<a<\frac{p-2}{2(4-p)}\left( \frac{4-p}{p}\right) ^{2/(p-2)}\Lambda
_{0}.$ Thus, by Lemma \ref{g6}, there exist two numbers $t_{a}^{\infty
,-},t_{a}^{\infty ,+}>0$ satisfying%
\begin{equation*}
1<t_{a}^{\infty ,-}<\sqrt{D(p)}\left( \frac{2}{4-p}\right)
^{1/(p-2)}<t_{a}^{\infty ,+},
\end{equation*}%
such that $t_{a}^{\infty ,\pm }w_{0}\in \mathbf{M}_{a}^{\infty ,\pm },$
where $\mathbf{M}_{a}^{\infty ,\pm }=\mathbf{M}_{a}^{\pm }$ with $f(x)\equiv
f_{\infty }$. Furthermore, there holds
\begin{equation*}
J_{a}^{\infty }(t_{a}^{\infty ,-}w_{0})=\sup_{0\leq t\leq t_{a}^{\infty
,+}}J_{a}^{\infty }(tw_{0})
\end{equation*}%
and
\begin{equation}
J_{a}^{\infty }(t_{a}^{\infty ,+}w_{0})=\inf_{t\geq t_{a}^{\infty
,-}}J_{a}^{\infty }(tw_{0})=\inf_{t\geq 0}J_{a}^{\infty }(tw_{0})<0.
\label{28}
\end{equation}%
A direct calculation shows that%
\begin{equation}
J_{a}^{\infty }(t_{a}^{\infty ,-}w_{0})<D(p)\frac{p-2}{2p}\left( \frac{%
2S_{p}^{p}}{f_{\infty }(4-p)}\right) ^{2/(p-2)}.  \label{27}
\end{equation}%
It follows from $(\ref{28})-(\ref{27})$ that $t_{a}^{\infty ,-}w_{0}\in
\mathbf{M}_{a}^{\infty ,(1)}$ and $t_{a}^{\infty ,+}w_{0}\in \mathbf{M}%
_{a}^{\infty ,(2)}.$ Namely, $\mathbf{M}_{a}^{\infty ,\left( j\right)
}(j=1,2)$ are nonempty.

Define
\begin{equation*}
\alpha _{a}^{\infty ,-}=\inf_{u\in \mathbf{M}_{a}^{\infty
,(1)}}J_{a}^{\infty }(u)=\inf_{u\in \mathbf{M}_{a}^{\infty ,-}}J_{a}^{\infty
}(u)\text{ for }N\geq 1.
\end{equation*}%
Then we have
\begin{equation*}
\frac{p-2}{4p}\left( \frac{S_{p}^{p}}{f_{\infty }}\right) ^{2/(p-2)}\leq
\alpha _{a}^{\infty ,-}<\frac{D(p)(p-2)}{2p}\left( \frac{2S_{p}^{p}}{%
f_{\infty }(4-p)}\right) ^{2/(p-2)}
\end{equation*}%
by Lemma \ref{g1} and $(\ref{27}).$

\textbf{We are now ready to prove Theorem \ref{t1}: }$\left( i\right) $ By
Lemma \ref{g7} and the Ekeland variational principle, there exists a
sequence $\{u_{n}\}\subset \mathbf{M}_{a}^{\infty ,(1)}$ satisfies
\begin{equation*}
J_{a}^{\infty }(u_{n})=\alpha _{a}^{\infty ,-}+o(1)\text{ and }%
(J_{a}^{\infty })^{\prime }(u_{n})=o(1)\text{ in }H^{-1}(\mathbb{R}^{N}).
\end{equation*}%
Applying Theorem \ref{g10} in Appendix, we obtain that for $0<a<\Lambda ,$
compactness holds for the sequence $\{u_{n}\}.$ Then for each $\theta >0$
there exist a positive constant $R=R\left( \theta \right) $ and a sequence $%
\{z_{n}\}\subset \mathbb{R}^{N}$ such that%
\begin{equation}
\int_{\left[ B_{R}(z_{n})\right] ^{c}}(|\nabla
u_{n}(x)|^{2}+u_{n}^{2}(x))dx<\theta \text{ uniformly for }n\geq 1.
\label{18-4}
\end{equation}%
Define a new sequence of functions $v_{n}:=u_{n}(\cdot +z_{n})\in H^{1}(%
\mathbb{R}^{N}).$ Then we have $\{v_{n}\}\subset \mathbf{M}_{a}^{\infty
,(1)} $ and $J_{a}^{\infty }(v_{n})=\alpha _{a}^{\infty ,-}+o(1).$ By virtue
of $(\ref{18-4})$, for each $\theta >0$ there exists a constant $R=R(\theta
)>0$ such that%
\begin{equation*}
\int_{\left[ B_{R}(z_{n})\right] ^{c}}(|\nabla
v_{n}(x)|^{2}+v_{n}^{2}(x))dx<\theta \text{ uniformly for }n\geq 1.
\end{equation*}%
Since $\{v_{n}\}$ is bounded in $H^{1}(\mathbb{R}^{N}),$ one can assume that
there exist a subsequence $\{v_{n}\}$ and $v_{a}^{-}\in H^{1}(\mathbb{R}%
^{N}) $ such that%
\begin{eqnarray*}
v_{n} &\rightharpoonup &v_{a}^{-}\text{ weakly in }H^{1}(\mathbb{R}^{N}); \\
v_{n} &\rightarrow &v_{a}^{-}\text{ strongly in }L_{loc}^{r}(\mathbb{R}^{N})%
\text{ for }2\leq r<2^{\ast }; \\
v_{n} &\rightarrow &v_{a}^{-}\text{ a.e. in }\mathbb{R}^{N}.
\end{eqnarray*}%
In the following, by adapting the argument of Theorem \ref{t0-3} $(i),$ we
obtain
\begin{equation*}
v_{n}\rightarrow v_{a}^{-}\text{ strongly in }H^{1}(\mathbb{R}^{N})
\end{equation*}%
and
\begin{equation*}
J_{a}^{\infty }(v_{n})\rightarrow J_{a}^{\infty }(v_{a}^{-})=\alpha
_{a}^{\infty ,-}\text{ as }n\rightarrow \infty .
\end{equation*}%
Thus, $v_{a}^{-}$ is a minimizer for $J_{a}^{\infty }$ on $\mathbf{M}%
_{a}^{\infty ,-}$ for each $0<a<\Lambda .$ By $(\ref{27})$ one has
\begin{equation*}
J_{a}^{\infty }(v_{a}^{-})=\alpha _{a}^{\infty ,-}\leq J_{a}^{\infty
}(t_{a}^{\infty ,-}w_{0})<\frac{p-2}{2p}D(p)\left( \frac{2S_{p}^{p}}{%
f_{\infty }(4-p)}\right) ^{\frac{2}{p-2}},
\end{equation*}%
which implies that $v_{a}^{-}\in \mathbf{M}_{a}^{\infty ,(1)}$. Since $%
|v_{a}^{-}|\in \mathbf{M}_{a}^{\infty ,-}\ $and$\ J_{a}^{\infty
}(|v_{a}^{-}|)=J_{a}^{\infty }(v_{a}^{-})=\alpha _{a}^{\infty ,-},$ one can
see that $v_{a}^{-}$ is a positive solution of Eq. $(E_{a})$ according to
Lemma \ref{g2}. Note that for $2<p<\min \{4,2^{\ast }\},$ there holds
\begin{equation*}
(4-p)\int_{\mathbb{R}^{N}}f_{\infty }|v_{a}^{-}|^{p}dx<2\left\Vert
v_{a}^{-}\right\Vert _{H^{1}}^{2}\text{ and }T_{f_{\infty
}}(v_{a}^{-})v_{a}^{-}\in \mathbf{M}_{0}^{\infty },
\end{equation*}%
where $\mathbf{M}_{0}^{\infty }=\mathbf{M}_{a}^{\infty }$ with $a=0$ and
\begin{equation*}
\left( \frac{4-p}{2}\right) ^{1/(p-2)}<T_{f_{\infty }}(v_{a}^{-}):=\left(
\frac{\left\Vert v_{a}^{-}\right\Vert _{H^{1}}^{2}}{\int_{\mathbb{R}%
^{N}}f_{\infty }|v_{a}^{-}|^{p}dx}\right) ^{1/(p-2)}<1.
\end{equation*}%
Then by Lemmas $\ref{g6}-\ref{g15}$, we have
\begin{equation*}
J_{a}^{\infty }(v_{a}^{-})=\sup_{0\leq t\leq t_{a}^{+}}J_{a}^{\infty
}(tv_{a}^{-}),
\end{equation*}%
where $1<\left( \frac{2}{4-p}\right) ^{1/(p-2)}T_{f_{\infty
}}(v_{a}^{-})<t_{a}^{+}$. Using the above equality, one get
\begin{equation*}
J_{a}^{\infty }(v_{a}^{-})>J_{a}^{\infty }(T_{f_{\infty
}}(v_{a}^{-})v_{a}^{-})>J_{0}^{\infty }(T_{f_{\infty
}}(v_{a}^{-})v_{a}^{-})\geq \alpha _{0}^{\infty }.
\end{equation*}

$(ii)$ Following the argument of Theorem 2.1 in \cite{SZ}. By $(i),$ we
obtain that Eq. $(E_{a})$ admits a positive solution $v_{a,1}^{-}\in H^{1}(%
\mathbb{R}^{N}).$ Applying Theorem 4.1 in \cite{HL1} gives $%
v_{a,1}^{-}\rightarrow 0$ as $|x|\rightarrow \infty .$ Then after
translation, we can make $v_{a,1}^{-}$ satisfy%
\begin{equation}
v_{a,1}^{-}>0,v_{a,1}^{-}(\infty )=0,v_{a,1}^{-}(0)=\max v_{a,1}^{-}(x).
\label{18-5}
\end{equation}

Now we show that $v_{a,1}^{-}$ is unique under $(\ref{18-5})$. Otherwise, we
assume that $v_{a,2}^{-}$ is another positive solution satisfying $(\ref%
{18-5})$. Let%
\begin{equation*}
K_{1}=b+a\int_{\mathbb{R}^{N}}|\nabla v_{a,1}^{-}|^{2}dx\text{ and }%
K_{2}=b+a\int_{\mathbb{R}^{N}}|\nabla v_{a,2}^{-}|^{2}dx.
\end{equation*}%
Then $v_{a,i}^{-}(i=1,2)$ is a solution of the problem%
\begin{equation*}
-\Delta u+\frac{1}{K_{i}}u=\frac{1}{K_{i}}|u|^{p-2}u\text{ in }\mathbb{R}%
^{N}.
\end{equation*}%
Let $w_{i}(x)=v_{a,i}^{-}(\sqrt{K_{i}}x).$ Then $w_{i}(x)$ is a solution of%
\begin{equation}
\left\{
\begin{array}{ll}
-\Delta w+w=|w|^{p-2}w\text{ } & \text{in }\mathbb{R}^{N}, \\
w>0,w(\infty )=0,w(0)=\max w(x). &
\end{array}%
\right.  \label{18-6}
\end{equation}%
It follows from \cite{K} that the solution of problem $(\ref{18-6})$ is
unique. So $w_{1}(x)\equiv w_{2}(x)$ i.e., $v_{a,1}^{-}(\sqrt{K_{1}}%
x)=v_{a,2}^{-}(\sqrt{K_{2}}x).$ Thus, we have
\begin{equation}
v_{a,2}^{-}(x)=v_{a,1}^{-}\left( \sqrt{\frac{K_{1}}{K_{2}}}x\right) .
\label{18-8}
\end{equation}%
Thus, we have%
\begin{eqnarray*}
K_{2} &=&b+a\int_{\mathbb{R}^{N}}|\nabla v_{a,2}^{-}|^{2}dx \\
&=&b+a\left( \frac{K_{2}}{K_{1}}\right) ^{\frac{N-2}{2}}\int_{\mathbb{R}%
^{N}}|\nabla v_{a,1}^{-}|^{2}dx \\
&=&b+\left( \frac{K_{2}}{K_{1}}\right) ^{\frac{N-2}{2}}(K_{1}-b),
\end{eqnarray*}%
which implies that%
\begin{equation}
\frac{(K_{2}-b)^{2}}{K_{2}^{N-2}}=\frac{(K_{1}-b)^{2}}{K_{1}^{N-2}}.
\label{18-7}
\end{equation}%
Define%
\begin{equation*}
y(x)=\frac{(x-b)^{2}}{x^{N-2}}\text{ for }x>b.
\end{equation*}%
A direct calculation shows that
\begin{equation*}
y^{\prime }(x)=\frac{\left[ (4-N)x+b(N-2)\right] (x-b)}{x^{N-1}}>0\text{ for
}x>b\text{ and }1\leq N\leq 4,
\end{equation*}
which implies that $y(x)$ is strictly increasing when $x>b$ and $1\leq N\leq
4.$ This indicates that $K_{1}=K_{2},$since $K_{1},K_{2}>b.$ So by $(\ref%
{18-8})$ one have $v_{a,1}^{-}=v_{a,2}^{-}.$ Since the unique solution $%
w_{1}(x)$ of problem $(\ref{18-6})$ is radially symmetric by \cite{K} and $%
w_{1}(x)=v_{a,1}^{-}(\sqrt{K_{1}}x)$, we obtain that $v_{a,1}^{-}$ is also
radially symmetric.

$\left( iii\right) $ Note that $\Lambda \leq \overline{a}_{\ast }$ for $%
N\geq 5.$ Then by virtue of Theorem \ref{t0-3} $(i),$ for each $0<a<\Lambda $
Eq. $(E_{a})$ admits a positive ground state solution $v_{a}^{+}\in H^{1}(%
\mathbb{R}^{N})$ such that $J_{a}^{\infty }(v_{a}^{+})<0$ when $N\geq 5.$
Clearly, $v_{a}^{+}\in \mathbf{M}_{a}^{\infty }\left( \frac{D(p)(p-2)}{2p}%
\left( \frac{2S_{p}^{p}}{f_{\infty }(4-p)}\right) ^{2/(p-2)}\right) ,$ since
$J_{a}^{\infty }(v_{a}^{+})<0.$ By Lemma \ref{g1}, we further get $%
v_{a}^{+}\in \mathbf{M}_{a}^{\infty ,(2)}.$ Thus, there holds%
\begin{equation*}
\left\Vert v_{a}^{+}\right\Vert _{H^{1}}>\sqrt{2}\left( \frac{2S_{p}^{p}}{%
f_{\infty }(4-p)}\right) ^{1/(p-2)}.
\end{equation*}%
Moreover, from $(i)$ it follows that for each $0<a<\Lambda ,$ Eq. $(E_{a})$
admits a positive solution $v_{a}^{-}\in H^{1}(\mathbb{R}^{N})$ satisfying%
\begin{equation*}
J_{a}^{\infty }(v_{a}^{-})>\frac{p-2}{4p}\left( \frac{S_{p}^{p}}{f_{\infty }}%
\right) ^{2/(p-2)}\text{ and }\left\Vert v_{a}^{-}\right\Vert
_{H^{1}}<\left( \frac{2S_{p}^{p}}{f_{\infty }(4-p)}\right) ^{1/(p-2)}.
\end{equation*}%
Consequently, we complete the proof of Theorem \ref{t1}.

\section{Proofs of Theorems \protect\ref{t2} and \protect\ref{t3}}

By virtue of Theorem \ref{t1} $(i)$, we know that Eq. $(E_{a}^{\infty })$
admits a positive solution $v_{a}^{-}\in \mathbf{M}_{a}^{\infty ,-}$ such
that
\begin{equation*}
J_{a}^{\infty }(v_{a}^{-})=\alpha _{a}^{\infty ,-}\text{ and }(4-p)\int_{%
\mathbb{R}^{N}}f_{\infty }|v_{a}^{-}|^{p}dx<2\Vert v_{a}^{-}\Vert
_{H^{1}}^{2}.
\end{equation*}%
According to $(\ref{2-0}),$ one has
\begin{equation}
T_{f_{\infty }}(v_{a}^{-})=\left( \frac{\left\Vert v_{a}^{-}\right\Vert
_{H^{1}}^{2}}{\int_{\mathbb{R}^{N}}f_{\infty }|v_{a}^{-}|^{p}dx}\right)
^{1/(p-2)}>\left( \frac{4-p}{2}\right) ^{1/(p-2)}.  \label{5-1}
\end{equation}%
Moreover, by Theorem \ref{t1} $(ii)$, we obtain that Eq. $(E_{a}^{\infty })$
admits a positive solution $v_{a}^{+}\in \mathbf{M}_{a}^{\infty ,+}$ such
that
\begin{equation*}
J_{a}^{\infty }(v_{a}^{+})=\alpha _{a}^{\infty ,+}\ \text{and }(4-p)\int_{%
\mathbb{R}^{N}}f_{\infty }|v_{a}^{+}|^{p}dx>2\Vert v_{a}^{+}\Vert
_{H^{1}}^{2}.
\end{equation*}%
Similar to $(\ref{5-1})$, we have
\begin{equation*}
T_{f_{\infty }}(v_{a}^{+})=\left( \frac{\left\Vert v_{a}^{+}\right\Vert
_{H^{1}}^{2}}{\int_{\mathbb{R}^{N}}f_{\infty }|v_{a}^{+}|^{p}dx}\right)
^{1/(p-2)}<\left( \frac{4-p}{2}\right) ^{1/(p-2)}.
\end{equation*}%
Then we have the following results.

\begin{lemma}
\label{m5}$(i)$ Suppose that $N\geq 1.$ Then for each $0<a<\Lambda ,$ there
exists $1<\left( \frac{2}{4-p}\right) ^{1/(p-2)}T_{f_{\infty
}}(v_{a}^{-})<t_{a}^{\infty ,-}$ such that
\begin{equation}
J_{a}^{\infty }(v_{a}^{-})=\sup_{0\leq t\leq t_{a}^{\infty ,-}}J_{a}^{\infty
}(tv_{a}^{-})=\alpha _{a}^{\infty ,-}.  \label{eqq36}
\end{equation}%
\newline
$(ii)$ Suppose that $N\geq 5.$ Then for each $0<a<\Lambda ,$ there exists $%
0<t_{a}^{\infty ,+}<1$ such that
\begin{equation*}
J_{a}^{\infty }(v_{a}^{+})=\inf_{t\geq t_{a}^{\infty ,+}}J_{a}^{\infty
}(tv_{a}^{+})=\alpha _{a}^{\infty ,+}.
\end{equation*}
\end{lemma}

\begin{proof}
$(i)$ Let
\begin{equation}
b_{a}^{\infty }(t)=t^{-2}\left\Vert v_{a}^{-}\right\Vert
_{H^{1}}^{2}-t^{p-4}\int_{\mathbb{R}^{N}}f_{\infty }|v_{a}^{-}|^{p}dx\text{
for }t>0.  \label{5-3}
\end{equation}%
Clearly, there holds
\begin{equation}
b_{a}^{\infty }(1)+a\left( \int_{\mathbb{R}^{N}}|\nabla
v_{a}^{-}|^{2}dx\right) ^{2}=0  \label{5-4}
\end{equation}%
for all $0<a<\Lambda .$ It is easy to verify that
\begin{equation*}
b_{a}^{\infty }(t_{f_{\infty }}(v_{a}^{-}))=0,\ \lim_{t\rightarrow
0^{+}}b_{a}^{\infty }(t)=\infty \text{ and }\lim_{t\rightarrow \infty
}b_{a}^{\infty }(t)=0.
\end{equation*}%
By calculating the derivative of $b_{a}^{\infty }(t)$ one has%
\begin{equation*}
(b_{a}^{\infty })^{\prime }(t)=t^{-3}\left( -2\left\Vert
v_{a}^{-}\right\Vert _{H^{1}}^{2}+(4-p)t^{p-2}\int_{\mathbb{R}^{N}}f_{\infty
}|v_{a}^{-}|^{p}dx\right) ,
\end{equation*}%
which implies that $b_{a}^{\infty }(t)$ is decreasing when $0<t<\left( \frac{%
2}{4-p}\right) ^{1/(p-2)}T_{f_{\infty }}(v_{a}^{-})$ and is increasing when $%
t>\left( \frac{2}{4-p}\right) ^{1/(p-2)}T_{f_{\infty }}(v_{a}^{-}).$ This
indicates that
\begin{equation}
\inf_{t>0}b_{a}^{\infty }(t)=b_{a}^{\infty }\left( \left( \frac{2}{4-p}%
\right) ^{1/(p-2)}T_{f_{\infty }}(v_{a}^{-})\right) .  \label{5-2}
\end{equation}%
Moreover, we notice that
\begin{equation}
\left( \frac{2}{4-p}\right) ^{1/(p-2)}T_{f_{\infty }}(v_{a}^{-})>1.
\label{5-5}
\end{equation}%
Thus, it follows from $(\ref{5-4})-(\ref{5-5})$ that
\begin{equation}
\inf_{t>0}b_{a}^{\infty }(t)<b_{a}^{\infty }(1)=-a\left( \int_{\mathbb{R}%
^{N}}|\nabla v_{a}^{-}|^{2}dx\right) ^{2},  \label{5-6}
\end{equation}%
which shows that there exists $1<\left( \frac{2}{4-p}\right)
^{1/(p-2)}T_{f_{\infty }}(v_{a}^{-})<t_{a}^{\infty ,-}$ such that
\begin{equation*}
b_{a}^{\infty }(t_{a}^{\infty ,-})+a\left( \int_{\mathbb{R}^{N}}|\nabla
v_{a}^{-}|^{2}dx\right) ^{2}=0.
\end{equation*}%
Using a similar argument as in Lemma \ref{g6}, we arrive at $(\ref{eqq36})$.%
\newline
$(ii)$ The proof is analogous to that of part $(i)$, and we omit it here.
\end{proof}

\begin{lemma}
\label{m3}Suppose that $N\geq 1$ and conditions ${(D1)}-{(D2),(D4)}-{(D5)}$
hold. Then for each $0<a<\Lambda ,$ there exist two constants $t_{a}^{(1),-}$
and $t_{a}^{(2),-}$ satisfying
\begin{equation*}
T_{f}(v_{a}^{-})<t_{a}^{(1),-}<\left( \frac{2}{4-p}\right)
^{1/(p-2)}T_{f}(v_{a}^{-})<t_{a}^{(2),-}
\end{equation*}%
such that $t_{a}^{(i),-}v_{a}^{-}\in \mathbf{M}_{a}^{(i)}\ (i=1,2),$ and
\begin{equation*}
J_{a}\left( t_{a}^{(1),-}v_{a}^{-}\right) =\sup_{0\leq t\leq
t_{a}^{(2),-}}J_{a}(tv_{a}^{-})<\alpha _{a}^{\infty ,-}\text{ and }%
J_{a}\left( t_{a}^{(2),-}v_{a}^{-}\right) =\inf_{t\geq
t_{a}^{(1),-}}J_{a}(tv_{a}^{-}),
\end{equation*}%
where $T_{f}(v_{a}^{-})$ is as $(\ref{2-0})$ with $u=v_{a}^{-}.$
\end{lemma}

\begin{proof}
Let
\begin{equation}
b_{a}(t)=t^{-2}\left\Vert v_{a}^{-}\right\Vert _{H^{1}}^{2}-t^{p-4}\int_{%
\mathbb{R}^{N}}f(x)|v_{a}^{-}|^{p}dx\text{ for }t>0.  \label{eqq41}
\end{equation}%
Apparently, $tv_{a}^{-}\in \mathbf{M}_{a}$ if and only if
\begin{equation*}
b_{a}(t)+a\left( \int_{\mathbb{R}^{N}}|\nabla v_{a}^{-}|^{2}dx\right) ^{2}=0.
\end{equation*}%
By analyzing $(\ref{eqq41}),$ we obtain
\begin{equation*}
b_{a}(T_{f}(v_{a}^{-}))=0,\ \lim_{t\rightarrow 0^{+}}b_{a}(t)=\infty \text{
and }\lim_{t\rightarrow \infty }b_{a}(t)=0.
\end{equation*}%
A direct calculation shows that
\begin{equation*}
b_{a}^{\prime }(t)=t^{-3}\left( -2\left\Vert v_{a}^{-}\right\Vert
_{H^{1}}^{2}+(4-p)t^{p-2}\int_{\mathbb{R}^{N}}f(x)|v_{a}^{-}|^{p}dx\right) ,
\end{equation*}%
which implies that $b_{a}(t)$ is decreasing on $0<t<\left( \frac{2}{4-p}%
\right) ^{1/(p-2)}T_{f}(v_{a}^{-})$ and is increasing on $t>\left( \frac{2}{%
4-p}\right) ^{1/(p-2)}T_{f}(v_{a}^{-}).$ By virtue of condition $(D5)$ one
has
\begin{equation*}
T_{f}(v_{a}^{-})\leq T_{f_{\infty }}(v_{a}^{-})<1\text{ and }b_{a}(t)\leq
b_{a}^{\infty }(t),
\end{equation*}%
where $b_{a}^{\infty }(t)$ is given in $(\ref{5-3}).$ Using condition $(D5)$
and $(\ref{5-6})$ gives
\begin{eqnarray*}
\inf_{t>0}b_{a}(t) &=&b_{a}\left( \left( \frac{2}{4-p}\right)
^{1/(p-2)}T_{f}(v_{a}^{-})\right) \\
&\leq &-\frac{p-2}{4-p}\left( \frac{4-p}{2}\right) ^{2/(p-2)}\left\Vert
v_{a}^{-}\right\Vert _{H^{1}}^{2}\left( \frac{\left\Vert
v_{a}^{-}\right\Vert _{H^{1}}^{2}}{\int_{\mathbb{R}^{N}}f_{\infty
}|v_{a}^{-}|^{p}dx}\right) ^{-2/(p-2)} \\
&=&\inf_{t>0}b_{a}^{\infty }(t)<-a\left( \int_{\mathbb{R}^{N}}|\nabla
v_{a}^{-}|^{2}dx\right) ^{2}.
\end{eqnarray*}%
This explicitly tells us that there are two constants $t_{a}^{(1),-}$ and $%
t_{a}^{(2),-}$ satisfying%
\begin{equation*}
T_{f}(v_{a}^{-})<t_{a}^{(1),-}<\left( \frac{2}{4-p}\right)
^{1/(p-2)}T_{f}\left( v_{a}^{-}\right) <t_{a}^{(2),-}
\end{equation*}%
such that
\begin{equation*}
b_{a}\left( t_{a}^{(i),-}\right) +a\left( \int_{\mathbb{R}^{N}}|\nabla
v_{a}^{-}|^{2}dx\right) ^{2}=0\text{ for }i=1,2.
\end{equation*}%
That is, $t_{a}^{(i),-}v_{a}^{-}\in \mathbf{M}_{a}$ $(i=1,2).$

A direct calculation on the second order derivatives gives
\begin{eqnarray*}
h_{a,t_{a}^{(1),-}v_{a}^{-}}^{\prime \prime }(1) &=&-2\left\Vert
t_{a}^{(1),-}v_{a}^{-}\right\Vert _{H^{1}}^{2}+(4-p)\int_{\mathbb{R}%
^{N}}f(x)\left\vert t_{a}^{(1),-}v_{a}^{-}\right\vert ^{p}dx \\
&=&\left( t_{a}^{(1),-}\right) ^{5}b_{a}^{\prime }\left(
t_{a}^{(1),-}\right) <0
\end{eqnarray*}%
and
\begin{eqnarray*}
h_{a,t_{a}^{(2),-}v_{a}^{-}}^{\prime \prime }(1) &=&-2\left\Vert
t_{a}^{(2),-}v_{a}^{-}\right\Vert _{H^{1}}^{2}+(4-p)\int_{\mathbb{R}%
^{N}}f(x)\left\vert t_{a}^{(2),-}v_{a}^{-}\right\vert ^{p}dx \\
&=&\left( t_{a}^{(2),-}\right) ^{5}b_{a}^{\prime }\left(
t_{a}^{(2),-}\right) >0.
\end{eqnarray*}%
So, we get $t_{a}^{(1),-}v_{a}^{-}\in \mathbf{M}_{a}^{-}$ and $%
t_{a}^{(2),-}v_{a}^{-}\in \mathbf{M}_{a}^{+}.$

Note that
\begin{equation*}
t_{a}^{(1),-}<\left( \frac{2}{4-p}\right) ^{\frac{1}{p-2}}T_{f}(v_{a}^{-})%
\leq \left( \frac{2}{4-p}\right) ^{\frac{1}{p-2}}T_{f_{\infty
}}(v_{a}^{-})<t_{a}^{\infty },
\end{equation*}%
where $t_{a}^{\infty }$ is the same as described in Lemma \ref{m5}. It
follows from Lemma \ref{m5}$(i)$ and condition $\left( D5\right) $ that for
each $0<a<\Lambda ,$%
\begin{eqnarray*}
J_{a}\left( t_{a}^{(1),-}v_{a}^{-}\right) &=&J_{a}^{\infty }\left(
t_{a}^{(1),-}v_{a}^{-}\right) -\frac{\left( t_{a}^{\left( 1\right)
,-}\right) ^{p}}{p}\int_{\mathbb{R}^{N}}(f(x)-f_{\infty })(v_{a}^{-})^{p}dx
\\
&\leq &\sup_{0\leq t\leq t_{a}^{\infty }}J_{a}^{\infty }(tv_{a}^{-})-\frac{%
\left( t_{a}^{\left( 1\right) ,-}\right) ^{p}}{p}\int_{\mathbb{R}%
^{N}}(f(x)-f_{\infty })(v_{a}^{-})^{p}dx \\
&<&\alpha _{a}^{\infty ,-}<\frac{p-2}{2p}D(p)\left( \frac{2S_{p}^{p}}{%
f_{\infty }(4-p)}\right) ^{2/(p-2)},
\end{eqnarray*}%
which indicates that $t_{a}^{(1),-}v_{a}^{-}\in \mathbf{M}_{a}^{(1)}$ and $%
J_{a}\left( t_{a}^{(1),-}v_{a}^{-}\right) <\alpha _{a}^{\infty ,-}.$ Note
that%
\begin{equation*}
h_{a,v_{a}^{-}}^{\prime }\left( t\right) =t^{3}\left( b_{a}\left( t\right)
+a\left( \int_{\mathbb{R}^{N}}\left\vert \nabla v_{a}^{-}\right\vert
^{2}dx\right) ^{2}\right) .
\end{equation*}%
Then we have $h_{a,v_{a}^{-}}^{\prime }(t)>0$ for all $t\in \left(
0,t_{a}^{(1),-}\right) \cup \left( t_{a}^{(2),-},\infty \right) $ and $%
h_{a,v_{a}^{-}}^{\prime }(t)<0$ for all $t\in \left(
t_{a}^{(1),-},t_{a}^{(2),-}\right) $. Consequently, we arrive at
\begin{equation*}
J_{a}\left( t_{a}^{(1),-}v_{a}^{-}\right) =\sup_{0\leq t\leq
t_{a}^{(2),-}}J_{a}(tv_{a}^{-})\text{ and }J_{a}\left(
t_{a}^{(2),-}v_{a}^{-}\right) =\inf_{t\geq t_{a}^{\left( 1\right)
,-}}J_{a}(tv_{a}^{-}).
\end{equation*}%
That is, $J_{a}\left( t_{a}^{(2),-}v_{a}^{-}\right) \leq J_{a}\left(
t_{a}^{(1),-}v_{a}^{-}\right) <\alpha _{a}^{\infty ,-}$, and so $%
t_{a}^{(2),-}v_{a}^{-}\in \mathbf{M}_{a}^{(2),-}.$ Consequently, this
completes the proof.
\end{proof}

\begin{lemma}
\label{m3-2}Suppose that $N\geq 5$ and conditions ${(D1)}-{(D4)}$ hold. Then
for each $0<a<\Lambda ,$ there exist two constants $t_{a}^{(1),+}$ and $%
t_{a}^{(2),+}$ satisfying
\begin{equation*}
T_{f}(v_{a}^{+})<t_{a}^{(1),+}<\left( \frac{2}{4-p}\right)
^{1/(p-2)}T_{f}(v_{a}^{+})<t_{a}^{(2),+}
\end{equation*}%
such that $t_{a}^{(i),+}v_{a}^{+}\in \mathbf{M}_{a}^{(i)}\ (i=1,2),$ and
\begin{equation*}
J_{a}\left( t_{a}^{(2),+}v_{a}^{+}\right) =\inf_{t\geq
t_{a}^{(1)}}J_{a}(tv_{a}^{+})<\alpha _{a}^{\infty ,+}\text{ and }J_{a}\left(
t_{a}^{(2)}v_{a}^{+}\right) =\inf_{t\geq t_{a}^{(1)}}J_{a}(tv_{a}^{+}),
\end{equation*}%
where $T_{f}(v_{a}^{+})$ is as $(\ref{2-0})$ with $u=v_{a}^{+}.$
\end{lemma}

\begin{proof}
The proof is analogous to that of Lemma \ref{m3}, and we omit it here.
\end{proof}

Following \cite{NT,T}, we have the following result.

\begin{lemma}
\label{g9}Suppose that $0<a<\Lambda .$ Then for each $u\in \mathbf{M}%
_{a}^{(j)}(j=1,2),$ there exist a constant $\sigma >0$ and a differentiable
function $t^{\ast }:B_{\sigma }(0)\subset H^{1}(\mathbb{R}^{N})\rightarrow
\mathbb{R}^{+}$ such that $t^{\ast }(0)=1\ $and$\ t^{\ast }(v)(u-v)\in
\mathbf{M}_{a}^{(j)}$ for all $v\in B_{\sigma }(0),$ and
\begin{eqnarray*}
&&\left\langle (t^{\ast })^{\prime }(0),\varphi \right\rangle \\
&=&\frac{2\int_{\mathbb{R}^{N}}(\nabla u\nabla \varphi +u\varphi
)dx+4a\left( \int_{\mathbb{R}^{N}}|\nabla u|dx\right) ^{2}\int_{\mathbb{R}%
^{N}}\nabla u\nabla \varphi dx-p\int_{\mathbb{R}^{N}}f(x)|u|^{p-2}u\varphi dx%
}{\left\Vert u\right\Vert _{H^{1}}^{2}-(p-1)\int_{\mathbb{R}%
^{N}}f(x)|u|^{p}dx}
\end{eqnarray*}%
for all $\varphi \in H^{1}(\mathbb{R}^{N}).$
\end{lemma}

\begin{proof}
For any $u\in \mathbf{M}_{a}^{(j)}$, we define the function $F_{u}:\mathbb{R}%
\times H^{1}(\mathbb{R}^{N})\rightarrow \mathbb{R}$ by%
\begin{eqnarray*}
F_{u}(t,v) &=&\left\langle J_{a}^{\prime }(t(u-v)),t(u-v)\right\rangle \\
&=&t^{2}\left\Vert u-v\right\Vert _{H^{1}}^{2}+at^{4}\left( \int_{\mathbb{R}%
^{N}}|\nabla u-\nabla v|^{2}dx\right) ^{2}-t^{p}\int_{\mathbb{R}%
^{N}}f(x)|u-v|^{p}dx.
\end{eqnarray*}%
Clearly, $F_{u}(1,0)=\left\langle J_{a}^{\prime }(u),u\right\rangle =0$ and%
\begin{eqnarray*}
\frac{d}{dt}F_{u}(1,0) &=&2\left\Vert u\right\Vert _{H^{1}}^{2}+4a\left(
\int_{\mathbb{R}^{N}}|\nabla u|^{2}dx\right) ^{2}-p\int_{\mathbb{R}%
^{N}}f(x)|u|^{p}dx \\
&=&-2\left\Vert u\right\Vert _{H^{1}}^{2}-\left( p-4\right) \int_{\mathbb{R}%
^{N}}f(x)|u|^{p}dx \\
&\neq &0.
\end{eqnarray*}%
Applying the implicit function theorem, there exist a constant $\sigma >0$
and a differentiable function $t^{\ast }:B_{\sigma }(0)\subset H^{1}(\mathbb{%
R}^{N})\rightarrow \mathbb{R}$ such that $t^{\ast }(0)=1$ and
\begin{eqnarray*}
&&\left\langle (t^{\ast })^{\prime }(0),\varphi \right\rangle \\
&=&\frac{2\int_{\mathbb{R}^{N}}(\nabla u\nabla \varphi +u\varphi
)dx+4a\left( \int_{\mathbb{R}^{N}}|\nabla u|dx\right) ^{2}\int_{\mathbb{R}%
^{N}}\nabla u\nabla \varphi dx-p\int_{\mathbb{R}^{N}}f(x)|u|^{p-2}u\varphi dx%
}{\left\Vert u\right\Vert _{H^{1}}^{2}-(p-1)\int_{\mathbb{R}%
^{N}}f(x)|u|^{p}dx}
\end{eqnarray*}%
for all $\varphi \in H^{1}(\mathbb{R}^{N}),$ and $F_{u}(t^{\ast }(v),v)=0$
for all $v\in B_{\sigma }(0),$ which is equivalent to%
\begin{equation*}
\left\langle J_{a}^{\prime }(t^{\ast }(v)(u-v)),t^{\ast
}(v)(u-v)\right\rangle =0\text{ for all }v\in B_{\sigma }(0).
\end{equation*}%
According to the continuity of the map $t^{\ast }$, for $\sigma $
sufficiently small we have
\begin{eqnarray*}
h_{a,t^{\ast }(v)(u-v)}^{\prime \prime }(1) &=&-2\left\Vert t^{\ast }\left(
v\right) \left( u-v\right) \right\Vert _{H^{1}}^{2}-(p-4)\int_{\mathbb{R}%
^{N}}f(x)|t^{\ast }(v)(u-v)|^{p}dx \\
&<&0,
\end{eqnarray*}%
and%
\begin{equation*}
J_{a}(t^{\ast }(v)(u-v))<\frac{(p-2)D(p)}{2p}\left( \frac{2S_{p}^{p}}{%
f_{\infty }(4-p)}\right) ^{2/(p-2)}.
\end{equation*}%
Hence, $t^{\ast }(v)(u-v)\in \mathbf{M}_{a}^{(j)}$ for all $v\in B_{\sigma
}(0).$ This completes the proof.
\end{proof}

By $(\ref{4-4})$ and Lemma $\ref{g7}$, we define
\begin{equation*}
\alpha _{a}^{-}=\inf_{u\in \mathbf{M}_{a}^{(1)}}J_{a}(u)=\inf_{u\in \mathbf{M%
}_{a}^{-}}J_{a}(u)\text{ for }N\geq 1.
\end{equation*}

\begin{proposition}
\label{g8}Suppose that $N\geq 1.$ Then for each $0<a<\Lambda ,$ there exists
a sequence $\{u_{n}\}\subset \mathbf{M}_{a}^{(1)}$ such that%
\begin{equation}
J_{a}(u_{n})=\alpha _{a}^{-}+o(1)\text{ and }J_{a}^{\prime }(u_{n})=o(1)%
\text{ in }H^{-1}(\mathbb{R}^{N}).  \label{eqq20}
\end{equation}
\end{proposition}

\begin{proof}
By the Ekeland variational principle \cite{E} and Lemma \ref{g1}, we obtain
that there exists a minimizing sequence $\{u_{n}\}\subset \mathbf{M}%
_{a}^{(1)}$ such that%
\begin{equation*}
J_{a}(u_{n})<\alpha _{a}^{-}+\frac{1}{n}
\end{equation*}%
and%
\begin{equation}
J_{a}(u_{n})\leq J_{a}(w)+\frac{1}{n}\left\Vert w-u_{n}\right\Vert _{H^{1}}%
\text{ for all }w\in \mathbf{M}_{a}^{(1)}.  \label{22}
\end{equation}%
Applying Lemma \ref{g9} with $u=u_{n}$, there exists a function $t_{n}^{\ast
}:B_{\epsilon _{n}}(0)\rightarrow \mathbb{R}$ for some $\epsilon _{n}>0$
such that $t_{n}^{\ast }(w)(u_{n}-w)\in \mathbf{M}_{a}^{(1)}.$ For $0<\delta
<\epsilon _{n}$ and $u\in H^{1}(\mathbb{R}^{N})$ with $u\not\equiv 0,$ we
set
\begin{equation*}
w_{\delta }=\frac{\delta u}{\left\Vert u\right\Vert _{H^{1}}}\text{ and }%
z_{\delta }=t_{n}^{\ast }(w_{\delta })(u_{n}-w_{\delta }).
\end{equation*}%
Since $z_{\delta }\in \mathbf{M}_{a}^{(1)},$ it follows from $(\ref{22})$
that%
\begin{equation*}
J_{a}(z_{\delta })-J_{a}(u_{n})\geq -\frac{1}{n}\Vert z_{\delta }-u_{n}\Vert
_{H^{1}}.
\end{equation*}%
Using the mean value theorem gives
\begin{equation*}
\left\langle J_{a}^{\prime }(u_{n}),z_{\delta }-u_{n}\right\rangle +o\left(
\left\Vert z_{\delta }-u_{n}\right\Vert _{H^{1}}\right) \geq -\frac{1}{n}%
\Vert z_{\delta }-u_{n}\Vert _{H^{1}}
\end{equation*}%
and
\begin{eqnarray}
&&\left\langle J_{a}^{\prime }(u_{n}),-w_{\delta }\right\rangle
+(t_{n}^{\ast }(w_{\delta })-1)\left\langle J_{a}^{\prime
}(u_{n}),u_{n}-w_{\delta }\right\rangle  \notag \\
&\geq &-\frac{1}{n}\Vert z_{\delta }-u_{n}\Vert _{H^{1}}+o(\Vert z_{\delta
}-u_{n}\Vert _{H^{1}}).  \label{23}
\end{eqnarray}%
Note that $t_{n}^{\ast }(w_{\delta })(u_{n}-w_{\delta })\in \mathbf{M}%
_{a}^{(1)}$. From $(\ref{23})$ it leads to
\begin{eqnarray*}
&&-\delta \left\langle J_{a}^{\prime }(u_{n}),\,\frac{u}{\left\Vert
u\right\Vert _{H^{1}}}\right\rangle +\frac{(t_{n}^{\ast }(w_{\delta })-1)}{%
t_{n}^{\ast }(w_{\delta })}\left\langle J_{a}^{\prime }(z_{\delta
}),t_{n}^{\ast }(w_{\delta })(u_{n}-w_{\delta })\right\rangle \\
&&+(t_{n}^{\ast }(w_{\delta })-1)\left\langle J_{a}^{\prime
}(u_{n})-J_{a}^{\prime }(z_{\delta }),u_{n}-w_{\delta }\right\rangle \\
&\geq &-\frac{1}{n}\left\Vert z_{\delta }-u_{n}\right\Vert
_{H^{1}}+o(\left\Vert z_{\delta }-u_{n}\right\Vert _{H^{1}}).
\end{eqnarray*}%
We rewrite the above inequality as
\begin{eqnarray}
\left\langle J_{a}^{\prime }(u_{n}),\frac{u}{\left\Vert u\right\Vert _{H^{1}}%
}\right\rangle &\leq &\frac{\left\Vert z_{\delta }-u_{n}\right\Vert _{H^{1}}%
}{\delta n}+\frac{o(\left\Vert z_{\delta }-u_{n}\right\Vert _{H^{1}})}{%
\delta }  \notag \\
&&+\frac{(t_{n}^{\ast }(w_{\delta })-1)}{\delta }\langle J_{a}^{\prime
}(u_{n})-J_{a}^{\prime }(z_{\delta }),u_{n}-w_{\delta }\rangle .  \label{24}
\end{eqnarray}%
There exists a constant $C>0$ independent of $\delta $ such that%
\begin{equation*}
\left\Vert z_{\delta }-u_{n}\right\Vert _{H^{1}}\leq \delta +C(\left\vert
t_{n}^{\ast }\left( w_{\delta }\right) -1\right\vert )
\end{equation*}%
and%
\begin{equation*}
\lim_{\delta \rightarrow 0}\frac{\left\vert t_{n}^{\ast }(w_{\delta
})-1\right\vert }{\delta }\leq \left\Vert (t_{n}^{\ast })^{\prime
}(0)\right\Vert \leq C.
\end{equation*}%
Letting $\delta \rightarrow 0$ in $(\ref{24})$ and using the fact that $%
\lim_{\delta \rightarrow 0}\left\Vert z_{\delta }-u_{n}\right\Vert
_{H^{1}}=0,$ we get%
\begin{equation*}
\left\langle J_{a}^{\prime }(u_{n}),\frac{u}{\left\Vert u\right\Vert _{H^{1}}%
}\right\rangle \leq \frac{C}{n},
\end{equation*}%
which leads to $(\ref{eqq20}).$
\end{proof}

Before proving Theorem \ref{t2}, we also need the following compactness
lemma which is an immediate conclusion of Proposition \ref{g10}.

\begin{lemma}
\label{lem2} Suppose that $N\geq 1$ and conditions ${(D1)}-{(D2),(D4)}-{(D5)}
$ hold. Let $\{u_{n}\}\subset \mathbf{M}_{a}^{(1)}$ be a $(PS)_{\beta }$%
--sequence in $H^{1}(\mathbb{R}^{N})$ for $J_{a}$ with $0<\beta <\alpha
_{a}^{\infty ,-}.$ Then there exist a subsequence $\{u_{n}\}$ and a nonzero $%
u_{0}$ in $H^{1}(\mathbb{R}^{N})$ such that $u_{n}\rightarrow u_{0}$
strongly in $H^{1}(\mathbb{R}^{N})$ and $J_{a}(u_{0})=\beta .$ Furthermore, $%
u_{0}$ is a nonzero solution of Eq. $(E_{a}).$
\end{lemma}

\textbf{We are now ready to prove Theorems \ref{t2} and \ref{t3}:} By
Proposition \ref{g8}, there exists a sequence $\{u_{n}\}\subset \mathbf{M}%
_{a}^{(1)}$ satisfying
\begin{equation*}
J_{a}(u_{n})=\alpha _{a}^{-}+o(1)\text{ and }J_{a}^{\prime }(u_{n})=o(1)%
\text{ in }H^{-1}(\mathbb{R}^{N}).
\end{equation*}%
It follows from Lemmas \ref{lem2} and \ref{m3} that Eq. $(E_{a})$ has a
nontrivial solution $u_{a}^{-}\in \mathbf{M}_{a}^{-}$ such that $%
J_{a}(u_{a}^{-})=\alpha _{a}^{-}.$ Thus, $u_{a}^{-}$ is a minimizer for $%
J_{a}$ on $\mathbf{M}_{a}^{-}.$ since
\begin{equation*}
\alpha _{a}^{-}<\alpha _{a}^{\infty ,-}<\frac{p-2}{2p}D(p)\left( \frac{%
2S_{p}^{p}}{f_{\infty }(4-p)}\right) ^{2/(p-2)},
\end{equation*}%
one has $u_{a}^{-}\in \mathbf{M}_{a}^{(1)}$. Similarly, we obtain $%
|u_{a}^{-}|\in \mathbf{M}_{a}^{-}$ and $J_{a}(|u_{a}^{-}|)=J_{a}(u_{a}^{-})=%
\alpha _{a}^{-}$. According to Lemma \ref{g2}, $u_{a}^{-}$ is a positive
solution of Eq. $(E_{a})$ when $N\geq 1.$ Consequently, the proof of Theorem %
\ref{t2} is complete.

Note that $\Lambda \leq \overline{a}_{\ast }$ for $N\geq 5.$ Then it follows
from Theorem \ref{t0-3}$(ii)$ that for each $0<a<\Lambda ,$ Eq. $(E_{a})$
admits a positive ground state solution $u_{a}^{+}\in H^{1}(\mathbb{R}^{N})$
such that $J_{a}(u_{a}^{+})<0.$ Clearly, $u_{a}^{+}\in \mathbf{M}_{a}\left(
\frac{D(p)(p-2)}{2p}\left( \frac{2S_{p}^{p}}{f_{\infty }(4-p)}\right)
^{2/(p-2)}\right) ,$ since $J_{a}(u_{a}^{+})<0.$ By virtue of Lemma \ref{g1}%
, we further have $u_{a}^{+}\in \mathbf{M}_{a}^{(2)},$ which implies that%
\begin{equation*}
\left\Vert u_{a}^{+}\right\Vert _{H^{1}}>\sqrt{2}\left( \frac{2S_{p}^{p}}{%
f_{\max }(4-p)}\right) ^{1/(p-2)}.
\end{equation*}%
Moreover, by Theorem \ref{t2} we obtain that for each $0<a<\Lambda ,$ Eq. $%
(E_{a})$ admits a positive solution $u_{a}^{-}\in H^{1}(\mathbb{R}^{N})$
satisfying%
\begin{equation*}
J_{a}(u_{a}^{-})>\frac{p-2}{4p}\left( \frac{S_{p}^{p}}{f_{\max }}\right)
^{2/(p-2)}\text{ and }\left\Vert u_{a}^{-}\right\Vert _{H^{1}}<\left( \frac{%
2S_{p}^{p}}{f_{\max }(4-p)}\right) ^{1/(p-2)}.
\end{equation*}%
Consequently, we complete the proof of Theorem \ref{t3}.

\section{Ground State Solutions}

\begin{lemma}
\label{g14}Suppose that $N=1$ and $f(x)\in C(\mathbb{R})$ is weakly
differentiable satisfying
\begin{equation*}
(p-1)(p-2)f\left( x\right) +2\langle \nabla f(x),x\rangle \geq 0.
\end{equation*}%
Let $u_{0}$ be a nontrivial solution of Eq. $(E_{a}).$ Then $u_{0}\in
\mathbf{M}_{a}^{-}.$
\end{lemma}

\begin{proof}
Since $u_{0}$ is a nontrivial solution of Eq. $(E_{a})$, we have
\begin{equation}
\left\Vert u_{0}\right\Vert _{H^{1}}^{2}+a\left( \int_{\mathbb{R}}|\nabla
u_{0}|^{2}dx\right) ^{2}-\int_{\mathbb{R}}f(x)|u_{0}|^{p}dx=0.  \label{3-1}
\end{equation}%
Following the argument of \cite[Lemma 3.1]{DM1}, $u_{0}$ satisfies the
Pohozaev type identity corresponding to Eq. $(E_{a})$ as follows
\begin{equation}
\frac{1}{2}\int_{\mathbb{R}}(-|\nabla u_{0}|^{2}+u_{0}^{2})dx-\frac{a}{2}%
\left( \int_{\mathbb{R}}|\nabla u_{0}|^{2}dx\right) ^{2}=\frac{1}{p}\int_{%
\mathbb{R}}(f(x)+\langle \nabla f(x),x\rangle )|u_{0}|^{p}dx.  \label{3-4}
\end{equation}%
Then it follows from $(\ref{3-1})-(\ref{3-4})$ and the assumption of $f(x)$
that%
\begin{eqnarray*}
h_{a,u_{0}}^{\prime \prime }(1) &=&-2\int_{\mathbb{R}}|\nabla u_{0}|^{2}dx+%
\frac{(p+2)}{p}\int_{\mathbb{R}}f(x)|u_{0}|^{p}dx+\frac{2}{p}\int_{\mathbb{R}%
}\langle \nabla f(x),x\rangle |u_{0}|^{p}dx \\
&&-\frac{4}{p}\int_{\mathbb{R}}f(x)|u_{0}|^{p}dx-\frac{4}{p}\int_{\mathbb{R}%
}\langle \nabla f(x),x\rangle |u_{0}|^{p}dx-(p-2)\int_{\mathbb{R}%
}f(x)|u_{0}|^{p}dx \\
&=&-2\int_{\mathbb{R}}|\nabla u_{0}|^{2}dx-\frac{1}{p}\int_{\mathbb{R}^{N}}%
\left[ (p-1)(p-2)f(x)+2\langle \nabla f(x),x\rangle \right] |u_{0}|^{p}dx \\
&<&0,
\end{eqnarray*}%
which shows that $u_{0}\in \mathbf{M}_{a}^{-}.$ This completes the proof.
\end{proof}

\begin{lemma}
\label{g13}Suppose that $N=2$ and $f(x)\in C(\mathbb{R}^{2})$ is weakly
differentiable satisfying
\begin{equation*}
(p-2)f(x)+\langle \nabla f(x),x\rangle \geq 0.
\end{equation*}
Let $u_{0}$ be a nontrivial solution of Eq. $(E_{a}).$ Then $u_{0}\in
\mathbf{M}_{a}^{-}.$
\end{lemma}

\begin{proof}
Since $u_{0}$ is a nontrivial solution of Eq. $(E_{a})$, there holds
\begin{equation}
\left\Vert u_{0}\right\Vert _{H^{1}}^{2}+a\left( \int_{\mathbb{R}%
^{2}}|\nabla u_{0}|^{2}dx\right) ^{2}-\int_{\mathbb{R}%
^{2}}f(x)|u_{0}|^{p}dx=0.  \label{3-3}
\end{equation}%
Moreover, $u_{0}$ satisfies the Pohozaev type identity corresponding to Eq. $%
(E_{a})$ as follows
\begin{equation}
\frac{2}{p}\int_{\mathbb{R}^{2}}u_{0}^{2}dx-\frac{1}{2}\int_{\mathbb{R}%
^{2}}\langle \nabla f(x),x\rangle |u_{0}|^{p}dx=\int_{\mathbb{R}^{2}}f\left(
x\right) |u_{0}|^{p}dx.  \notag
\end{equation}%
Using the above two equalities gives%
\begin{equation}
\frac{2}{p}\int_{\mathbb{R}^{2}}u_{0}^{2}dx-\frac{1}{2}\int_{\mathbb{R}%
^{2}}\langle \nabla f(x),x\rangle |u_{0}|^{p}dx=\left\Vert u_{0}\right\Vert
_{H^{1}}^{2}+a\left( \int_{\mathbb{R}^{2}}|\nabla u_{0}|^{2}dx\right) ^{2}.
\label{3-6}
\end{equation}%
Then it follows from $(\ref{3-3})-(\ref{3-6})$ and the assumption of $f(x)$
that%
\begin{eqnarray*}
h_{a,u_{0}}^{\prime \prime }(1) &=&2a\left( \int_{\mathbb{R}^{2}}|\nabla
u_{0}|^{2}dx\right) ^{2}-(p-2)\int_{\mathbb{R}^{2}}f\left( x\right)
|u_{0}|^{p}dx \\
&=&-2\int_{\mathbb{R}^{2}}|\nabla u_{0}|^{2}dx-\frac{2(p-2)}{p}\int_{\mathbb{%
R}^{2}}u_{0}^{2}dx-\int_{\mathbb{R}^{2}}\langle \nabla f(x),x\rangle
|u_{0}|^{p}dx \\
&&-(p-2)\int_{\mathbb{R}^{2}}f\left( x\right) |u_{0}|^{p}dx \\
&=&-2\int_{\mathbb{R}^{2}}|\nabla u_{0}|^{2}dx-\frac{2(p-2)}{p}\int_{\mathbb{%
R}^{2}}u_{0}^{2}dx-\int_{\mathbb{R}^{2}}\left[ (p-2)f(x)+\langle \nabla
f(x),x\rangle \right] |u_{0}|^{p}dx \\
&<&0,
\end{eqnarray*}%
which shows that $u_{0}\in \mathbf{M}_{a}^{-}.$ This completes the proof.
\end{proof}

\textbf{We are now ready to prove Theorem \ref{t4}:} Let $u^{-}$ be the
positive solution of Eq. $(E_{a})$ as described in Theorem \ref{t1}$(i)$ or %
\ref{t2}$(i)$. Then there holds $u^{-}\in \mathbf{M}_{a}^{-}$ and $%
J_{a}(u^{-})=\inf_{u\in \mathbf{M}_{a}^{-}}J_{a}(u)=\alpha _{a}^{-}.$ By
Lemma \ref{g14} or \ref{g13}, $u^{-}$ is a positive ground state solution of
Eq. $(E_{a}).$

\begin{lemma}
\label{g3}Suppose that $N=3$ and $f(x)\equiv f_{\infty }.$ Let $u_{0}$ be a
nontrivial solution of Eq. $(E_{a}).$ Then for each $a>0$ with $\sqrt{a^{2}+4%
}+\frac{2}{a}\geq A_{0},$ there holds $u_{0}\in \mathbf{M}_{a}^{-},$ where $%
A_{0}>0$ is defined as $(\ref{1-3}).$
\end{lemma}

\begin{proof}
Let $u_{0}$ be a nontrivial solution of Eq. $(E_{a})$. Then there holds
\begin{equation}
\left\Vert u_{0}\right\Vert _{H^{1}}^{2}+a\left( \int_{\mathbb{R}%
^{3}}|\nabla u_{0}|^{2}dx\right) ^{2}=f_{\infty }\int_{\mathbb{R}%
^{3}}|u_{0}|^{p}dx.  \label{6-9}
\end{equation}%
Moreover, $u_{0}$ satisfies the Pohozaev type identity corresponding to Eq. $%
(E_{a})$ as follows%
\begin{equation}
\frac{p}{6}\int_{\mathbb{R}^{3}}(|\nabla u_{0}|^{2}+3u_{0}^{2})dx+\frac{ap}{6%
}\left( \int_{\mathbb{R}^{3}}|\nabla u_{0}|^{2}dx\right) ^{2}=f_{\infty
}\int_{\mathbb{R}^{3}}|u_{0}|^{p}dx.  \label{6-10}
\end{equation}%
By Azzollini \cite[Theorem 1.1]{Az2}, for each $a>0$ there exists a constant%
\begin{equation*}
t_{a}=\frac{-a+\sqrt{a^{2}+4}}{2}>0
\end{equation*}%
such that $u_{0}(\cdot )=w_{0}(t_{a}\cdot ).$ Then it follows from $\left( %
\ref{1-8}\right) ,(\ref{6-9})-(\ref{6-10})$ that%
\begin{eqnarray*}
h_{a,u_{0}}^{\prime \prime }(1) &=&2a\left( \int_{\mathbb{R}^{3}}|\nabla
u_{0}|^{2}dx\right) ^{2}-\left( p-2\right) f_{\infty }\int_{\mathbb{R}%
^{3}}|u_{0}|^{p}dx \\
&=&\frac{\int_{\mathbb{R}^{3}}|\nabla u_{0}|^{2}dx}{6}\left[
a(-p^{2}+2p+12)\int_{\mathbb{R}^{3}}|\nabla u_{0}|^{2}dx-p(p-2)\right] -%
\frac{p(p-2)}{2}\int_{\mathbb{R}^{3}}u_{0}^{2}dx \\
&=&\frac{\int_{\mathbb{R}^{3}}|\nabla u_{0}|^{2}dx}{6}\left[
a(-p^{2}+2p+12)t_{a}^{-2}\int_{\mathbb{R}^{3}}|\nabla w_{0}|^{2}dx-p(p-2)%
\right] -\frac{p(p-2)}{2}\int_{\mathbb{R}^{3}}u_{0}^{2}dx \\
&=&\frac{\int_{\mathbb{R}^{3}}|\nabla u_{0}|^{2}dx}{6}\left[ \frac{%
2a(-p^{2}+2p+12)\int_{\mathbb{R}^{3}}|\nabla w_{0}|^{2}dx}{2+a\sqrt{a^{2}+4}}%
-p(p-2)\right] -\frac{p(p-2)}{2}\int_{\mathbb{R}^{3}}u_{0}^{2}dx. \\
&=&\frac{\int_{\mathbb{R}^{3}}|\nabla u_{0}|^{2}dx}{6}\left[ \frac{3\left(
p-1\right) (-p^{2}+2p+12)}{p\left( \frac{2}{a}+\sqrt{a^{2}+4}\right) }\left(
\frac{S_{p}^{p}}{f_{\infty }}\right) ^{\frac{2}{p-2}}-p(p-2)\right] -\frac{%
p(p-2)}{2}\int_{\mathbb{R}^{3}}u_{0}^{2}dx
\end{eqnarray*}%
This implies that%
\begin{equation*}
f_{\infty }\geq \left[ \frac{3\left( p-1\right) (-p^{2}+2p+12)}{%
p^{2}(p-2)\inf_{a>0}\left( \sqrt{a^{2}+4}+\frac{2}{a}\right) }\right]
^{2/\left( p-2\right) }S_{p}^{p}
\end{equation*}%
\begin{equation*}
h_{a,u_{0}}^{\prime \prime }(1)<0\text{ if }\sqrt{a^{2}+4}+\frac{2}{a}\geq
\inf_{a>0}\sqrt{a^{2}+4}+\frac{2}{a}\geq A_{0}=\frac{3\left( p-1\right)
(-p^{2}+2p+12)}{p^{2}(p-2)}\left( \frac{S_{p}^{p}}{f_{\infty }}\right) ^{%
\frac{2}{p-2}},
\end{equation*}%
and so for each $a>0$ with $\sqrt{a^{2}+4}+\frac{2}{a}\geq A_{0},$ there
holds $u_{0}\in \mathbf{M}_{a}^{-}$. This completes the proof.
\end{proof}

\begin{lemma}
\label{g5}Suppose that $N=4$ and $f(x)\equiv f_{\infty }$. Let $u_{0}$ be a
nontrivial solution of Eq. $(E_{a}).$ Then for each $0<a\leq \overline{A}%
_{0},$ there holds $u_{0}\in \mathbf{M}_{a}^{-},$ where $\overline{A}_{0}>0$
is defined as $(\ref{1-4}).$
\end{lemma}

\begin{proof}
Let $u_{0}$ be a nontrivial solution of Eq. $(E_{a})$. Then we have
\begin{equation}
\left\Vert u_{0}\right\Vert _{H^{1}}^{2}+a\left( \int_{\mathbb{R}%
^{4}}|\nabla u_{0}|^{2}dx\right) ^{2}=f_{\infty }\int_{\mathbb{R}%
^{4}}|u_{0}|^{p}dx.  \label{6-1}
\end{equation}%
Moreover, $u_{0}$ satisfies the Pohozaev type identity corresponding to Eq. $%
(E_{a})$ as follows%
\begin{equation}
\int_{\mathbb{R}^{4}}|\nabla u_{0}|^{2}dx+a\left( \int_{\mathbb{R}%
^{4}}|\nabla u_{0}|^{2}dx\right) ^{2}+2\int_{\mathbb{R}^{4}}u_{0}^{2}dx=%
\frac{4}{p}f_{\infty }\int_{\mathbb{R}^{4}}|u_{0}|^{p}dx  \label{6-2}
\end{equation}%
By Azzollini \cite[Theorem 1.1]{Az2}, for each $0<a<\left( \int_{\mathbb{R}%
^{4}}|\nabla w_{0}|^{2}dx\right) ^{-1}$ there exists a constant%
\begin{equation*}
t_{a}=\sqrt{1-a\int_{\mathbb{R}^{4}}|\nabla w_{0}|^{2}dx}>0
\end{equation*}%
such that $u_{0}(\cdot )=w_{0}(t_{a}\cdot ).$ Then it follows from $(\ref%
{6-1})-(\ref{6-2})$ that%
\begin{eqnarray*}
h_{a,u_{0}}^{\prime \prime }(1) &=&2a\left( \int_{\mathbb{R}^{4}}|\nabla
u_{0}|^{2}dx\right) ^{2}-\left( p-2\right) f_{\infty }\int_{\mathbb{R}%
^{4}}|u_{0}|^{p}dx \\
&=&\frac{\int_{\mathbb{R}^{4}}|\nabla u_{0}|^{2}dx}{4}\left[
a(4-p)(p+2)\int_{\mathbb{R}^{4}}|\nabla u_{0}|^{2}dx-p(p-2)\right] -\frac{%
p(p-2)}{4}\int_{\mathbb{R}^{4}}u_{0}^{2}dx \\
&=&\frac{\int_{\mathbb{R}^{4}}|\nabla u_{0}|^{2}dx}{4}\left[
a(4-p)(p+2)t_{a}^{-2}\int_{\mathbb{R}^{4}}|\nabla w_{0}|^{2}dx-p(p-2)\right]
-\frac{p(p-2)}{2}\int_{\mathbb{R}^{4}}u_{0}^{2}dx \\
&=&\frac{\int_{\mathbb{R}^{4}}|\nabla u_{0}|^{2}dx}{4}\left[ \frac{%
a(4-p)(p+2)\int_{\mathbb{R}^{4}}|\nabla w_{0}|^{2}dx}{1-a\int_{\mathbb{R}%
^{4}}|\nabla w_{0}|^{2}dx}-p(p-2)\right] -\frac{p(p-2)}{2}\int_{\mathbb{R}%
^{4}}u_{0}^{2}dx \\
&<&0\text{ for all }0<a\leq \overline{A}_{0},
\end{eqnarray*}%
which indicates that $u_{0}\in \mathbf{M}_{a}^{-}$ for all $0<a\leq
\overline{A}_{0}.$ This completes the proof.
\end{proof}

\begin{lemma}
\label{g16}Suppose that $N=4$ and $f(x)\equiv f_{\infty }$. Let $u_{0}$ be a
nontrivial solution of Eq. $(E_{a}).$ Then for each $a>A_{0}^{\ast },$ there
holds $u_{0}\in \mathbf{M}_{a}^{+},$ where $A_{0}^{\ast }>0$ is defined as $(%
\ref{1-5}).$
\end{lemma}

\begin{proof}
Let $u_{0}$ be a nontrivial solution of Eq. $(E_{a}).$ By $(\ref{6-1})-(\ref%
{6-2}),$ we have%
\begin{equation}
-(p-2)\left\Vert u_{0}\right\Vert _{H^{1}}^{2}+a(4-p)\left( \int_{\mathbb{R}%
^{4}}|\nabla u_{0}|^{2}dx\right) ^{2}=p\int_{\mathbb{R}^{4}}u_{0}^{2}dx-2%
\left\Vert u_{0}\right\Vert _{H^{1}}^{2}.  \label{D-1}
\end{equation}%
Applying the Gagliardo-Nirenberg and Young inequalities leads to%
\begin{eqnarray*}
\left\Vert u_{0}\right\Vert _{H^{1}}^{2}+a\left( \int_{\mathbb{R}%
^{4}}|\nabla u_{0}|^{2}dx\right) ^{2} &\leq &f_{\infty }C^{p}\left( \int_{%
\mathbb{R}^{4}}|\nabla u_{0}|^{2}dx\right) ^{p-2}\left( \int_{\mathbb{R}%
^{4}}u_{0}^{2}dx\right) ^{\left( 4-p\right) /2} \\
&\leq &\frac{p-2}{2}\left( \frac{4-p}{p}\right) ^{\left( 4-p\right) /\left(
p-2\right) }\left( f_{\infty }C^{p}\right) ^{\frac{2}{p-2}}\left( \int_{%
\mathbb{R}^{4}}|\nabla u_{0}|^{2}dx\right) ^{2} \\
&&+\frac{p}{2}\int_{\mathbb{R}^{4}}u_{0}^{2}dx,
\end{eqnarray*}%
which implies that%
\begin{equation}
p\int_{\mathbb{R}^{4}}u_{0}^{2}dx-2\left\Vert u_{0}\right\Vert
_{H^{1}}^{2}\geq 2\left( a-\frac{p-2}{2}\left( \frac{4-p}{p}\right) ^{\left(
4-p\right) /\left( p-2\right) }(f_{\infty }C^{p})^{2/\left( p-2\right)
}\right) \left( \int_{\mathbb{R}^{4}}|\nabla u_{0}|^{2}dx\right) ^{2}.
\label{D-0}
\end{equation}%
Thus, by $(\ref{2-2})$ and $(\ref{D-1})-(\ref{D-0}),$ for each $%
a>A_{0}^{\ast }$ one has%
\begin{eqnarray*}
h_{a,u_{0}}^{\prime \prime }(1) &=&-(p-2)\left\Vert u_{0}\right\Vert
_{H^{1}}^{2}+a(4-p)\left( \int_{\mathbb{R}^{4}}|\nabla u_{0}|^{2}dx\right)
^{2} \\
&=&p\int_{\mathbb{R}^{4}}u_{0}^{2}dx-2\left\Vert u_{0}\right\Vert
_{H^{1}}^{2} \\
&\geq &2\left( a-\frac{p-2}{2}\left( \frac{4-p}{p}\right) ^{\frac{4-p}{p-2}%
}(f_{\infty }C^{p})^{\frac{2}{p-2}}\right) \left( \int_{\mathbb{R}%
^{4}}|\nabla u_{0}|^{2}dx\right) ^{2} \\
&>&0,
\end{eqnarray*}%
which indicates that $u_{0}\in \mathbf{M}_{a}^{+}.$ This complete the proof.
\end{proof}

\textbf{We are now ready to prove Theorem \ref{t5}:} By Lemmas \ref{g3}--\ref%
{g16}, we can arrive at the conclusions directly.

\section{Appendix}

In order to verify $(\ref{18-4}),$ we use the concentration-compactness
lemma \cite{Li1,Li2}. First of all, we discuss the three possibilities on
the measures defined by a functional related to $J_{a}^{\infty }.$

For $2<p<\min \{4,2^{\ast }\},$ let $\{u_{n}\}\subset \mathbf{M}_{a}^{\infty
,(1)}$ be a sequence such that
\begin{equation}
\lim_{n\rightarrow \infty }J_{a}^{\infty }(u_{n})=\alpha _{a}^{\infty ,-}>0.
\label{7-1}
\end{equation}%
Define the functional $\Phi _{a}^{\infty }:H^{1}(\mathbb{R}^{N})\rightarrow
\mathbb{R}$ by%
\begin{equation*}
\Phi _{a}^{\infty }(u)=\frac{1}{4}\left\Vert u\right\Vert _{H^{1}}^{2}-\frac{%
4-p}{4p}\int_{\mathbb{R}^{N}}f_{\infty }|u|^{p}dx.
\end{equation*}%
By Lemma \ref{g1}, for any $u\in \mathbf{M}_{a}^{\infty ,(1)}$ one has%
\begin{equation*}
J_{a}^{\infty }(u)=\Phi _{a}^{\infty }(u)>0.
\end{equation*}%
Note that $\{u_{n}\}$ is bounded in $H^{1}(\mathbb{R}^{N}),$ since $%
\{u_{n}\}\subset \mathbf{M}_{a}^{\infty ,(1)}.$ Then there exist a
subsequence $\{u_{n}\}$ and $u_{\infty }\in H^{1}(\mathbb{R}^{N})$ such that%
\begin{eqnarray*}
u_{n} &\rightharpoonup &u_{\infty }\text{ weakly in }H^{1}(\mathbb{R}^{N}),
\\
u_{n} &\rightarrow &u_{\infty }\text{ strongly in }L_{loc}^{s}(\mathbb{R}%
^{N})\text{ for }2\leq s<2^{\ast }.
\end{eqnarray*}%
For any $u_{n}\in \mathbf{M}_{a}^{\infty ,(1)}$, we define the measure $%
y_{n}(\Omega )$ by%
\begin{equation}
y_{n}(\Omega )=\frac{1}{4}\int_{\Omega }(|\nabla u_{n}|^{2}+u_{n}^{2})dx-%
\frac{4-p}{4p}\int_{\Omega }f_{\infty }|u_{n}|^{p}dx.  \label{7-4}
\end{equation}%
Since $u_{n}\in \mathbf{M}_{a}^{\infty ,\left( 1\right) },$ we have $%
\left\Vert u_{n}\right\Vert _{H^{1}}<D_{1}.$ This implies that $y_{n}(\Omega
)$ are positive measures. Furthermore, using $(\ref{7-1})$ gives%
\begin{equation*}
y_{n}(\mathbb{R}^{N})=\Phi _{a}^{\infty }(u_{n})=\alpha _{a}^{\infty
,-}+o\left( 1\right)
\end{equation*}%
and then, by P.L. Lions \cite{Li1}, there are three possibilities:

\textbf{vanishing:} for all $r>0,$%
\begin{equation}
\lim_{n\rightarrow \infty }\sup_{\xi \in \mathbb{R}^{N}}\int_{B_{r}(\xi
)}dy_{n}=0,  \label{7-5}
\end{equation}%
where $B_{r}(\xi )=\{x\in \mathbb{R}^{N}:|x-\xi |<r\};$

\textbf{dichotomy:} there exist a constant $\alpha \in (0,\alpha
_{a}^{\infty ,-})$, two sequences $\{\xi _{n}\}$ and $\{r_{n}\},$ with $%
r_{n}\rightarrow +\infty $ and two nonnegative measures $y_{n}^{1}$ and $%
y_{n}^{2}$ such that%
\begin{equation}
y_{n}-(y_{n}^{1}+y_{n}^{2})\rightarrow 0,\text{ }y_{n}^{1}(\mathbb{R}%
^{N})\rightarrow \alpha ,\text{ }y_{n}^{2}(\mathbb{R}^{N})\rightarrow \alpha
_{a}^{\infty ,-}-\alpha ,  \label{7-6}
\end{equation}%
\begin{equation}
\text{\textrm{supp }}(y_{n}^{1})\subset B_{r_{n}}(\xi _{n}),\text{ \textrm{%
supp }}(y_{n}^{2})\subset \mathbb{R}^{N}\backslash B_{2r_{n}}(\xi _{n});
\label{7-7}
\end{equation}

\textbf{compactness:} there exists a sequence $\{\xi _{n}\}\subset \mathbb{R}%
^{N}$ with the following property: for any $\delta >0,$ there exists $%
r=r(\delta )>0$ such that%
\begin{equation}
\int_{B_{r}(\xi _{n})}dy_{n}\geq \alpha _{a}^{\infty ,-}-\delta ,\text{ for
large }n.  \label{7-8}
\end{equation}

\begin{theorem}
\label{g10} For $0<a<\Lambda ,$ compactness holds for the sequence of
measures $\{u_{n}\}$ defined in $(\ref{7-4}).$
\end{theorem}

\begin{proof}
$(I)$ Vanishing does not occur. Suppose the contrary. Then for all $r>0,$ $(%
\ref{7-5})$ holds. In particular, we deduce that there exists $\bar{r}>0$
such that%
\begin{equation*}
\lim_{n\rightarrow \infty }\sup_{\xi \in \mathbb{R}^{N}}\int_{B_{\bar{r}%
}(\xi )}u_{n}^{2}dx=0,
\end{equation*}%
which implies that $u_{n}\rightarrow 0$ strongly in $L^{s}(\mathbb{R}^{N})$
for $2<s<2^{\ast }$ by using \cite[Lemma I.1]{Li2}. Thus, for $%
\{u_{n}\}\subset \mathbf{M}_{a}^{\infty ,(1)},$ it follows from Lemma \ref%
{g1} that%
\begin{eqnarray*}
\frac{p-2}{4p}\left( \frac{S_{p}^{p}}{f_{\infty }}\right) ^{\frac{2}{p-2}}
&<&J_{a}^{\infty }(u_{n})=-\frac{a}{4}\left( \int_{\mathbb{R}^{N}}|\nabla
u_{n}|^{2}dx\right) ^{2}+\frac{f_{\infty }(p-2)}{2p}\int_{\mathbb{R}%
^{N}}|u_{n}|^{p}dx \\
&\leq &\frac{f_{\infty }(p-2)}{2p}\int_{\mathbb{R}^{N}}|u_{n}|^{p}dx%
\rightarrow 0,
\end{eqnarray*}%
which is a contradiction.\newline
$(II)$ Dichotomy does not occur. Suppose by contradiction that there exist a
constant $\alpha \in (0,\alpha _{a}^{\infty ,-})$, two sequences $\{\xi
_{n}\}$ and $\{r_{n}\},$ with $r_{n}\rightarrow +\infty $ and two
nonnegative measures $y_{n}^{1}$ and $y_{n}^{2}$ such that $(\ref{7-6})-(\ref%
{7-7})$ holds. Let $\rho _{n}\in C^{1}(\mathbb{R}^{N})$ be such that $\rho
_{n}\equiv 1$ in $B_{r_{n}}(\xi _{n}),$ $\rho _{n}\equiv 0$ in $\mathbb{R}%
^{N}\backslash B_{2r_{n}}(\xi _{n}),$ $0\leq \rho _{n}\leq 1$ and $|\nabla
\rho _{n}|\leq 2/r_{n}.$ We set%
\begin{equation*}
h_{n}:=\rho _{n}u_{n},\text{ }w_{n}:=(1-\rho _{n})u_{n}.
\end{equation*}%
It is not difficult to verify that%
\begin{equation}
\liminf_{n\rightarrow \infty }\Phi _{a}^{\infty }(h_{n})\geq \alpha \text{
and }\liminf_{n\rightarrow \infty }\Phi _{a}^{\infty }(w_{n})\geq \alpha
_{a}^{\infty ,-}-\alpha .  \label{7-10}
\end{equation}%
Moreover, let us denote $\Omega _{n}:=B_{2r_{n}}(\xi _{n})\backslash
B_{r_{n}}(\xi _{n}).$ Then there holds%
\begin{equation*}
y_{n}(\Omega _{n})\rightarrow 0\text{ as }n\rightarrow \infty ,
\end{equation*}%
namely%
\begin{equation*}
\int_{\Omega _{n}}(|\nabla u_{n}|^{2}+u_{n}^{2})dx\rightarrow 0\text{ as }%
n\rightarrow \infty
\end{equation*}%
and%
\begin{equation*}
\int_{\Omega _{n}}f_{\infty }|u_{n}|^{p}dx\rightarrow 0\text{ as }%
n\rightarrow \infty .
\end{equation*}%
A direct calculation shows that%
\begin{equation*}
\int_{\Omega _{n}}(|\nabla h_{n}|^{2}+h_{n}^{2})dx\rightarrow 0\text{ as }%
n\rightarrow \infty
\end{equation*}%
and
\begin{equation*}
\int_{\Omega _{n}}(|\nabla w_{n}|^{2}+w_{n}^{2})dx\rightarrow 0\text{ as }%
n\rightarrow \infty .
\end{equation*}%
Thus, we obtain that%
\begin{equation}
\left\Vert u_{n}\right\Vert _{H^{1}}^{2}=\left\Vert h_{n}\right\Vert
_{H^{1}}^{2}+\left\Vert w_{n}\right\Vert _{H^{1}}^{2}+o_{n}(1)  \label{7-9}
\end{equation}%
and%
\begin{equation}
\int_{\mathbb{R}^{N}}|u_{n}|^{p}dx=\int_{\mathbb{R}^{N}}|h_{n}|^{p}dx+\int_{%
\mathbb{R}^{N}}|w_{n}|^{p}dx+o_{n}(1).  \label{7-13}
\end{equation}%
Moreover, using $(\ref{7-9})$ leads to%
\begin{eqnarray}
\left( \int_{\mathbb{R}^{N}}|\nabla u_{n}|^{2}dx\right) ^{2} &=&\left( \int_{%
\mathbb{R}^{N}}|\nabla h_{n}|^{2}dx+\int_{\mathbb{R}^{N}}|\nabla
w_{n}|^{2}dx+o_{n}(1)\right) ^{2}  \notag \\
&\geq &\left( \int_{\mathbb{R}^{N}}|\nabla h_{n}|^{2}dx\right) ^{2}+\left(
\int_{\mathbb{R}^{N}}|\nabla w_{n}|^{2}dx\right) ^{2}+o_{n}(1).  \label{7-23}
\end{eqnarray}%
It follows from $(\ref{7-9})-\ref{7-13})$ that%
\begin{equation*}
\Phi _{a}^{\infty }(u_{n})=\Phi _{a}^{\infty }(h_{n})+\Phi _{a}^{\infty
}(w_{n})+o_{n}(1).
\end{equation*}%
Thus, there holds%
\begin{equation*}
\alpha _{a}^{\infty ,-}=\lim_{n\rightarrow \infty }\Phi _{a}^{\infty
}(u_{n})=\lim_{n\rightarrow \infty }\Phi _{a}^{\infty
}(h_{n})+\lim_{n\rightarrow \infty }\Phi _{a}^{\infty }(w_{n}).
\end{equation*}%
Using the above equality, together with $(\ref{7-10})$ yields%
\begin{equation}
\lim_{n\rightarrow \infty }\Phi _{a}^{\infty }(h_{n})=\alpha \text{ and }%
\lim_{n\rightarrow \infty }\Phi _{a}^{\infty }(w_{n})=\alpha _{a}^{\infty
,-}-\alpha .  \label{7-24}
\end{equation}%
Furthermore, by $(\ref{7-9})-(\ref{7-23})$ we get%
\begin{eqnarray}
0 &=&\left\langle (J_{a}^{\infty })^{\prime }(u_{n}),u_{n}\right\rangle
\notag \\
&\geq &\left\langle J_{a}^{\infty })^{\prime }(h_{n}),h_{n}\right\rangle
+\left\langle J_{a}^{\infty })^{\prime }(w_{n}),w_{n}\right\rangle +o_{n}(1).
\label{7-18}
\end{eqnarray}%
We have to distinguish two cases as follows:

Case $(i):$ Up to a subsequence, $\left\langle J_{a}^{\infty })^{\prime
}(h_{n}),h_{n}\right\rangle \leq 0$ or $\left\langle J_{a}^{\infty
})^{\prime }(w_{n}),w_{n}\right\rangle \leq 0.$

Without loss of generality, we only consider the case of $\left\langle
J_{a}^{\infty })^{\prime }(h_{n}),h_{n}\right\rangle \leq 0$, since the case
of $\left\langle \left( J_{a}^{\infty }\right) ^{\prime }\left( w_{n}\right)
,w_{n}\right\rangle \leq 0$ is similar. Note that for all $n\geq 1,$ there
holds%
\begin{equation}
\max \left\{ \left\Vert h_{n}\right\Vert _{H^{1}},\left\Vert
w_{n}\right\Vert _{H^{1}}\right\} <D_{1}<\left( \frac{2S_{p}^{p}}{f_{\infty
}(4-p)}\right) ^{1/\left( p-2\right) }  \label{7-17}
\end{equation}%
Then for each $a\leq \Lambda ,$ we have%
\begin{eqnarray*}
\frac{p}{4-p}\left( \frac{2a(4-p)}{p-2}\right) ^{\left( p-2\right)
/2}\left\Vert h_{n}\right\Vert _{H^{1}}^{p} &=&\frac{p}{4-p}\left( \frac{%
2a(4-p)}{p-2}\right) ^{\left( p-2\right) /2}\left\Vert h_{n}\right\Vert
_{H^{1}}^{p-2}\left\Vert h_{n}\right\Vert _{H^{1}}^{2} \\
&<&\frac{2pS_{p}^{p}}{f_{\infty }(4-p)^{2}}\left( \frac{2a(4-p)}{p-2}\right)
^{\left( p-2\right) /2}\left\Vert h_{n}\right\Vert _{H^{1}}^{2} \\
&\leq &\left\Vert h_{n}\right\Vert _{H^{1}}^{2}+a\left( \int_{\mathbb{R}%
^{N}}\left\vert \nabla h_{n}\right\vert ^{2}dx\right) ^{2} \\
&\leq &\int_{\mathbb{R}^{N}}f_{\infty }|h_{n}|^{p}dx\text{ for }N=1,2,3.
\end{eqnarray*}%
By Lemmas $\ref{g6}-\ref{g15},$ for any $n\geq 1$, there exists%
\begin{equation*}
T_{f_{\infty }}(h_{n})<t_{a,n}^{-}<\sqrt{D(p)}\left( \frac{2}{4-p}\right)
^{1/(p-2)}T_{f_{\infty }}(h_{n})
\end{equation*}%
such that $t_{a,n}^{-}h_{n}\in \mathbf{M}_{a}^{\infty ,-}$, where
\begin{equation*}
T_{f_{\infty }}(h_{n})=\left( \frac{\left\Vert h_{n}\right\Vert _{H^{1}}^{2}%
}{\int_{\mathbb{R}^{N}}f_{\infty }|h_{n}|^{p}dx}\right) ^{1/(p-2)}>0.
\end{equation*}%
We now prove that $t_{a,n}^{-}\leq 1.$ Suppose the contrary. Then $%
t_{a,n}^{-}>1.$ Since $t_{a,n}^{-}h_{n}\in \mathbf{M}_{a}^{\infty ,-},$ we
have%
\begin{equation*}
a\left( \int_{\mathbb{R}^{N}}|\nabla h_{n}|^{2}dx\right)
^{2}=-(t_{a,n}^{-})^{-2}\left\Vert h_{n}\right\Vert
_{H^{1}}^{2}+(t_{a,n}^{-})^{p-4}\int_{\mathbb{R}^{N}}f_{\infty
}|h_{n}|^{p}dx.
\end{equation*}%
Using the above equality gives%
\begin{eqnarray}
0 &\geq &\left\langle (J_{a}^{\infty })^{\prime }(h_{n}),h_{n}\right\rangle
=\left\Vert h_{n}\right\Vert _{H^{1}}^{2}+a\left( \int_{\mathbb{R}%
^{N}}|\nabla h_{n}|^{2}dx\right) ^{2}-\int_{\mathbb{R}^{N}}f_{\infty
}|h_{n}|^{p}dx  \notag \\
&=&\left\Vert h_{n}\right\Vert _{H^{1}}^{2}-(t_{a,n}^{-})^{-2}\left\Vert
h_{n}\right\Vert _{H^{1}}^{2}+(t_{a,n}^{-})^{p-4}\int_{\mathbb{R}%
^{N}}f_{\infty }|h_{n}|^{p}dx  \notag \\
&&-\int_{\mathbb{R}^{N}}f_{\infty }|h_{n}|^{p}dx  \notag \\
&=&[1-(t_{a,n}^{-})^{-2}]\left\Vert h_{n}\right\Vert
_{H^{1}}^{2}+[(t_{a,n}^{-})^{p-4}-1]\int_{\mathbb{R}^{N}}f_{\infty
}|h_{n}|^{p}dx.  \label{7-16}
\end{eqnarray}%
Note that
\begin{equation}
\int_{\mathbb{R}^{N}}f_{\infty }|h_{n}|^{p}dx<\frac{2}{4-p}%
(t_{a,n}^{-})^{2-p}\left\Vert h_{n}\right\Vert _{H^{1}}^{2}  \label{7-15}
\end{equation}%
by $(\ref{2-2}).$ It follows from $(\ref{7-16})-(\ref{7-15})$ that%
\begin{eqnarray}
0 &\geq &[1-(t_{a,n}^{-})^{-2}]\left\Vert h_{n}\right\Vert _{H^{1}}^{2}+%
\frac{2}{4-p}[(t_{a,n}^{-})^{-2}-(t_{a,n}^{-})^{2-p}]\left\Vert
h_{n}\right\Vert _{H^{1}}^{2}  \notag \\
&=&\left[ 1+\frac{p-2}{4-p}(t_{a,n}^{-})^{-2}-\frac{2}{4-p}%
(t_{a,n}^{-})^{2-p}\right] \left\Vert h_{n}\right\Vert _{H^{1}}^{2}  \notag
\\
&=&(t_{a,n}^{-})^{-2}\left[ (t_{a,n}^{-})^{2}-\frac{2}{4-p}%
(t_{a,n}^{-})^{4-p}+\frac{p-2}{4-p}\right] \left\Vert h_{n}\right\Vert
_{H^{1}}^{2},  \label{7-21}
\end{eqnarray}%
which implies that%
\begin{equation*}
(t_{a,n}^{-})^{2}-\frac{2}{4-p}(t_{a,n}^{-})^{4-p}+\frac{p-2}{4-p}\leq 0.
\end{equation*}%
However, we observe that for $2<p<\min \{4,2^{\ast }\},$%
\begin{equation*}
t^{2}-\frac{2}{4-p}t^{4-p}+\frac{p-2}{4-p}>0\text{ for }t>1.
\end{equation*}%
This is a contradiction. Thus, $t_{a,n}^{-}\leq 1.$

Next, let us consider the functional $\Phi _{a}^{\infty }(th_{n})$ defined by%
\begin{equation*}
\Phi _{a}^{\infty }(th_{n})=\frac{t^{2}}{4}\left\Vert h_{n}\right\Vert
_{H^{1}}^{2}-\frac{(4-p)t^{p}}{4p}\int_{\mathbb{R}^{N}}f_{\infty
}|h_{n}|^{p}dx\text{ for }t>0.
\end{equation*}%
A direct calculation shows that there exists a constant%
\begin{equation*}
t_{a}^{\infty }(h_{n})=\left( \frac{2\left\Vert h_{n}\right\Vert _{H^{1}}^{2}%
}{(4-p)f_{\infty }\int_{\mathbb{R}^{N}}|h_{n}|^{p}dx}\right) ^{1/(p-2)}>0
\end{equation*}%
such that $\Phi _{a}^{\infty }(th_{n})$ is increasing on $(0,t_{a}^{\infty
}(h_{n}))$ and is decreasing on $(t_{a}^{\infty }(h_{n}),\infty ).$
Furthermore, by the Sobolev inequality and $(\ref{7-17})$ we have%
\begin{equation*}
t_{a}^{\infty }(h_{n})\geq \left( \frac{2}{(4-p)f_{\infty
}S_{p}^{-p}\left\Vert h_{n}\right\Vert _{H^{1}}^{p-2}}\right) ^{1/(p-2)}>1.
\end{equation*}%
This indicates that $\Phi _{a}^{\infty }(t_{a,n}^{-}h_{n})\leq \Phi
_{a}^{\infty }(h_{n}).$ Hence, for all $n\geq 1,$ there holds%
\begin{equation*}
\alpha _{a}^{\infty ,-}\leq J_{a}^{\infty }(t_{a,n}^{-}h_{n})=\Phi
_{a}^{\infty }(t_{a,n}^{-}h_{n})\leq \Phi _{a}^{\infty }(h_{n})\rightarrow
\alpha <\alpha _{a}^{\infty ,-},
\end{equation*}%
which is a contradiction.\newline
Case $(ii):$ Up to a subsequence, $\left\langle (J_{a}^{\infty })^{\prime
}(h_{n}),h_{n}\right\rangle >0$ and $\left\langle (J_{a}^{\infty })^{\prime
}(w_{n}),w_{n}\right\rangle >0.$

By $(\ref{7-18})$ one has $\left\langle (J_{a}^{\infty })^{\prime
}(h_{n}),h_{n}\right\rangle =o_{n}(1)$ and $\left\langle (J_{a}^{\infty
})^{\prime }(w_{n}),w_{n}\right\rangle =o_{n}(1).$ If $t_{a,n}^{-}\leq
1+o_{n}(1),$ then we can repeat the argument of Case $(i)$ and arrive at the
contradiction. Suppose that
\begin{equation*}
\lim_{n\rightarrow \infty }t_{a,n}^{-}=t_{a,\infty }^{-}>1.
\end{equation*}%
Similar to the argument of $(\ref{7-16})$, we have%
\begin{eqnarray*}
o_{n}(1) &=&\left\langle (J_{a}^{\infty })^{\prime
}(h_{n}),h_{n}\right\rangle \\
&=&[1-(t_{a,n}^{-})^{-2}]\left\Vert h_{n}\right\Vert
_{H^{1}}^{2}+[(t_{a,n}^{-})^{p-4}-1]\int_{\mathbb{R}^{N}}f_{\infty
}|h_{n}|^{p}dx.
\end{eqnarray*}%
Similar to the argument of $(\ref{7-21}),$ one get%
\begin{equation}
o_{n}(1)\geq (t_{a,n}^{-})^{-2}\left[ (t_{a,n}^{-})^{2}-\frac{2}{4-p}%
(t_{a,n}^{-})^{4-p}+\frac{p-2}{4-p}\right] \left\Vert h_{n}\right\Vert
_{H^{1}}^{2},  \notag
\end{equation}%
which shows that
\begin{equation*}
\left\Vert h_{n}\right\Vert _{H^{1}}^{2}\rightarrow 0\text{ as }n\rightarrow
\infty ,
\end{equation*}%
where we have used the fact of $(t_{a,n}^{-})^{2}-\frac{2}{4-p}%
(t_{a,n}^{-})^{4-p}+\frac{p-2}{4-p}>0.$ Then%
\begin{equation*}
\int_{\mathbb{R}^{N}}|h_{n}|^{p}dx\rightarrow 0\text{ as }n\rightarrow
\infty .
\end{equation*}%
Hence, $\Phi _{a}^{\infty }(h_{n})\rightarrow 0$ as $n\rightarrow \infty ,$
which contradicts $(\ref{7-24})$. Therefore, the dichotomy cannot occur.
\end{proof}

\section*{Acknowledgments}

J. Sun is supported by the National Natural Science Foundation of China
(Grant No. 11671236). T.F. Wu is supported in part by the Ministry of
Science and Technology, Taiwan (Grant 108-2115-M-390-007-MY2), the
Mathematics Research Promotion Center, Taiwan and the National Center for
Theoretical Sciences, Taiwan.


\begin{thebibliography}{99}
\bibitem{Az1} A. Azzollini, The elliptic Kirchhoff equation in $\mathbb{R}%
^{N}$ perturbed by a local nonlinearity, Differential and Integral equations
25 (2012) 543--554.

\bibitem{Az2} A. Azzollini, A note on the elliptic Kirchhoff equation in $%
\mathbb{R}^{N}$ perturbed by a local nonlinearity, Comm. Contemporary Math.
17 (2015) 1450039.

\bibitem{BL} H. Berestycki and P.L. Lions, Nonlinear scalar field equations,
I. Existence of a ground state, Arch. Ration. Mech. Anal. 82 (1983) 313--345.

\bibitem{BZ} K.J. Brown, Y. Zhang, The Nehari manifold for a semilinear
elliptic equation with a sign-changing weight function, J. Differential
Equations 193 (2003) 481--499.

\bibitem{CKW} C.Y. Chen, Y.C. Kuo, T.F. Wu, The Nehari manifold for a
Kirchhoff type problem involving sign-changing weight functions, J.
Differential Equations 250 (2011) 1876--1908.

\bibitem{DM1} T. D'Aprile, D. Mugnai, Non-existence results for the coupled
Klein-Gordon-Maxwell equations, Adv. Nonlinear Stud. 4 (2004) 307--322.

\bibitem{DS} P. D'Ancona, Y. Shibata, On global solvability of non-linear
viscoelastic equations in the analytic category, Math. Methods Appl. Sci. 17
(1994) 477-489.

\bibitem{DS1} P. D'Ancona, S. Spagnolo, Global solvability for the
degenerate Kirchhoff equation with real analytic data, Invent. Math. 108
(1992) 247--262.

\bibitem{DP} P. Dr\'{a}bek, S.I. Pohozaev, Positive solutions for the $p$%
--Laplacian: application of the fibering method, Proc. Roy. Soc. Edinburgh
Sect. A 127 (1997) 703--726.

\bibitem{DPS} Y. Deng, S. Peng, W. Shuai, Existence and asymptotic behavior
of nodal solutions for the Kirchhoff-type problems in $\mathbb{R}^{3},$ J.
Funct. Anal. 269 (2015) 3500--3527.

\bibitem{E} I. Ekeland, On the variational principle, J. Math. Anal. Appl.
17 (1974) 324--353.

\bibitem{G} Z. Guo, Ground states for Kirchhoff equations without compact
condition, J. Differential Equations, 259 (2015) 2884--2902.

\bibitem{HL1} Q. Han, F. Lin, Elliptic Partial Differential Equations, Lect.
Notes, Am. Math. Soc., Providence, 2000.

\bibitem{HL} Y. He, G. Li, Standing waves for a class of Kirchhoff type
problems in $\mathbb{R}^{3}$ involving critical Sobolev exponents, Calc.
Var. Partial Differential Equations 54 (2015) 3067--3106.

\bibitem{HZ} X. He, W. Zou, Existence and concentration behavior of positive
solutions for a Kirchhoff equation in $\mathbb{R}^{3},$ J. Differential
Equations 252 (2012) 1813--1834.

\bibitem{J} L. Jeanjean, On the existence of bounded Palais--Smale sequences
and application to a Landsman--Lazer-type problem set on $\mathbb{R}^{N}$,
Proc. Edinburgh. Math. Soc. 129 (1999) 787--809.

\bibitem{K} M.K. Kwong, Uniqueness of positive solution of $\Delta
u-u+u^{p}=0$ in $\mathbb{R}^{N}$, Arch. Ration. Mech. Anal. 105 (1989)
243--266.

\bibitem{LLS} Y. Li, F. Li, J. Shi, Existence of a positive solution to
Kirchhoff type problems without compactness conditions, J. Differential
Equations 253 (2012) 2285--2294.

\bibitem{LY} G. Li, H. Ye, Existence of positive ground state solutions for
the nonlinear Kirchhoff type equations in $\mathbb{R}^{N},$ J. Differential
Equations 257 (2014) 566--600.

\bibitem{LLS1} Z. Liang, F. Li, J. Shi, Positive solutions to Kirchhoff type
equations with nonlinearity having prescribed asymptotic behavior, Ann.
Inst. H. Poincar\'{e} Anal. Non Lin\'{e}aire 31 (2014) 155-167.

\bibitem{L} J.L. Lions, On some questions in boundary value problems of
mathematical physics, in: Contemporary Developments in Continuum Mechanics
and Partial Differential Equations, in: North-Holl. Math. Stud., vol. 30,
NorthHolland, Amsterdam, New York, 1978, pp. 284--346.

\bibitem{Li1} P.L. Lions, The concentration-compactness principle in the
calculus of variations. The locally compact case I, Ann. Inst. H. Poincar%
\'{e} Anal. Non Lin\'{e}aire 1 (1984) 109--145.

\bibitem{Li2} P.L. Lions, The concentration-compactness principle in the
calculus of variations. The locally compact case II, Ann. Inst. H. Poincar%
\'{e} Anal. Non Lin\'{e}aire 1 (1984) 223--283.

\bibitem{NT} W.M. Ni, I. Takagi, On the shape of least energy solution to a
Neumann problem, Comm. Pure Appl. Math. 44 (1991) 819--851.

\bibitem{N1} D. Naimen, The critical problem of Kirchhoff type elliptic
equations in dimension four, J. Differential Equations 257 (2014) 1168--1193.

\bibitem{N} K. Nishihara, On a global solution of some quasilinear
hyperbolic equation, Tokyo J. Math. 7 (1984) 437--459.

\bibitem{P} S.I. Pohozaev, A certain class of quasilinear hyperbolic
equations. Mat. Sb. (N.S.) 96 (1975) 152--166.

\bibitem{R1} D. Ruiz, The Schr\"{o}dinger--Poisson equation under the effect
of a nonlinear local term, J. Funct. Anal. 237 (2006) 655--674.

\bibitem{SZ} D. Sun, Z. Zhang, Uniqueness, existence and concentration of
positive ground state solutions for Kirchhoff type problems in $\mathbb{R}%
^{3},$ J. Math. Anal. Appl. 461 (2018) 128--149.

\bibitem{SW} J. Sun, T.F. Wu, Ground state solutions for an indefinite
Kirchhoff type problem with steep potential well, J. Differential Equations
256 (2014) 1791--1792.

\bibitem{SW1} J. Sun, T.F. Wu, Existence and multiplicity of solutions for
an indefinite Kirchhoff-type equation in bounded domains, Proc. Roy. Soc.
Edinburgh Sect. A 146 (2016) 435--448.

\bibitem{SWF1} J. Sun, T.F. Wu, Z. Feng, On the non-autonomous Schr\"{o}%
dinger--Poisson problems in $\mathbb{R}^{3}$, Discrete Contin. Dyn. Syst. 38
(2018) 1889--1933.

\bibitem{SWF2} J. Sun, T.F. Wu, Z. Feng, Multiplicity of positive solutions
for a nonlinear Schr\"{o}dinger--Poisson system, J. Differential Equations
260 (2016) 586--627.

\bibitem{TC} X. Tang, S. Chen, Ground state solutions of Nehari--Pohozaev
type for Kirchhoff-type problems with general potentials, Calc. Var. Partial
Differential Equations (2017) 56:110.

\bibitem{T} G. Tarantello, On nonhomogeneous elliptic equations involving
critical Sobolev exponent, Ann. Inst. H. Poincar\'{e} Anal. Non Lin\'{e}aire
9 (1992) 281--304.

\bibitem{Y} H. Ye, Positive high energy solution for Kirchhoff equation in $%
\mathbb{R}^{3}$ with superlinear nonlinearities via Nehari-Pohozaev
manifold, Discrete Contin. Dyn. Syst. 35 (2015) 3857--3877.
\end{thebibliography}
\end{document}